\pdfoutput=1
\documentclass[11pt]{article}

\usepackage[a4paper,margin=1.15in]{geometry}
\usepackage{amsmath,amsthm,amssymb,mathtools}
\usepackage{mathrsfs}
\usepackage{tikz}
\usetikzlibrary{calc,positioning}
\usepackage{enumitem}
\usepackage{microtype}
\usepackage[colorlinks=true,linkcolor=purple,citecolor=magenta,urlcolor=cyan]{hyperref}
\usepackage{aliascnt}
\usepackage[nameinlink,capitalize]{cleveref}

\newtheorem{theorem}{Theorem}[section]

\newaliascnt{proposition}{theorem}
\newtheorem{proposition}[proposition]{Proposition}
\aliascntresetthe{proposition}
\crefname{proposition}{Proposition}{Propositions}

\newaliascnt{lemma}{theorem}
\newtheorem{lemma}[lemma]{Lemma}
\aliascntresetthe{lemma}
\crefname{lemma}{Lemma}{Lemmas}

\newaliascnt{corollary}{theorem}
\newtheorem{corollary}[corollary]{Corollary}
\aliascntresetthe{corollary}
\crefname{corollary}{Corollary}{Corollaries}

\theoremstyle{definition}
\newaliascnt{definition}{theorem}
\newtheorem{definition}[definition]{Definition}
\aliascntresetthe{definition}
\crefname{definition}{Definition}{Definitions}

\newaliascnt{example}{theorem}
\newtheorem{example}[example]{Example}
\aliascntresetthe{example}
\crefname{example}{Example}{Examples}

\newaliascnt{question}{theorem}
\newtheorem{question}[question]{Question}
\aliascntresetthe{question}
\crefname{question}{Question}{Questions}

\theoremstyle{remark}
\newaliascnt{remark}{theorem}
\newtheorem{remark}[remark]{Remark}
\aliascntresetthe{remark}
\crefname{remark}{Remark}{Remarks}

\newcommand{\Raiz}{\Phi}
\newcommand{\RaizSimp}{\Delta}
\newcommand{\RaizReti}{Q}
\newcommand{\Pesos}{\Lambda}
\newcommand{\Weyl}{W}
\newcommand{\QO}{\mathscr{Q}}
\newcommand{\chr}{\operatorname{ch}}
\newcommand{\supp}{\operatorname{supp}}
\newcommand{\Mu}{\mathsf{M}}
\newcommand{\typeA}{\ensuremath{\mathrm{A}}}
\newcommand{\typeB}{\ensuremath{\mathrm{B}}}
\newcommand{\typeC}{\ensuremath{\mathrm{C}}}
\newcommand{\typeD}{\ensuremath{\mathrm{D}}}
\newcommand{\typeE}{\ensuremath{\mathrm{E}}}
\newcommand{\typeF}{\ensuremath{\mathrm{F}}}
\newcommand{\typeG}{\ensuremath{\mathrm{G}}}
\newcommand{\Z}{\mathbb Z}

\newcommand{\C}{\mathbb C}

\font\ssfnt=cmss10
\def\LiE{{\ssfnt L\kern-.25em\raise0.59ex\hbox{\i}\kern-0.03em E}}

\title{Atomic Decompositions of Lie Characters and the Dominant Weight Poset}
\author{Felipe de Mattos Chafik Hindi \and Waldeck Sch\"utzer}
\date{September 7, 2026}

\begin{document}
\maketitle

\begin{abstract}
Let $\mathfrak g$ be a complex semisimple Lie algebra and, for a dominant
integral weight $\lambda$, let $\Theta_\lambda$ be the multiplicity-free sum
of the weights of the irreducible module $V(\lambda)$.  The $\Theta_\mu$ form
a $\Z$-basis of the $\Weyl$-invariants, so there are unique integers
$a(\mu,\lambda)$, the \emph{atomic numbers}, with
\[
    \chr V(\lambda)=\sum_{\mu\in\Pesos^+}a(\mu,\lambda)\Theta_\mu .
\]
We ask when they are nonnegative.  They form the M\"obius transform of the
weight-multiplicity function on the dominant-weight poset, which the crosscut
theorem turns into an alternating sum of at most
$2^{\operatorname{rank}\Raiz}$ multiplicities, indexed by subsets of the
weights covering $\mu$ in $[\mu,\lambda]$.

We prove that $a(\mu,\lambda)\ge0$ whenever every connected component of the
Dynkin diagram is a path, and we identify the coefficient with the dimension
of an explicit weight space: it is cut out of $V(\lambda)$ by alternately
taking kernels of raising operators and cokernels of lowering operators, one
for each of these covers, processed in order along the path.  Type
$\typeD_4$ shows the hypothesis is necessary.  Put
$\beta=\alpha_1+2\alpha_2+2\alpha_3+2\alpha_4$; for every dominant $\mu$ with
$\langle\mu,\alpha_2^\vee\rangle=1$ and $\mu+\beta$ dominant,
$a(\mu,\mu+\beta)$ equals $-2$ if $\langle\mu,\alpha_1^\vee\rangle=0$ and
$-1$ otherwise.  Restriction to the support of $\lambda-\mu$ carries this
family into every irreducible type with a trivalent node.  An irreducible
finite root system therefore has all atomic numbers nonnegative precisely
when its Dynkin diagram is a path, namely in types $\typeA_n$, $\typeB_n$,
$\typeC_n$, $\typeF_4$, and $\typeG_2$.

Deep in the dominant chamber, in an explicit range, $a(\mu,\lambda)$ is the
Kostant partition number of $\lambda-\mu$ for the nonsimple positive roots;
negative atomic numbers are therefore confined to boundary slabs.
\end{abstract}


\section{Introduction}

Computing weight multiplicities in irreducible highest-weight modules is a
classical problem in the representation theory of semisimple Lie algebras.
The Weyl character formula determines the character, and Kostant's formula
extracts individual multiplicities from it, but both formulas hide substantial
cancellation.  A complementary problem is to refine the character into simpler
multiplicity-free pieces.  Such refinements appear in polytope expansions,
girdle expansions, and in the theory of atomic decompositions and atomic
polynomials; see, for example, \cite{Walton2,Schutzer,Shimozono,MR4178925,Patimo}.

Fix a complex semisimple Lie algebra $\mathfrak g$, with root system $\Raiz$,
positive roots $\Raiz^+$, simple roots
$\RaizSimp=\{\alpha_i:i\in I\}$, simple coroots $\alpha_i^\vee$,
fundamental weights $\omega_i$ determined by
$\langle\omega_i,\alpha_j^\vee\rangle=\delta_{ij}$, weight lattice $\Pesos$,
dominant cone $\Pesos^+$, Weyl group $\Weyl$, and
$\rho=\frac12\sum_{\alpha\in\Raiz^+}\alpha$.  For a weight $\nu$, its
$i$-th \emph{Dynkin label} is $\nu_i=\langle\nu,\alpha_i^\vee\rangle$.
Thus $\nu=\sum_i\nu_i\omega_i$.  In numerical examples we abbreviate a
weight by its ordered Dynkin labels.  Nonnegative single-digit labels are
concatenated, as in $(0100)$; if any label is negative, commas separate all
labels, as in $(1,-1)$.  Such weight expressions are parenthesized, including
in subscripts, but their own parentheses may be omitted when they are
arguments of a function, as in $a(0100,0022)$.  For symbolic or multidigit
labels we use explicit sums of fundamental weights to avoid ambiguity.
Tuples of simple-root coefficients retain their commas and are explicitly
identified as such.
Fix Chevalley generators
\[
 e_i\in\mathfrak g_{\alpha_i},\qquad
 f_i\in\mathfrak g_{-\alpha_i},\qquad
 [e_i,f_i]=\alpha_i^\vee
 \quad(i\in I).
\]
For $\lambda\in\Pesos^+$, let $V(\lambda)$ denote the irreducible module of
highest weight $\lambda$.  For every $\nu\in\Pesos$ set
\[
 m(\nu,\lambda)=\dim V(\lambda)_\nu,
\]
with $m(\nu,\lambda)=0$ when $\nu$ is not a weight of $V(\lambda)$.
More generally, write $m_V(\nu)=\dim V_\nu$ for a finite-dimensional
$\mathfrak g$-module $V$.  Two weights are \emph{Weyl-conjugate} if
$\eta=w\nu$ for some $w\in\Weyl$, using the ordinary, unshifted action.
Their multiplicities in every such $V$ agree.  Indeed, for each simple root
$\alpha_i$, the sum of the weight spaces in $\nu+\Z\alpha_i$ is a
finite-dimensional $\mathfrak{sl}_2(\alpha_i)$-module.  Symmetry of its
weight strings gives $m_V(s_i\nu)=m_V(\nu)$, and the simple reflections
generate $\Weyl$.  If
$\Pesos(\lambda)$ denotes the set of weights of $V(\lambda)$, define
\[
    \Theta_\lambda=\sum_{\nu\in\Pesos(\lambda)} e^\nu.
\]
Thus $\Theta_\lambda$ remembers which weights occur in $V(\lambda)$ but forgets
their multiplicities.

The integers $a(\mu,\lambda)$ are defined by the triangular expansion
\begin{equation}\label{eq:intro-atomic-expansion}
    \chr V(\lambda)
      =\sum_{\mu\in\Pesos^+} a(\mu,\lambda)\Theta_\mu.
\end{equation}
We call $a(\mu,\lambda)$ the \emph{atomic number} associated with
$(\mu,\lambda)$.  The problem addressed in this paper is not whether the
integral expansion \eqref{eq:intro-atomic-expansion} exists (it always does!)
but whether its coefficients are nonnegative.

Our main organizing observation is that the two-variable functions
\[
   m(\mu,\lambda)=\dim V(\lambda)_\mu,
   \qquad a(\mu,\lambda),
\]
are elements of the incidence algebra of the dominant-weight poset, and that
\begin{equation}\label{eq:intro-mobius}
    a(\mu,\lambda)
      =\sum_{\eta\in[\mu,\lambda]}
        \Mu(\mu,\eta)m(\eta,\lambda),
\end{equation}
where $\Mu$ is the M\"obius function of the poset.  Thus atomic positivity is a
positivity problem for a M\"obius transform of the multiplicity function.

Although M\"obius inversion on the dominance order appears in earlier work,
notably \cite{MR4178925}, we use the incidence algebra as the organizing
framework.  The objects of study are the full two-variable tables $m$ and $a$,
not their columns separately.  The M\"obius function is then a
representation-independent inversion kernel, computed from the dominant-weight
poset and used uniformly for all highest weights.

The M\"obius formula is the first of three complementary descriptions of the
same two-variable function $a$.  They are best viewed not as separate results,
but as three forms of one inversion principle:
\begin{equation}\label{eq:intro-three-descriptions}
   m=\zeta*a,
   \qquad a=\Mu*m;
   \qquad
   a*\kappa=\delta;
   \qquad
   a(\mu,\mu+\eta)=p_{\rm ns}(\eta)
   \quad\text{in the stable chamber}.
\end{equation}
Here $\kappa$ is the transition kernel obtained by expanding the girdles
$\Theta_\lambda$ in irreducible characters, and $p_{\rm ns}$ is the Kostant
partition function for the nonsimple positive roots.  The first description is
poset-theoretic: for a fixed highest weight, M\"obius inversion on the
dominant-weight poset extracts the atomic numbers from the weight
multiplicities.  The second is a highest-weight recursion: the character--girdle
formula of \cite{Schutzer} evaluates $\kappa$ explicitly, and
$a*\kappa=\delta$ makes the atomic table the incidence-algebra inverse of this
kernel.  Equivalently it gives a recursion in the highest-weight variable that
involves no multiplicities (\cref{thm:atomic-recursion}); the multiplicity
recursion of \cite[Corollary~3.1]{Schutzer} is its order-filter sum.

The third description is the stable form of the second.  Deep in the dominant
chamber the shifted Weyl corrections disappear, so the highest-weight
recursion becomes the translation-invariant finite-difference operator
\[
   \prod_{\alpha\in\Raiz^{\rm ns}}(1-T_\alpha).
\]
Its fundamental solution has generating function
$\prod_{\alpha\in\Raiz^{\rm ns}}(1-e^{-\alpha})^{-1}$, hence coefficients
$p_{\rm ns}$.  Thus the stable partition function is not an unrelated
asymptotic formula: it is the translation-invariant model forced by the same
inverse relation $a*\kappa=\delta$.  \Cref{thm:atomic-stabilization} makes this
exact on an explicit finite stable range, sharpening a limit observed by
Lecouvey and Lenart \cite{MR4178925}.  In particular every negative atomic
number is a near-the-wall phenomenon, which is consistent with, and quantifies,
the $\typeD_4$ analysis of \cref{sec:D4}.

A sustained exhaustive exploration in \LiE\ \cite{LiE}, together with
purpose-built code based on Stembridge's dominant-weight poset
\cite{Stembridge}, first revealed the path--fork dichotomy and the local
$\typeD_4$ patterns below.

Crosscut inversion makes the sign mechanism explicit.  If
$\gamma_i=\mu+\beta_i$ are the atoms above $\mu$ and
$x_S=\mu+\sum_{i\in S}\beta_i$, then the crown formula is an alternating sum
of multiplicities at the dominant joins $\widehat{x_S}$.  Thus two issues
control the sign: the raw Boolean finite difference at the $x_S$, and the
effect of dominant stabilization $x_S\mapsto\widehat{x_S}$.  Support reduction
makes this analysis local in $\supp(\lambda-\mu)$.

Rank two gives the first local models.  The crowns are singletons, chains, or
diamonds.  In types $\typeA_1\times\typeA_1$, $\typeA_2$, $\typeB_2$, and
$\typeC_2$ all atomic numbers are $0$ or $1$.  In type $\typeG_2$ the diamond
coefficient is a bounded partition count by the four nonsimple positive roots
and can exceed $1$; compare the stronger $q$-atomicity result of
Muniz--Plaza--Rojas-And\'ias \cite{Muniz-Plaza-Rojas-G2}.

On a path, both parts of the sign problem are favorable.  Distinct atom
differences have disjoint connected supports and generate a regular path
forest.  Every raw sum reaches its least dominant majorant through reflections
at Dynkin label $-1$, so stabilization preserves multiplicities.  Elementary
$\mathfrak{sl}_2$-string theory then realizes the raw Boolean difference as the
dimension of a successive subquotient obtained by alternating primitive
kernels and lowering quotients along the path.  Hence all types
$\typeA_n,\typeB_n,\typeC_n,\typeF_4$, and $\typeG_2$ are atomically positive,
and every atomic number in a path type has a representation-theoretic model
(\cref{cor:path-primitive-model}).  In type $\typeA$, Plaza and Sagurie
\cite{Plaza-Sagurie} independently obtain a related positivity theorem and an
algorithm for Lascoux's atomic decomposition.

The first failure occurs in type $\typeD_4$, and it occurs in the second
mechanism: the raw Boolean difference remains nonnegative, but when the
trivalent Dynkin label is $1$, dominant stabilization contributes an additional
crown term.  This yields an infinite family.  With $J=\{1,3,4\}$ and
$\{i,j,k\}=J$, put
\[
 \beta_i=\alpha_i+2\alpha_2+2\alpha_j+2\alpha_k.
\]
Whenever $\mu$ and $\lambda=\mu+\beta_i$ are dominant with $\mu_2=1$,
\cref{prop:D4-minimal-family} gives
\[
 a(\mu,\lambda)=
 \begin{cases}
 -2,&\mu_i=0,\\
 -1,&\mu_i>0.
 \end{cases}
\]
Support reduction propagates these examples to every irreducible branched
type.  Together with path positivity, this shows that negativity requires a
support component of type $\typeD$ or $\typeE$, equivalently a trivalent
simple root.  Whether every negative coefficient in a larger $\typeD/\typeE$
diagram is already detected on the radius-one $\typeD_4$ neighborhood remains
open.

The paper is organized as follows.  \Cref{sec:atomic} develops the incidence
algebra of the dominant-weight poset, including the crown, crosscut inversion,
and the Boolean envelope.  \Cref{sec:reductions} gives the second inverse
description $a*\kappa=\delta$ and the structural locality reductions.
Rank-two combinatorics is worked out in \cref{sec:rank2}, and positivity for
path diagrams (\cref{thm:path-positive}) is proved in \cref{sec:path-positive}.  Forks are analyzed in \cref{sec:D4},
beginning with $\typeD_4$ and ending with the $\typeD/\typeE$ obstruction,
propagation, and
locality question.  Finally, \cref{sec:stabilization} identifies the stable
value with $p_{\rm ns}$ and confines negativity to boundary slabs.

\section{The incidence algebra of the dominant-weight poset}\label{sec:atomic}\label{sec:dominant-poset}

We use the usual root order on the weight lattice:
\[
   \mu\le \lambda
   \quad\Longleftrightarrow\quad
   \lambda-\mu\in \Z_{\ge0}\RaizSimp.
\]
Its restriction to $\Pesos^+$ will be called the \emph{dominant-weight order}.
Every interval in $\Pesos^+$ is finite.

For $\lambda\in\Pesos^+$ we use throughout
\[
    m(\nu,\lambda)=\dim V(\lambda)_\nu
    \qquad(\nu\in\Pesos),
\]
with the value $0$ off the weight set; thus
\[
    \chr V(\lambda)
      =\sum_{\nu\in\Pesos}m(\nu,\lambda)e^\nu.
\]
If $\mu\in\Pesos^+$, then the saturated-set property for highest-weight
modules implies that
\[
   \mu\in\Pesos(\lambda)
   \quad\Longleftrightarrow\quad
   \mu\le\lambda.
\]
Accordingly, the coefficient of $e^\mu$ in $\Theta_\eta$ is $1$ exactly when
$\mu\le\eta$.

\begin{proposition}\label{prop:integral-expansion}
The girdles $\{\Theta_\mu:\mu\in\Pesos^+\}$ form a $\Z$-basis of the
$\Weyl$-invariants $\Z[\Pesos]^\Weyl$.  In particular, for each
$\lambda\in\Pesos^+$ there are unique integers $a(\mu,\lambda)$,
$\mu\in\Pesos^+$, such that
\[
    \chr V(\lambda)
      =\sum_{\mu\in\Pesos^+}a(\mu,\lambda)\Theta_\mu.
\]
Moreover, $a(\mu,\lambda)=0$ unless $\mu\le\lambda$; as $\Pesos(\lambda)$ is
finite, the sum is therefore finite.
\end{proposition}

The orbit sums $\mathcal O_\nu=\sum_{\gamma\in\Weyl\nu}e^\gamma$,
$\nu\in\Pesos^+$, form a $\Z$-basis of $\Z[\Pesos]^\Weyl$, and
$\Theta_\eta=\sum_{\nu\in\Pesos^+,\ \nu\le\eta}\mathcal O_\nu$; each such index set is
finite, so this transition is unitriangular with finite rows and invertible
over $\Z$, which gives the first assertion.  That $a(\mu,\lambda)$ vanishes
for $\mu\nleq\lambda$ follows by comparing coefficients at a weight maximal
among those with $a(\mu,\lambda)\ne0$.  We omit the routine verifications.
The basis property is also a consequence of the character formula of
\cite[Theorem~2.1]{Schutzer}, which expresses $\Theta_\lambda$ as
$\chr V(\lambda)$ plus an integral combination of characters $\chr V(\eta)$
with $\eta<\lambda$.

\begin{definition}
The integers in \cref{prop:integral-expansion} are the \emph{atomic numbers}.
We say that $\chr V(\lambda)$ admits an \emph{atomic decomposition} if
\[
     a(\mu,\lambda)\ge0
     \qquad\text{for all }\mu\in\Pesos^+.
\]
A root system (or its semisimple Lie algebra) is called \emph{atomic}, or
\emph{universally atomically positive}, if every irreducible character admits
an atomic decomposition.
\end{definition}

\begin{remark}\label{rem:schutzer-thm22}
The character formula \cite[Theorem~2.1]{Schutzer} used above is independent
of \cite[Theorem~2.2]{Schutzer}, which asserts that the coefficients
$a(\mu,\lambda)$ are always nonnegative.  That assertion is not correct: the
first counterexamples occur in type $\typeD_4$, as observed in
\cite[Example~2.6]{MR4178925} and as \cref{prop:D4-minimal-family} makes
explicit.  Determining exactly when nonnegativity does hold is the subject of
the present paper.
\end{remark}

Comparing the coefficient of $e^\mu$ on both sides of the atomic expansion
gives the basic relation
\begin{equation}\label{eq:m-from-a}
    m(\mu,\lambda)
       =\sum_{\substack{\eta\in\Pesos^+\\ \mu\le\eta\le\lambda}}
          a(\eta,\lambda).
\end{equation}

\subsection{The atomic table in the incidence algebra}
\label{subsec:incidence-framework}

The use of M\"obius inversion in connection with atomic decompositions is not
new in itself.  Lecouvey and Lenart observe that, for a fixed highest weight,
the triangular relation defining atomic polynomials can be inverted on the
dominance order by M\"obius inversion \cite{MR4178925}.  Our point of view is
different in two related respects.  First, we keep both weight variables and
regard the entire multiplicity table and the entire atomic table simultaneously
as elements of one incidence algebra.  Second, we use structural information
about the M\"obius function of the dominant-weight poset to turn the formal
inversion into a sparse computational procedure.

Put $P=\Pesos^+$.  Its incidence algebra $I(P;\Z)$ \cite{Stanley} consists of the
integer-valued functions on comparable pairs, with convolution
\[
   (f*g)(\mu,\lambda)
      =\sum_{\mu\le\eta\le\lambda}f(\mu,\eta)g(\eta,\lambda).
\]
Let $\zeta(\mu,\lambda)=1$ for $\mu\le\lambda$, and let
$\Mu=\zeta^{-1}$ be the M\"obius function.  Extend $m$ and $a$ by zero to
non-comparable pairs.  Then \eqref{eq:m-from-a} is exactly
\begin{equation}\label{eq:global-incidence-identity}
       m=\zeta*a,
       \qquad\text{hence}\qquad
       a=\Mu*m.
\end{equation}
Equivalently, for every dominant pair $\mu\le\lambda$,
\begin{equation}\label{eq:mobius-atomic}
   a(\mu,\lambda)
      =\sum_{\eta\in[\mu,\lambda]}
          \Mu(\mu,\eta)m(\eta,\lambda).
\end{equation}
Thus the character of each individual highest-weight module gives one column
of a global identity in $I(P;\Z)$; the inversion kernel $\Mu$ depends only on
the dominant-weight poset and not on the representation.

\begin{definition}\label{def:crown}
For a dominant weight $\mu$, define the \emph{M\"obius crown above $\mu$} by
\[
   \QO(\mu)
      =\{\eta\in\Pesos^+:\eta\ge\mu,\ \Mu(\mu,\eta)\ne0\}.
\]
For $\mu\le\lambda$, its truncation at $\lambda$ is
\[
   \QO(\mu,\lambda)=\QO(\mu)\cap[\mu,\lambda].
\]
\end{definition}

By the definition of the crown, \eqref{eq:mobius-atomic} is equivalently the
crown-supported formula
\begin{equation}\label{eq:crown-formula}
   a(\mu,\lambda)
      =\sum_{\eta\in\QO(\mu,\lambda)}
          \Mu(\mu,\eta)m(\eta,\lambda).
\end{equation}

The terminology ``crown'' emphasizes its computational role: it discards every
element of the interval that cannot
contribute to the atomic number.  Thus, for fixed $\mu\le\lambda$, the
arithmetic part of the inversion uses $|\QO(\mu,\lambda)|$ multiplicities rather
than $|[\mu,\lambda]|$ multiplicities.  Once the support of the row
$\Mu(\mu,-)$ and its M\"obius values are known, the representation-theoretic
task of obtaining those multiplicities is separated from the poset-theoretic
task of computing the sparse inversion kernel.

As will follow from \cref{prop:crown-envelope}, the crown is contained in the
join-image of a Boolean lattice on at most $\operatorname{rank}\Raiz$ atoms;
hence $|\QO(\mu,\lambda)|\le2^{\operatorname{rank}\Raiz}$, although the
crown itself need not be Boolean because distinct atom subsets can have the
same join.

The sparsity is especially strong for the dominant-weight poset.  Stembridge
proved that, for an irreducible root system, its M\"obius function takes only
the values
\[
       0,\ \pm1,\ \pm2,
\]
and his analysis of the covering relation and of the M\"obius function gives
considerably more information about when the nonzero values may occur
\cite[Sections~2 and~4]{Stembridge}.  Thus \eqref{eq:mobius-atomic} is a short
signed linear combination of multiplicities with very small integral
coefficients.

\subsection{Lattice structure, crowns, and crosscut inversion}
The weight lattice decomposes into cosets modulo the root lattice $\RaizReti$.
Fix one such coset and a base point $\kappa$ in it.  Writing
\[
 \nu=\kappa+\sum_i x_i\alpha_i,
 \qquad
 \eta=\kappa+\sum_i y_i\alpha_i,
\]
identifies the root order with the product order on $\Z^{\operatorname{rank}\Raiz}$.
In this full weight-lattice order we therefore have
\begin{equation}\label{eq:full-root-join}
 \nu\vee\eta
   =\kappa+\sum_i\max\{x_i,y_i\}\alpha_i.
\end{equation}
We use the same symbol $\vee$ for the join in a dominant component.  Thus, in
formulas containing both kinds of joins, a join of dominant weights is taken in
$\Pesos^+$, whereas a join of arbitrary weights is the coordinatewise join
\eqref{eq:full-root-join}.  Intersecting a root-lattice coset
with $\Pesos^+$ gives a connected component of the dominant-weight poset.
Stembridge proved that each such component is a lattice and that $\Pesos^+$
is a sub-meet-semilattice of $\Pesos$ \cite{Stembridge}.  Each dominant
component has a minimum element, either $0$ or a minuscule weight.

\begin{lemma}\label{lem:dynkin-poset}
Let $\Gamma$ be the automorphism group of the Dynkin diagram.  Every
$g\in\Gamma$ induces an automorphism of the posets $\Pesos$ and $\Pesos^+$.
Consequently,
\[
    \Mu(g\mu,g\eta)=\Mu(\mu,\eta).
\]
\end{lemma}

\begin{proof}
A diagram automorphism permutes the simple roots and hence preserves
$\Z_{\ge0}\RaizSimp$.  Thus
$\mu\le\eta$ if and only if $g\mu\le g\eta$.  M\"obius functions are invariant
under poset isomorphisms.
\end{proof}

For every weight $\nu$, denote by $\widehat\nu$ the least dominant weight
above $\nu$ in its root-lattice coset.  We verify that it exists.  Since
$2\rho=\sum_{\alpha\in\Raiz^+}\alpha\in\RaizReti$ is strictly dominant,
$\eta_0=\nu+N(2\rho)$ is a dominant majorant for $N\gg0$.  Let
\[
 D_0=\{\eta\in\Pesos^+:\nu\le\eta\le\eta_0\}.
\]
This set is finite and nonempty.  Stembridge's sub-meet-semilattice property
implies that the iterated meet
\[
 \widehat\nu:=\bigwedge_{\eta\in D_0}\eta
\]
is dominant; because every $\eta\in D_0$ majorizes $\nu$, the coordinatewise
meet also majorizes $\nu$.  If $\xi$ is any dominant majorant of $\nu$, then
$\xi\wedge\eta_0$ is dominant, belongs to $D_0$, and satisfies
\[
 \widehat\nu\le \xi\wedge\eta_0\le\xi.
\]
Thus $\widehat\nu$ is the unique least dominant majorant \cite{Stembridge}.
We call the operation $\nu\mapsto\widehat\nu$ \emph{dominant stabilization}.
It is defined by the root order; it need not send $\nu$ to the dominant
representative of its Weyl orbit.  Thus an equality
$m_V(\widehat\nu)=m_V(\nu)$ requires justification; in the path argument
we prove it by establishing Weyl conjugacy.

\begin{lemma}\label{lem:dominant-covering-join}
If $\nu$ and $\eta$ lie in the same root-lattice coset, then
\[
    \widehat\nu\vee\widehat\eta
       =\widehat{\nu\vee\eta}.
\]
Here the join on the left is taken in the dominant component, while the join
inside the hat on the right is the full weight-lattice join
\eqref{eq:full-root-join}.
\end{lemma}

\begin{proof}
Since $\nu\le\widehat\nu$ and $\eta\le\widehat\eta$,
\[
   \nu\vee\eta\le \widehat\nu\vee\widehat\eta,
\]
so minimality of the dominant covering gives
$\widehat{\nu\vee\eta}\le\widehat\nu\vee\widehat\eta$.
Conversely, $\nu,\eta\le\nu\vee\eta\le\widehat{\nu\vee\eta}$, hence
$\widehat\nu,\widehat\eta\le\widehat{\nu\vee\eta}$ and therefore
$\widehat\nu\vee\widehat\eta\le\widehat{\nu\vee\eta}$.
\end{proof}

Write $\mu\lessdot\gamma$ when $\gamma$ covers $\mu$ in the dominant-weight
poset.  For $\mu\le\lambda$, let
\[
   \mathcal A_\mu(\lambda)
      =\{\gamma\in[\mu,\lambda]:\mu\lessdot\gamma\}
\]
be the set of atoms of the interval $[\mu,\lambda]$.  We use the convention
$\bigvee\varnothing=\mu$.

\begin{proposition}
\label{prop:crosscut-lower}
For every $\mu\le\eta$,
\begin{equation}\label{eq:crosscut-mobius}
   \Mu(\mu,\eta)
      =\sum_{\substack{S\subseteq\mathcal A_\mu(\eta)\\
                       \bigvee S=\eta}}
          (-1)^{|S|}.
\end{equation}
Consequently, if $\Mu(\mu,\eta)\ne0$, then $\eta$ is the join of a subset of
the atoms above $\mu$.
\end{proposition}

\begin{proof}
The interval $[\mu,\eta]$ is a finite lattice.  Formula
\eqref{eq:crosscut-mobius} is the crosscut theorem applied to the crosscut
formed by the atoms of this interval; see, for example,
\cite[Section~3.9]{StanleyEC1}.  If $\Mu(\mu,\eta)\ne0$, then the sum on the
right-hand side of \eqref{eq:crosscut-mobius} is nonzero.  Hence its indexing
set is nonempty, so there exists
$S\subseteq\mathcal A_\mu(\eta)$ with $\bigvee S=\eta$.
\end{proof}

\begin{corollary}
\label{cor:atomic-finite-difference}
For every $\mu\le\lambda$,
\begin{equation}\label{eq:atomic-finite-difference}
   a(\mu,\lambda)
      =\sum_{S\subseteq\mathcal A_\mu(\lambda)}
          (-1)^{|S|}
          m\!\left(\bigvee S,\lambda\right).
\end{equation}
\end{corollary}

\begin{proof}
Insert \eqref{eq:crosscut-mobius} into
\eqref{eq:mobius-atomic}.  Every subset $S\subseteq\mathcal A_\mu(\lambda)$
has a join in $[\mu,\lambda]$, and grouping the terms by this join gives
\eqref{eq:atomic-finite-difference}.
\end{proof}

Thus the crown determines a finite signed difference operator, while the
representation enters only through the multiplicities evaluated at the joins
of its atoms.  This avoids any need to identify intersections of distinct
weight spaces with joins in the dominant-weight lattice.

\begin{lemma}\label{lem:atoms-simple-roots}
Let $\gamma$ cover $\mu$ in the dominant-weight poset.  Then there is a simple
root $\alpha_i$ such that
\begin{equation}\label{eq:atom-dominant-cover}
     \gamma=\widehat{\mu+\alpha_i}.
\end{equation}
In particular, there are at most $\operatorname{rank}\Raiz$ atoms above any
dominant weight.
\end{lemma}

\begin{proof}
Write
$\gamma-\mu=\sum_i c_i\alpha_i$ with $c_i\in\Z_{\ge0}$ and choose $i$ with
$c_i>0$.  Then
\[
      \mu<\mu+\alpha_i\le\gamma.
\]
Since $\gamma$ is dominant, minimality of the dominant covering gives
\[
      \mu<\widehat{\mu+\alpha_i}\le\gamma.
\]
The covering relation $\mu\lessdot\gamma$ therefore forces
\eqref{eq:atom-dominant-cover}.  There are only
$\operatorname{rank}\Raiz$ possible simple roots.
\end{proof}

We use the Boolean configuration
\begin{equation}\label{eq:boolean-envelope}
   B(\mu)
      =\left\{\mu+\sum_{\alpha\in S}\alpha:
                    S\subseteq\RaizSimp\right\}
\end{equation}
and its dominant covering
\[
   \widehat B(\mu)=\{\widehat\nu:\nu\in B(\mu)\}.
\]

\begin{proposition}\label{prop:crown-envelope}
For every dominant weight $\mu$,
\begin{equation}\label{eq:crown-envelope}
       \QO(\mu)\subseteq\widehat B(\mu).
\end{equation}
In particular, $|\QO(\mu)|\le2^{\operatorname{rank}\Raiz}$ and the crown has
at most $\operatorname{rank}\Raiz$ atoms.
\end{proposition}

\begin{proof}
Let $\eta\in\QO(\mu)$.  By \cref{prop:crosscut-lower},
$\eta=\bigvee_{\gamma\in S}\gamma$ for some set $S$ of atoms above $\mu$.  By
\cref{lem:atoms-simple-roots} every $\gamma\in S$ has the form
$\widehat{\mu+\alpha_i}$ for some $i\in I$; choose one such index for each
$\gamma\in S$ and let $T$ be the set of chosen indices.  Distinct atoms
receive distinct indices, since $i$ determines $\widehat{\mu+\alpha_i}$, so
the $\alpha_i$ with $i\in T$ are distinct simple roots and the coordinatewise
join \eqref{eq:full-root-join} of the weights $\mu+\alpha_i$, $i\in T$, is
$\mu+\sum_{i\in T}\alpha_i$.  Repeated use of
\cref{lem:dominant-covering-join} now gives
\[
   \eta
      =\bigvee_{i\in T}\widehat{\mu+\alpha_i}
      =\widehat{\mu+\sum_{i\in T}\alpha_i}
      \in\widehat B(\mu).
\]
The cardinality and atom bounds follow.
\end{proof}

It is often convenient to index the finite difference
\eqref{eq:atomic-finite-difference} by subsets of the simple roots rather than
by subsets of the atoms.  The two indexings agree, with no hypothesis on
$\mu$ and $\lambda$.

\begin{corollary}
\label{cor:atomic-finite-difference-simple}
Let $\mu\le\lambda$ be dominant and, for $S\subseteq I$, put
$\alpha_S=\sum_{i\in S}\alpha_i$, with the convention
$\widehat{\mu+\alpha_\varnothing}=\mu$.  Then
\begin{equation}\label{eq:atomic-finite-difference-simple}
   a(\mu,\lambda)
      =\sum_{S\subseteq I}(-1)^{|S|}
          m\bigl(\widehat{\mu+\alpha_S},\lambda\bigr).
\end{equation}
\end{corollary}

\begin{proof}
Write $c_i=\widehat{\mu+\alpha_i}$ and $K=\supp(\lambda-\mu)$.  As in the
proof of \cref{prop:crown-envelope}, repeated use of
\cref{lem:dominant-covering-join} gives
$\widehat{\mu+\alpha_S}=\bigvee_{i\in S}c_i$ for $S\ne\varnothing$.  If
$i\notin K$, then $\mu+\alpha_i\nleq\lambda$, so
$\widehat{\mu+\alpha_S}\ge\mu+\alpha_i$ fails to lie below $\lambda$ and the
corresponding term vanishes, for every $S\ni i$.  The sum may therefore be
restricted to $S\subseteq K$.  If $i\in K$, then $\lambda$ is a dominant
weight above $\mu+\alpha_i$, so $\mu<c_i\le\lambda$ by minimality of the least
dominant majorant.  Two facts about these weights are used below.  First,
every atom $\gamma$ of $[\mu,\lambda]$ equals $c_i$ for some $i\in K$: by
\cref{lem:atoms-simple-roots} one has $\gamma=\widehat{\mu+\alpha_i}=c_i$ for
some $i$, and then $\mu+\alpha_i\le\gamma\le\lambda$ forces $i\in K$.  Second,
every $c_j$ with $j\in K$ lies above an atom of $[\mu,\lambda]$, because
$c_j>\mu$ and the interval $[\mu,c_j]$ is finite.

Suppose $c_i\le c_j$ for some $i\ne j$ in $K$.  If $j\in S$, then
$\bigvee_{r\in S}c_r=\bigvee_{r\in S\cup\{i\}}c_r$, so
$S\mapsto S\bigtriangleup\{i\}$ is a sign-reversing involution of
$\{S\subseteq K:j\in S\}$ that preserves $\widehat{\mu+\alpha_S}$; hence the
terms with $j\in S$ cancel in pairs and the index $j$ may be deleted from $K$
without changing the sum.  Now choose one index $i$ for each atom of
$[\mu,\lambda]$ and let $K'$ be the set of chosen indices.  By the second
fact above, every $j\in K\setminus K'$ admits some $i\in K'$ with
$c_i\le c_j$, and $i\ne j$ because $j\notin K'$; this remains true after any
deletion, since no element of $K'$ is ever deleted.
Deleting the elements of $K\setminus K'$ one at a time leaves
\[
   \sum_{S\subseteq K'}(-1)^{|S|}
     m\Bigl(\bigvee_{i\in S}c_i,\lambda\Bigr).
\]
By the first fact above, $i\mapsto c_i$ is a bijection from $K'$ onto
$\mathcal A_\mu(\lambda)$, so this is \eqref{eq:atomic-finite-difference}.
\end{proof}

\begin{corollary}
\label{cor:all-simple-root-atoms}
Let $r=\operatorname{rank}\Raiz$.  The truncated crown
$\QO(\mu,\lambda)$ has $r$ atoms if and only if
\begin{equation}\label{eq:all-atoms-criterion}
      \lambda-\mu\ge\sum_{i=1}^{r}\alpha_i
      \qquad\text{and}\qquad
      \mu+\alpha_i\in\Pesos^+
      \quad(1\le i\le r).
\end{equation}
When these conditions hold, the atoms are precisely
$\mu+\alpha_1,\ldots,\mu+\alpha_r$.
\end{corollary}

\begin{proof}
If the conditions in \eqref{eq:all-atoms-criterion} hold, each
$\mu+\alpha_i$ is dominant, lies below $\lambda$, and covers $\mu$; hence all
$r$ occur as atoms.  Conversely, suppose there are $r$ atoms.  By
\cref{lem:atoms-simple-roots}, every atom is one of the $r$ dominant coverings
$\widehat{\mu+\alpha_i}$; hence all $r$ of these coverings must be distinct
atoms.  If $\mu+\alpha_i$ were not dominant for some $i$, then
$\widehat{\mu+\alpha_i}-\mu$ could not be a positive multiple of
$\alpha_i$ alone: indeed, dominance of $\mu+k\alpha_i$ for some $k\ge1$
would imply dominance of $\mu+\alpha_i$, since the off-diagonal Cartan
entries are nonpositive.  Thus
$\widehat{\mu+\alpha_i}-\mu$ involves a second simple root $\alpha_j$.
It follows that
$\mu+\alpha_j\le\widehat{\mu+\alpha_i}$ and hence
$\widehat{\mu+\alpha_j}\le\widehat{\mu+\alpha_i}$.  Since both are
atoms, they must coincide, contradicting distinctness.  Hence every
$\mu+\alpha_i$ is dominant.  Since each of these atoms lies below $\lambda$,
every simple-root coefficient of $\lambda-\mu$ is at least $1$, giving the
first condition.
\end{proof}

\begin{corollary}\label{cor:rank2-crowns}
If the root system has rank two, every crown $\QO(\mu,\lambda)$ is one of the
following three posets (\cref{fig:rank2-crowns}): a singleton, a two-element
chain, or a four-element diamond.  In the diamond case the M\"obius values from the bottom are
$1,-1,-1,1$.
\end{corollary}

\begin{proof}
By \cref{lem:atoms-simple-roots}, the interval has at most two atoms.  With no
atoms the interval is the singleton $\{\mu\}$.  With one atom,
\cref{prop:crosscut-lower} shows that only $\mu$ and that atom can belong to
the crown.  With two
distinct atoms $\gamma_1,\gamma_2$, the only possible joins are
$\mu,\gamma_1,\gamma_2$, and $\gamma_1\vee\gamma_2$, and
\eqref{eq:crosscut-mobius} gives the displayed M\"obius values.
\end{proof}

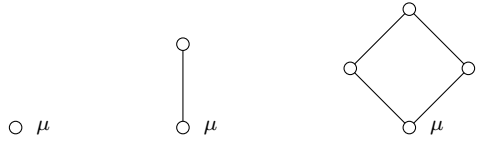
\begin{figure}[ht]
\centering
\pgfmathsetmacro{\dvr}{.15}
\begin{tabular}{c@{\qquad\qquad}c@{\qquad\qquad}c}
\begin{tikzpicture}[scale=0.55]
  \draw[fill=white] (0,0) circle (\dvr cm)
       node[right=4pt]{$\scriptstyle\mu$};
\end{tikzpicture}
&
\begin{tikzpicture}[scale=0.55]
  \draw (0,0)--(0,2);
  \draw[fill=white] (0,0) circle (\dvr cm)
       node[right=4pt]{$\scriptstyle\mu$};
  \draw[fill=white] (0,2) circle (\dvr cm);
\end{tikzpicture}
&
\begin{tikzpicture}[scale=0.55]
  \draw (-1.42,0)--(0,1.42)--(1.42,0)--(0,-1.42)--cycle;
  \draw[fill=white] (-1.42,0) circle (\dvr cm);
  \draw[fill=white] (1.42,0) circle (\dvr cm);
  \draw[fill=white] (0,1.42) circle (\dvr cm);
  \draw[fill=white] (0,-1.42) circle (\dvr cm)
       node[right=4pt]{$\scriptstyle\mu$};
\end{tikzpicture}
\end{tabular}
\caption{The three possible M\"obius crowns in rank two, as proved in
\cref{cor:rank2-crowns}: a singleton, a two-element chain, and a diamond.}
\label{fig:rank2-crowns}
\end{figure}

\section{Highest-weight inversion and structural locality}\label{sec:reductions}

\subsection{A recursion for the atomic numbers}
\label{subsec:atomic-recursion}

The M\"obius formula \eqref{eq:mobius-atomic} computes $a(\mu,\lambda)$ from
the weight multiplicities of the single module $V(\lambda)$.  There is a
second, complementary way to compute the atomic numbers: a recursion in
$\lambda$ that involves no multiplicities at all, only atomic numbers at
strictly smaller dominant weights.  It follows from the character--girdle
formula of \cite{Schutzer}, in the same way that the multiplicity recursion
of that paper does, and it refines that recursion.

We first fix notation.  Let
\[
   \Raiz^{\rm ns}=\Raiz^+\setminus\RaizSimp
\]
be the set of \emph{nonsimple} positive roots, and for $S\subseteq\Raiz^{\rm ns}$
write $\Sigma S=\sum_{\alpha\in S}\alpha$.  For $\xi\in\RaizReti$ put
\[
   c_\xi=\#\{S\subseteq\Raiz^{\rm ns}:\Sigma S=\xi,\ |S|\text{ even}\}
        -\#\{S\subseteq\Raiz^{\rm ns}:\Sigma S=\xi,\ |S|\text{ odd}\},
\]
and let
\[
   \mathcal F=\{\xi\in\RaizReti:\xi\ne0,\ c_\xi\ne0\},
   \qquad
   \mathcal F_0=\mathcal F\cup\{0\}.
\]
Only $S=\varnothing$ has $\Sigma S=0$, so $c_0=1$, and in the group algebra
\begin{equation}\label{eq:ns-generating}
   \prod_{\alpha\in\Raiz^{\rm ns}}\bigl(1-e^{-\alpha}\bigr)
      =\sum_{\xi\in\mathcal F_0}c_\xi e^{-\xi}.
\end{equation}
Both $\mathcal F$ and the integers $c_\xi$ depend only on the root system;
they are computed once and for all.  The cardinality of $\mathcal F$, that is,
the number of terms actually occurring in \eqref{eq:ns-generating}, is
tabulated in \cite[Table~1]{Schutzer}; because of cancellation it may be
strictly smaller than the number of distinct sums $\Sigma S$.

Next we recall the signed characters of \cite{Schutzer}.  A weight $\eta$
is \emph{regular} if $\langle\eta,\alpha^\vee\rangle\ne0$ for every
root $\alpha$, and \emph{singular} otherwise.  For $\nu\in\Pesos$
such that $\nu+\rho$ is regular, let $w_\nu\in\Weyl$ be the unique element with
$w_\nu(\nu+\rho)$ dominant, and set
\[
   \nu^\star=w_\nu(\nu+\rho)-\rho\in\Pesos^+,
   \qquad
   \widetilde\chi_\nu=(-1)^{\ell(w_\nu)}\chr V(\nu^\star);
\]
if $\nu+\rho$ is singular, set $\widetilde\chi_\nu=0$.  Thus
$\widetilde\chi_\nu=\chr V(\nu)$ when $\nu$ is dominant.  The operation
$\nu\mapsto\nu^\star$ selects the dominant representative for the
$\rho$-shifted (dot) action $w\mathbin{\cdot}\nu=w(\nu+\rho)-\rho$.
It differs from the least dominant majorant $\widehat\nu$ of
\cref{sec:dominant-poset}.  Since $w\nu-\nu\in\RaizReti$ for every $w\in\Weyl$,
the weight $\nu^\star$ lies in the same $\RaizReti$-coset as $\nu$.

The following is \cite[Theorem~2.1]{Schutzer}.

\begin{theorem}
\label{thm:schutzer-girdle}
For every $\lambda\in\Pesos^+$,
\begin{equation}\label{eq:schutzer-girdle}
   \Theta_\lambda=\sum_{\xi\in\mathcal F_0}c_\xi\,\widetilde\chi_{\lambda-\xi}.
\end{equation}
Moreover, if $\widetilde\chi_{\lambda-\xi}\ne0$ for some $\xi\in\mathcal F$,
then $(\lambda-\xi)^\star<\lambda$.
\end{theorem}

We transfer \eqref{eq:schutzer-girdle} to the atomic numbers by taking
coordinates in the girdle basis of \cref{prop:integral-expansion}.  For
$\mu\in\Pesos^+$ let $A_\mu:\Z[\Pesos]^\Weyl\to\Z$ be the coordinate
functional attached to $\Theta_\mu$, so that
$A_\mu(\chr V(\lambda))=a(\mu,\lambda)$ and
$A_\mu(\Theta_\lambda)=\delta_{\mu\lambda}$.  Extend the atomic numbers to a
signed function of an arbitrary second argument by setting, for
$\mu\in\Pesos^+$ and $\nu\in\Pesos$,
\begin{equation}\label{eq:signed-atomic}
   \widetilde a(\mu,\nu)=A_\mu\bigl(\widetilde\chi_\nu\bigr)
      =\begin{cases}
        (-1)^{\ell(w_\nu)}\,a(\mu,\nu^\star),&\nu+\rho\text{ regular},\\[1mm]
        0,&\nu+\rho\text{ singular},
      \end{cases}
\end{equation}
and define $\widetilde m(\mu,\nu)$ in the same way from
$m(\mu,\cdot)$.  For dominant $\nu$ one has
$\widetilde a(\mu,\nu)=a(\mu,\nu)$ and $\widetilde m(\mu,\nu)=m(\mu,\nu)$.

\begin{theorem}\label{thm:atomic-recursion}
For all $\mu,\lambda\in\Pesos^+$,
\begin{equation}\label{eq:atomic-recursion}
   \sum_{\xi\in\mathcal F_0}c_\xi\,\widetilde a(\mu,\lambda-\xi)
      =\delta_{\mu\lambda};
\end{equation}
equivalently,
\begin{equation}\label{eq:atomic-recursion-solved}
   a(\mu,\lambda)
     =\delta_{\mu\lambda}
       -\sum_{\xi\in\mathcal F}c_\xi\,\widetilde a(\mu,\lambda-\xi).
\end{equation}
\end{theorem}

\begin{proof}
All the terms of \eqref{eq:schutzer-girdle} lie in $\Z[\Pesos]^\Weyl$, and
$\mathcal F_0$ is finite.  Apply the $\Z$-linear functional $A_\mu$ to
\eqref{eq:schutzer-girdle}.  The left-hand side
gives $A_\mu(\Theta_\lambda)=\delta_{\mu\lambda}$ and the right-hand side
gives $\sum_\xi c_\xi\widetilde a(\mu,\lambda-\xi)$ by
\eqref{eq:signed-atomic}.
\end{proof}

\begin{corollary}\label{cor:atomic-recursion-effective}
Every nonzero summand of the sum on the right of
\eqref{eq:atomic-recursion-solved} is of the form $\pm c_\xi\,a(\mu,\eta)$
with $\eta\in\Pesos^+$, $\eta<\lambda$, and $\eta$ in the $\RaizReti$-coset of
$\lambda$.  Consequently
\eqref{eq:atomic-recursion-solved} determines the entire atomic table by
induction along the dominant-weight order, the base case being
\[
   a(\mu,\lambda)=\delta_{\mu\lambda}
   \qquad\text{for }\lambda\text{ minimal in its dominant component.}
\]
\end{corollary}

\begin{proof}
The first assertion is the triangularity statement of
\cref{thm:schutzer-girdle} together with the remark that $\nu^\star$ and $\nu$
lie in the same $\RaizReti$-coset.  If $\lambda$ is minimal in its component
there is no dominant weight strictly below it, so every $\widetilde\chi_{\lambda-\xi}$
with $\xi\in\mathcal F$ vanishes; equivalently, $\chr V(\lambda)=\Theta_\lambda$,
which is the familiar statement that $V(\lambda)$ is trivial or minuscule.
\end{proof}

The next corollary shows that \eqref{eq:atomic-recursion} refines the
multiplicity recursion of \cite[Corollary~3.1]{Schutzer}: summing the atomic
recursion over a principal dominant order filter returns it exactly.

\begin{corollary}
\label{cor:multiplicity-recursion}
For all $\nu,\lambda\in\Pesos^+$,
\[
   \sum_{\xi\in\mathcal F_0}c_\xi\,\widetilde m(\nu,\lambda-\xi)
      =\begin{cases}1,&\nu\le\lambda,\\0,&\text{otherwise.}\end{cases}
\]
\end{corollary}

\begin{proof}
Fix $\nu$ and sum \eqref{eq:atomic-recursion} over all $\mu\in\Pesos^+$ with
$\mu\ge\nu$.  For each $\xi$ only finitely many terms are nonzero, and by
\eqref{eq:m-from-a},
\[
   \sum_{\mu\ge\nu}a(\mu,\eta)=m(\nu,\eta)
   \qquad(\eta\in\Pesos^+),
\]
since $a(\mu,\eta)=0$ unless $\mu\le\eta$.  Hence the left-hand side becomes
$\sum_\xi c_\xi\widetilde m(\nu,\lambda-\xi)$, while
$\sum_{\mu\ge\nu}\delta_{\mu\lambda}$ is $1$ if $\nu\le\lambda$ and $0$
otherwise.
\end{proof}

\begin{remark}
\label{rem:atomic-vs-multiplicity-recursion}
The two recursions are driven by the same finite-difference operator; they
differ only in the inhomogeneous term.  For multiplicities that term is the
indicator function of $\Pesos(\lambda)$, that is, the girdle itself, which is
already a nontrivial object.  For atomic numbers it is a Kronecker delta.  In
this sense the atomic table, and not the multiplicity table, is the natural
fundamental solution of the operator
$\prod_{\alpha\in\Raiz^{\rm ns}}(1-T_\alpha)$, where $T_\alpha$ denotes the
shift $\nu\mapsto\nu-\alpha$ in the second variable.
\end{remark}

\begin{remark}\label{rem:kappa-inverse}
The recursion has a compact form in the incidence algebra $I(P;\Z)$ of
\cref{subsec:incidence-framework}.  Since $\{\Theta_\mu\}$ and
$\{\chr V(\mu)\}$ are two $\Z$-bases of $\Z[\Pesos]^\Weyl$, there are unique
integers $\kappa(\eta,\lambda)$, vanishing unless $\eta\le\lambda$, with
\[
   \Theta_\lambda=\sum_{\eta\le\lambda}\kappa(\eta,\lambda)\chr V(\eta),
   \qquad \kappa(\lambda,\lambda)=1 ,
\]
so $\kappa\in I(P;\Z)$; and \cref{thm:schutzer-girdle} evaluates it:
\begin{equation}\label{eq:kappa-explicit}
   \kappa(\eta,\lambda)
     =\sum_{\substack{\xi\in\mathcal F_0\\
                      \lambda-\xi+\rho\ \text{regular}\\
                      (\lambda-\xi)^\star=\eta}}
        (-1)^{\ell(w_{\lambda-\xi})}c_\xi .
\end{equation}
Substituting one expansion into the other gives $a*\kappa=\delta$, which is
exactly \eqref{eq:atomic-recursion}.  Since $\kappa(\lambda,\lambda)=1$ for
every $\lambda$, the element $\kappa$ is a unit of $I(P;\Z)$; hence
$a*\kappa=\delta$ forces $a=\kappa^{-1}$, and therefore also
$\kappa*a=\delta$.  Combining with \eqref{eq:global-incidence-identity}
yields $m*\kappa=\zeta*a*\kappa=\zeta$, which is
\cref{cor:multiplicity-recursion}.  Thus the three identities
\[
   a=\Mu*m,
   \qquad
   a*\kappa=\delta,
   \qquad
   m*\kappa=\zeta
\]
are one statement: \emph{the atomic function $a$ is a unit of $I(P;\Z)$, and
\cite[Theorem~2.1]{Schutzer} is an explicit formula for $a^{-1}$}, namely
\eqref{eq:kappa-explicit}.  Equivalently $\kappa=m^{-1}*\zeta$.  In the stable
chamber the recursion represented by this inverse relation becomes
translation-invariant; its inverse kernel is the nonsimple-root partition
function $p_{\rm ns}$, as made precise in \cref{sec:stabilization}.
\end{remark}

\begin{example}\label{ex:A2-recursion}
Here $\Raiz^{\rm ns}=\{\alpha_1+\alpha_2\}=\{\rho\}$, so
$\mathcal F=\{\rho\}$ and $c_\rho=-1$.  Since
$(\lambda-\rho)+\rho=\lambda$, the shifted term is nonzero exactly when
$\lambda$ is regular, and then $(\lambda-\rho)^\star=\lambda-\rho$ with sign
$+1$.  Writing $\lambda=m_1\omega_1+m_2\omega_2$,
\eqref{eq:atomic-recursion-solved} becomes
\[
   a(\mu,\lambda)=
   \begin{cases}
     \delta_{\mu\lambda}+a(\mu,\lambda-\alpha_1-\alpha_2),&m_1,m_2>0,\\
     \delta_{\mu\lambda},&\min\{m_1,m_2\}=0,
   \end{cases}
\]
which is the atomic form of \eqref{eq:A2-character-recursion} and yields
\eqref{eq:A2-atomic-formula} at once.
\end{example}

\begin{example}\label{ex:B2-recursion}
With $\alpha_3=\alpha_1+\alpha_2$ and $\alpha_4=\alpha_1+2\alpha_2$ one has
$\Raiz^{\rm ns}=\{\alpha_3,\alpha_4\}$,
$\mathcal F=\{\alpha_3,\alpha_4,\alpha_3+\alpha_4\}$ and
$c_{\alpha_3}=c_{\alpha_4}=-1$, $c_{\alpha_3+\alpha_4}=1$.  Hence
\[
   a(\mu,\lambda)=\delta_{\mu\lambda}
     +\widetilde a(\mu,\lambda-\alpha_3)
     +\widetilde a(\mu,\lambda-\alpha_4)
     -\widetilde a(\mu,\lambda-\alpha_3-\alpha_4),
\]
an inclusion--exclusion in the two directions $\alpha_3,\alpha_4$; iterating
it reproves \cref{prop:B2-direct} and \cref{cor:B2-atomic}.  In type
$\typeG_2$ the same recursion has $11$ terms, with coefficients $\pm1$
\cite[Table~2]{Schutzer}.
\end{example}

\begin{remark}\label{rem:recursion-cost}
The two computations of $a(\mu,\lambda)$ are complementary rather than
competing.  Formula \eqref{eq:crown-formula} evaluates at most
$2^{\operatorname{rank}\Raiz}$ weight multiplicities of one module, but each
such evaluation is itself expensive.  The recursion
\eqref{eq:atomic-recursion-solved} uses no multiplicities, but its length is
governed by $|\mathcal F|$, which grows quickly with the rank: already
$|\mathcal F|=121$ in $\typeD_4$ and $|\mathcal F|=4781$ in $\typeF_4$
\cite[Table~1]{Schutzer}.  In practice many terms are singular near the walls;
for instance in $\typeD_4$ only $31$ of the $121$ terms are nonzero at
$\lambda=(1111)$, whereas $107$ are nonzero at $\lambda=(3333)$ and all
$121$ once $\lambda$ is deep enough.
\end{remark}

\begin{remark}
\label{rem:recursion-no-positivity}
The coefficients $c_\xi$ have mixed signs, so
\eqref{eq:atomic-recursion-solved} gives no direct induction for
$a(\mu,\lambda)\ge0$, except in the degenerate case
$|\Raiz^{\rm ns}|=1$, i.e.\ irreducible type $\typeA_2$
(\cref{ex:A2-recursion}).  Its use here is structural and computational, and,
through \cref{thm:atomic-stabilization} below, asymptotic.
\end{remark}

The following reductions make the incidence-algebra problem local in the Dynkin diagram.

\subsection{Dynkin-diagram symmetry}

A diagram automorphism $g$ permutes the simple roots and leaves the Cartan
matrix invariant, so the induced permutation of indices preserves the Serre
presentation of $\mathfrak g$ and therefore lifts to an automorphism of
$\mathfrak g$ that stabilizes $\mathfrak h$ and sends $\mathfrak g_{\alpha_i}$
to $\mathfrak g_{\alpha_{g(i)}}$.  Transporting $V(\lambda)$ along this
automorphism gives $V(g\lambda)$ and identifies $V(\lambda)_\nu$ with
$V(g\lambda)_{g\nu}$, whence $m(g\nu,g\lambda)=m(\nu,\lambda)$.  The M\"obius
function is likewise invariant by \cref{lem:dynkin-poset}.

\begin{proposition}\label{prop:atomic-dynkin}
For every diagram automorphism $g$ and every pair $\mu\le\lambda$ of dominant
weights,
\[
     a(g\mu,g\lambda)=a(\mu,\lambda).
\]
\end{proposition}

\begin{proof}
Using \eqref{eq:mobius-atomic},
\begin{align*}
 a(g\mu,g\lambda)
  &=\sum_{\xi\in[g\mu,g\lambda]}
       \Mu(g\mu,\xi)m(\xi,g\lambda)\\
  &=\sum_{\eta\in[\mu,\lambda]}
       \Mu(g\mu,g\eta)m(g\eta,g\lambda)\\
  &=\sum_{\eta\in[\mu,\lambda]}
       \Mu(\mu,\eta)m(\eta,\lambda).
\end{align*}
\end{proof}

\subsection{Reduction to the support of $\lambda-\mu$}

Let
\[
     I=\supp(\lambda-\mu)
       =\{i:\text{the coefficient of }\alpha_i
                    \text{ in }\lambda-\mu\text{ is nonzero}\}.
\]
Let $\Raiz_I$ be the root subsystem generated by $\{\alpha_i:i\in I\}$.
For an ambient weight $\nu$, define its projected weight $\nu_I$ in the
weight lattice of $\Raiz_I$ by its Dynkin labels
\begin{equation}\label{eq:support-projection}
 \langle\nu_I,\alpha_i^\vee\rangle
   =\langle\nu,\alpha_i^\vee\rangle
 \qquad(i\in I).
\end{equation}
Thus $\mu_I$ and $\lambda_I$ below are weights of the subsystem
$\Raiz_I$, and $a(\mu_I,\lambda_I)$ denotes the atomic number computed in
that subsystem.

\begin{theorem}\label{thm:support-reduction}
Let $\mu\le\lambda$ be dominant weights and put
$I=\supp(\lambda-\mu)$.
\begin{enumerate}[label=\textnormal{(\roman*)}]
\item \label{item:support-one}
      $a(\mu,\lambda)=a(\mu_I,\lambda_I)$.
\item \label{item:support-product}
      If $I=I_1\sqcup\cdots\sqcup I_t$ is the decomposition into connected
      components, then
      \[
          a(\mu,\lambda)
            =\prod_{r=1}^t a(\mu_{I_r},\lambda_{I_r}).
      \]
\end{enumerate}
\end{theorem}

\begin{proof}
For \cref{item:support-one}, let $\eta\in[\mu,\lambda]$.  Then
$\lambda-\eta$ lies in $\Z_{\ge0}\{\alpha_i:i\in I\}$, so
$\supp(\lambda-\eta)\subseteq I$ and, for $j\notin I$,
\[
 \langle\eta,\alpha_j^\vee\rangle
   =\langle\lambda,\alpha_j^\vee\rangle
      -\langle\lambda-\eta,\alpha_j^\vee\rangle
   \ge\langle\lambda,\alpha_j^\vee\rangle\ge0,
\]
because $\langle\alpha_i,\alpha_j^\vee\rangle\le0$ for $i\ne j$.  Conversely,
the Cartan matrix of $\Raiz_I$ is invertible, so an element of
$\Z_{\ge0}\{\alpha_i:i\in I\}$ is determined by its Dynkin labels at the
$\alpha_i^\vee$, $i\in I$; given $\eta'\in[\mu_I,\lambda_I]$, let $\xi$ be the
element of $\Z_{\ge0}\{\alpha_i:i\in I\}$ with $\xi_I=\lambda_I-\eta'$.  The
displayed computation makes $\lambda-\xi$ dominant, and
$(\lambda-\xi)-\mu$ has the same labels as $\eta'-\mu_I$ on $I$ and no
coefficient outside $I$, hence lies in $\Z_{\ge0}\{\alpha_i:i\in I\}$.  The
map $\eta\mapsto\eta_I$ is therefore a bijection
$[\mu,\lambda]\to[\mu_I,\lambda_I]$, and it preserves the order in both
directions, since for $\eta,\eta''$ in the interval the difference
$\eta''-\eta$ lies in $\Z\{\alpha_i:i\in I\}$ and is a nonnegative
combination of the $\alpha_i$ exactly when $\eta''_I-\eta_I$ is; compare
\cite[Lemma~3.1]{Stembridge}.  Since $\supp(\lambda-\eta)\subseteq I$ for
every $\eta$ in the interval, the multiplicity is unchanged after restriction
to $\Raiz_I$, that is, $m(\eta,\lambda)=m_I(\eta_I,\lambda_I)$; see, for
example, \cite[Proposition~2.4(1)]{BZ1}.  The M\"obius formula
\eqref{eq:mobius-atomic} therefore gives the equality of atomic numbers.

For \cref{item:support-product}, the interval for the disconnected subsystem
is the direct product of the intervals for its connected components.  Both the
M\"obius function and the weight multiplicity factor over direct products.
Applying \eqref{eq:mobius-atomic} gives the stated product formula.  For the
standard product formula for M\"obius functions see, e.g.,
\cite[Section~3.8]{StanleyEC1}; for
the multiplicity factorization compare \cite[Proposition~2.4(2)]{BZ1}.
\end{proof}

\begin{corollary}\label{cor:minimal-obstruction-connected}
If $a(\mu,\lambda)<0$, then at least one connected component $I_r$ of
$\supp(\lambda-\mu)$ has
\[
    a(\mu_{I_r},\lambda_{I_r})<0.
\]
Thus every minimal obstruction to atomic positivity may be sought in an
irreducible root system.
\end{corollary}

\section{Rank-two atomic combinatorics}\label{sec:rank2}

In type $\typeA_1$, every irreducible module has one-dimensional weight
spaces, so
\[
 \chr V(\lambda)=\Theta_\lambda
\]
and its only nonzero atomic number is $a(\lambda,\lambda)=1$.  We refer
to this as the rank-one case below.

\begin{theorem}\label{thm:rank2-all}
Every irreducible character for a semisimple Lie algebra of rank at most two
admits an atomic decomposition.  In types
$\typeA_1\times\typeA_1$, $\typeA_2$, $\typeB_2$, and $\typeC_2$ all atomic
coefficients are $0$ or $1$; in type $\typeG_2$ they are nonnegative integers
and are given by \eqref{eq:G2-atomic-count} when the crown is a diamond.
\end{theorem}

We first obtain explicit expansions in
$\typeA_1\times\typeA_1$, $\typeA_2$, $\typeB_2$, and $\typeC_2$ from
the rank-two character recursions of \cite{Schutzer}.  Type $\typeG_2$ is
then treated from \eqref{eq:crown-formula} and Kostant's multiplicity formula.

\subsection{Type $\typeA_1\times\typeA_1$}

The reducible rank-two case follows from the tensor-product structure and is
useful to record explicitly.
Write $\lambda=\lambda_1+\lambda_2$ according to the two $\typeA_1$ factors.
Every weight of the irreducible module
$V(\lambda)=V(\lambda_1)\boxtimes V(\lambda_2)$ occurs with multiplicity
one.  Hence its ordinary character is already multiplicity-free.

\begin{proposition}
\label{prop:A1xA1-direct}
For dominant weights $\mu\le\lambda$ in type
$\typeA_1\times\typeA_1$,
\begin{equation}\label{eq:A1xA1-atomic-formula}
 a(\mu,\lambda)=
 \begin{cases}
  1,&\mu=\lambda,\\
  0,&\mu\ne\lambda.
 \end{cases}
\end{equation}
\end{proposition}

\begin{proof}
Each irreducible $\typeA_1$-module has one-dimensional weight spaces.  Therefore
the external tensor product $V(\lambda_1)\boxtimes V(\lambda_2)$ also has
one-dimensional weight spaces, indexed by pairs of weights from the two
factors.  Thus $\chr V(\lambda)=\Theta_\lambda$, and uniqueness of the
triangular expansion in the $\Theta$-basis gives
\eqref{eq:A1xA1-atomic-formula}.
\end{proof}

\subsection{Type $\typeA_2$}

Let $\alpha_3=\alpha_1+\alpha_2$.  If
$\lambda=m_1\omega_1+m_2\omega_2$, the type-$\typeA_2$ recursion of
\cite{Schutzer} reduces, for dominant weights, to
\begin{equation}\label{eq:A2-character-recursion}
 \chr V(\lambda)=
 \begin{cases}
   \Theta_\lambda+\chr V(\lambda-\alpha_3),&m_1,m_2>0,\\
   \Theta_\lambda,&\min\{m_1,m_2\}=0.
 \end{cases}
\end{equation}
Indeed, in the first case
$\lambda-\alpha_3=(m_1-1)\omega_1+(m_2-1)\omega_2$ is dominant, whereas on a
wall the shifted term in the general recursion is singular and vanishes.

\begin{proposition}\label{prop:A2-direct}
Let $\lambda=m_1\omega_1+m_2\omega_2$ and put
$r=\min\{m_1,m_2\}$.  Then
\begin{equation}\label{eq:A2-direct-expansion}
  \chr V(\lambda)=\sum_{i=0}^{r}\Theta_{\lambda-i\alpha_3}.
\end{equation}
Consequently, if $\mu\le\lambda$ and
$\lambda-\mu=k_1\alpha_1+k_2\alpha_2$, then
\begin{equation}\label{eq:A2-atomic-formula}
 a(\mu,\lambda)=
 \begin{cases}
  1,&k_1=k_2,\\
  0,&k_1\ne k_2.
 \end{cases}
\end{equation}
\end{proposition}

\begin{proof}
Iterating \eqref{eq:A2-character-recursion} gives
\eqref{eq:A2-direct-expansion}.  The indices in that expansion are precisely
$\lambda-i(\alpha_1+\alpha_2)$, and each occurs once.  Thus a dominant
$\mu\le\lambda$ occurs as an atomic index exactly when
$\lambda-\mu=i\alpha_1+i\alpha_2$, proving
\eqref{eq:A2-atomic-formula}.
\end{proof}

As a byproduct, comparison of the coefficient of $e^\mu$ on both sides of
\eqref{eq:A2-direct-expansion} recovers the familiar formula
\[
 m(\mu,\lambda)=1+\min\{k_1,k_2,m_1,m_2\}.
\]
Thus the direct atomic expansion and the multiplicity formula are two forms of
the same triangular identity.

\subsection{Types $\typeB_2$ and $\typeC_2$}

We use the following convention for $\typeB_2$:
\[
 \alpha_3=\alpha_1+\alpha_2,\qquad
 \alpha_4=\alpha_1+2\alpha_2,
\]
so that, in the fundamental-weight basis,
\[
 \alpha_1=2\omega_1-2\omega_2,\qquad
 \alpha_2=-\omega_1+2\omega_2,\qquad
 \alpha_3=\omega_1,\qquad
 \alpha_4=2\omega_2.
\]
The type-$\typeB_2$ recursion of \cite{Schutzer} may be iterated in the
$\alpha_3$- and $\alpha_4$-directions.  This yields the following direct
expansion.

\begin{proposition}\label{prop:B2-direct}
Let $\lambda=m_1\omega_1+m_2\omega_2$, let
$t\equiv m_1\pmod 2$, $s\equiv m_2\pmod 2$, with $s,t\in\{0,1\}$.  Then
\begin{equation}\label{eq:B2-direct-expansion}
 \chr V(\lambda)=\Xi_\lambda+\Psi_\lambda,
\end{equation}
where
\begin{equation}\label{eq:B2-Xi}
 \Xi_\lambda=
 \sum_{i=0}^{m_1}\sum_{j=1}^{(m_2-s)/2}
       \Theta_{i\omega_1+(s+2j)\omega_2},
\end{equation}
and
\begin{equation}\label{eq:B2-Psi}
 \Psi_\lambda=
 \begin{cases}
  \Theta_{m_1\omega_1}+\Theta_{(m_1-2)\omega_1}+\cdots+
       \Theta_{t\omega_1},&s=0,\\[1mm]
  \Theta_{m_1\omega_1+\omega_2}+\Theta_{(m_1-1)\omega_1+\omega_2}
       +\cdots+\Theta_{\omega_2},&s=1.
 \end{cases}
\end{equation}
Every index occurring on the right-hand side is distinct.
\end{proposition}

\begin{proof}
Write $\eta=a\omega_1+b\omega_2$, where $a$ and $b$ are its Dynkin
labels, so that $\alpha_3=(10)$ and $\alpha_4=(02)$.  Since
$\alpha_3^\vee=2\alpha_1^\vee+\alpha_2^\vee$ and
$\alpha_4^\vee=\alpha_1^\vee+\alpha_2^\vee$, the weight $\eta$ is regular
exactly when
\begin{equation}\label{eq:B2-regularity}
 a\ne0,\qquad b\ne0,\qquad 2a+b\ne0,\qquad a+b\ne0,
\end{equation}
and $s_2(a\omega_1+b\omega_2)=(a+b)\omega_1-b\omega_2$.  By \cref{ex:B2-recursion},
\eqref{eq:schutzer-girdle} reads
\begin{equation}\label{eq:B2-recursion}
 \Theta_\nu=\widetilde\chi_\nu-\widetilde\chi_{\nu-\alpha_3}
   -\widetilde\chi_{\nu-\alpha_4}+\widetilde\chi_{\nu-\alpha_3-\alpha_4}
 \qquad(\nu\in\Pesos^+).
\end{equation}
Throughout, $\chr V(\eta)$ is read as $0$ when $\eta$ is not dominant.

Put $F(\nu)=\widetilde\chi_\nu-\widetilde\chi_{\nu-\alpha_3}$.  Then
\eqref{eq:B2-recursion} is $F(\nu)=\Theta_\nu+F(\nu-\alpha_4)$, and iterating
it along a chain of dominant weights gives
\begin{equation}\label{eq:B2-alpha4-iteration}
 F(\nu)=\sum_{j=0}^{k}\Theta_{\nu-j\alpha_4}+F(\nu-(k+1)\alpha_4)
\end{equation}
whenever $\nu,\nu-\alpha_4,\ldots,\nu-k\alpha_4$ are dominant.  Assume
$m_1\ge1$ and $m_2\ge2$, and put
\[
 k_3=m_1-1,\qquad k_4=\frac{m_2-2-s}{2}.
\]
For $0\le i\le k_3$ and $0\le j\le k_4+1$ the weight
$\lambda-i\alpha_3-j\alpha_4=(m_1-i)\omega_1+(m_2-2j)\omega_2$ is dominant, so
\eqref{eq:B2-alpha4-iteration} applies with $\nu=\lambda-i\alpha_3$ and the
same $k=k_4$ for every $i$.  Summing the resulting identities over
$0\le i\le k_3$, both the left-hand sides and the terms
$F(\lambda-i\alpha_3-(k_4+1)\alpha_4)$ telescope in the
$\alpha_3$-direction, and we obtain
\[
 \widetilde\chi_\lambda-\widetilde\chi_{\lambda-(k_3+1)\alpha_3}
  =\sum_{i=0}^{k_3}\sum_{j=0}^{k_4}\Theta_{\lambda-i\alpha_3-j\alpha_4}
   +\widetilde\chi_{\lambda-(k_4+1)\alpha_4}
   -\widetilde\chi_{\lambda-(k_3+1)\alpha_3-(k_4+1)\alpha_4}.
\]
Here $\lambda-(k_3+1)\alpha_3=m_2\omega_2$,
$\lambda-(k_4+1)\alpha_4=m_1\omega_1+s\omega_2$ and
$\lambda-(k_3+1)\alpha_3-(k_4+1)\alpha_4=s\omega_2$ are dominant, so every
signed character occurring is an ordinary character and
\begin{align}\label{eq:B2-telescope}
 \chr V(\lambda)={}&
 \sum_{i=0}^{m_1-1}\sum_{j=0}^{(m_2-2-s)/2}
       \Theta_{\lambda-i\alpha_3-j\alpha_4}\\
 &+\chr V(m_2\omega_2)
  +\chr V(m_1\omega_1+s\omega_2)-\chr V(s\omega_2).\nonumber
\end{align}

The three boundary characters are evaluated from \eqref{eq:B2-recursion}
directly; each of the three walls behaves differently, and only the first is a
pure cancellation of nonvanishing terms.

Let $\nu=m_1\omega_1$.  The weight
$\nu-\alpha_3=(m_1-1)\omega_1$ is dominant for $m_1\ge1$, while for $m_1=0$ the
shift $\nu-\alpha_3+\rho=(01)$ is singular; in both cases
$\widetilde\chi_{\nu-\alpha_3}=\chr V((m_1-1)\omega_1)$.  Next,
$\nu-\alpha_4=m_1\omega_1-2\omega_2$ is not dominant, but its shift
$\nu-\alpha_4+\rho=(m_1+1)\omega_1-\omega_2$ is regular by \eqref{eq:B2-regularity} exactly
when $m_1\ge1$, and then $s_2\bigl((m_1+1)\omega_1-\omega_2\bigr)=m_1\omega_1+\omega_2$ is strictly dominant, so
$(\nu-\alpha_4)^\star=(m_1-1)\omega_1$ and
$\widetilde\chi_{\nu-\alpha_4}=-\chr V((m_1-1)\omega_1)$; for $m_1=0$ both
sides are $0$.  In the same way $\nu-\alpha_3-\alpha_4+\rho=m_1\omega_1-\omega_2$ is
regular exactly when $m_1\ge2$, with $s_2\bigl(m_1\omega_1-\omega_2\bigr)=(m_1-1)\omega_1+\omega_2$, so
$\widetilde\chi_{\nu-\alpha_3-\alpha_4}=-\chr V((m_1-2)\omega_1)$ for every
$m_1\ge0$.  The two terms $\pm\chr V((m_1-1)\omega_1)$ cancel and
\begin{equation}\label{eq:B2-wall-1}
 \Theta_{m_1\omega_1}=\chr V(m_1\omega_1)-\chr V((m_1-2)\omega_1)
 \qquad(m_1\ge0).
\end{equation}
Note that the $\alpha_4$-term here is not singular: the identity emerges only
after it cancels the $\alpha_3$-term.

Let $\nu=m_1\omega_1+\omega_2$.  Now
$\nu-\alpha_4+\rho=(m_1+1)\omega_1$ and $\nu-\alpha_3-\alpha_4+\rho=m_1\omega_1$ are
singular by \eqref{eq:B2-regularity}, so both terms vanish, while
$\widetilde\chi_{\nu-\alpha_3}=\chr V((m_1-1)\omega_1+\omega_2)$, this being
$0$ for $m_1=0$ because $\nu-\alpha_3+\rho=(02)$ is then singular.  Hence
\begin{equation}\label{eq:B2-wall-2}
 \Theta_{m_1\omega_1+\omega_2}
  =\chr V(m_1\omega_1+\omega_2)-\chr V((m_1-1)\omega_1+\omega_2)
 \qquad(m_1\ge0).
\end{equation}

Let $\nu=m_2\omega_2$.  Here
$\nu-\alpha_3+\rho=(m_2+1)\omega_2$ and $\nu-\alpha_3-\alpha_4+\rho=(m_2-1)\omega_2$ are
singular, while $\nu-\alpha_4=(m_2-2)\omega_2$ is dominant for $m_2\ge2$ and, for
$m_2\le1$, has singular shift $(10)$ or $(1,-1)$.  Hence
\begin{equation}\label{eq:B2-wall-3}
 \Theta_{m_2\omega_2}=\chr V(m_2\omega_2)-\chr V((m_2-2)\omega_2)
 \qquad(m_2\ge0).
\end{equation}

Iterating \eqref{eq:B2-wall-1}, \eqref{eq:B2-wall-2} and
\eqref{eq:B2-wall-3} down to a nondominant index, which contributes $0$, gives
\begin{align*}
 \chr V(m_1\omega_1)
   &=\Theta_{m_1\omega_1}+\Theta_{(m_1-2)\omega_1}+\cdots+
      \Theta_{t\omega_1},\\
 \chr V(m_1\omega_1+\omega_2)
   &=\Theta_{m_1\omega_1+\omega_2}+\Theta_{(m_1-1)\omega_1+\omega_2}
       +\cdots+\Theta_{\omega_2},\\
 \chr V(m_2\omega_2)
   &=\Theta_{m_2\omega_2}+\Theta_{(m_2-2)\omega_2}+\cdots+
      \Theta_{s\omega_2},
\end{align*}
for all $m_1,m_2\ge0$.  In the boundary cases, where
\eqref{eq:B2-telescope} is unavailable, these three expansions already give
\eqref{eq:B2-direct-expansion}: if $m_2\le1$ then $\Xi_\lambda$ is empty and
$\Psi_\lambda$ is the first or the second expansion according as $s=0$ or
$s=1$; if $m_1=0$ then $\Xi_\lambda$ together with $\Psi_\lambda$ is the third
expansion.  In the remaining case $m_1\ge1$, $m_2\ge2$, substituting the three
expansions into \eqref{eq:B2-telescope}, the final
$\Theta_{s\omega_2}$ cancels the last term.  Since
$\alpha_3=\omega_1$ and $\alpha_4=2\omega_2$, the remaining double sum and the
$\omega_2$-boundary terms combine to give \eqref{eq:B2-Xi}, while the other
boundary terms give \eqref{eq:B2-Psi}.  The second Dynkin label is
at least $2+s$ in \eqref{eq:B2-Xi} and is $0$ or $1$ in
\eqref{eq:B2-Psi}, so the two sets of indices are disjoint.  Within
\eqref{eq:B2-Xi}, equality
\[
 i\omega_1+(s+2j)\omega_2=i'\omega_1+(s+2j')\omega_2
\]
forces $i=i'$ and $j=j'$.  Within \eqref{eq:B2-Psi}, the displayed first
Dynkin labels are pairwise distinct.  Thus every index in
\eqref{eq:B2-direct-expansion} occurs exactly once.
\end{proof}

Define $P_{\typeB_2}(\lambda)$ to be the set of dominant weights indexing the
right-hand side of \eqref{eq:B2-direct-expansion}.  Explicitly,
\begin{align*}
P_{\typeB_2}(\lambda)={}&
 \{i\omega_1+(s+2j)\omega_2:0\le i\le m_1,
            1\le j\le(m_2-s)/2\}\\
&\cup
\begin{cases}
 \{(m_1-2j)\omega_1:0\le j\le(m_1-t)/2\},&s=0,\\
 \{i\omega_1+\omega_2:0\le i\le m_1\},&s=1.
\end{cases}
\end{align*}

\begin{corollary}\label{cor:B2-atomic}
For dominant $\mu\le\lambda$ in type $\typeB_2$,
\[
 a(\mu,\lambda)=
 \begin{cases}
   1,&\mu\in P_{\typeB_2}(\lambda),\\
   0,&\mu\notin P_{\typeB_2}(\lambda).
 \end{cases}
\]
In particular, every irreducible character of type $\typeB_2$ has an atomic
decomposition.
\end{corollary}

\begin{proof}
By \cref{prop:B2-direct},
\[
   \chr V(\lambda)=\sum_{\eta\in P_{\typeB_2}(\lambda)}\Theta_\eta,
\]
and the indices $\eta$ are pairwise distinct.  Uniqueness of the triangular
expansion in \cref{prop:integral-expansion} therefore gives coefficient $1$
for $\eta\in P_{\typeB_2}(\lambda)$ and coefficient $0$ for every other
dominant $\eta\le\lambda$.
\end{proof}

For type $\typeC_2$, let $\iota$ be the root-datum isomorphism obtained by
interchanging the two nodes:
\[
 \iota(\alpha_1^{\typeB})=\alpha_2^{\typeC},\qquad
 \iota(\alpha_2^{\typeB})=\alpha_1^{\typeC},\qquad
 \iota(\omega_1^{\typeB})=\omega_2^{\typeC},\qquad
 \iota(\omega_2^{\typeB})=\omega_1^{\typeC}.
\]
It carries dominant weights to dominant weights, preserves the dominance
order, and satisfies
\[
 \iota\!\left(\chr V_B(\lambda)\right)=\chr V_C(\iota\lambda),
 \qquad
 \iota(\Theta_\eta)=\Theta_{\iota\eta}.
\]
Applying $\iota$ to \eqref{eq:B2-direct-expansion} therefore gives the
$\typeC_2$ expansion.  We denote by $P_{\typeC_2}(\lambda)$ the image of the
corresponding $P_{\typeB_2}(\iota^{-1}\lambda)$ under $\iota$.

\begin{corollary}\label{cor:C2-atomic}
For dominant $\mu\le\lambda$ in type $\typeC_2$,
\[
 a(\mu,\lambda)=
 \begin{cases}
   1,&\mu\in P_{\typeC_2}(\lambda),\\
   0,&\mu\notin P_{\typeC_2}(\lambda).
 \end{cases}
\]
In particular, every irreducible character of type $\typeC_2$ has an atomic
decomposition.
\end{corollary}

Taking coefficients of $e^\nu$ in \eqref{eq:B2-direct-expansion} also gives,
for \emph{dominant} $\nu$, the useful counting formula
\begin{equation}\label{eq:B2-multiplicity-count}
 m(\nu,\lambda)
   =\sum_{\eta\in P_{\typeB_2}(\lambda)}\zeta(\nu,\eta)
   =\#\{\eta\in P_{\typeB_2}(\lambda):\nu\le\eta\}.
\end{equation}
Thus the atomic expansion also gives a direct lattice-point interpretation of the multiplicity.

\subsection{Type $\typeG_2$}\label{subsec:G2}

We now prove positivity in the remaining irreducible rank-two type.  Number the
simple roots so that $\alpha_1$ is short and $\alpha_2$ is long.  Thus
\begin{equation}\label{eq:G2-simple-fundamental}
 \alpha_1=2\omega_1-\omega_2,
 \qquad
 \alpha_2=-3\omega_1+2\omega_2,
\end{equation}
and the four non-simple positive roots are
\begin{equation}\label{eq:G2-nonsimple-roots}
 \beta_1=\alpha_1+\alpha_2,\quad
 \beta_2=2\alpha_1+\alpha_2,\quad
 \beta_3=3\alpha_1+\alpha_2,\quad
 \beta_4=3\alpha_1+2\alpha_2.
\end{equation}
Let $p$ denote Kostant's partition function for all six positive roots and let
$p_*$ denote the partition function using only the four roots in
\eqref{eq:G2-nonsimple-roots}.  Both functions are extended by zero outside the
nonnegative root cone.

\begin{lemma}\label{lem:G2-partition-difference}
For every element $\xi$ of the root lattice,
\begin{equation}\label{eq:G2-partition-difference}
 p_*(\xi)
 =p(\xi)-p(\xi-\alpha_1)-p(\xi-\alpha_2)
      +p(\xi-\alpha_1-\alpha_2).
\end{equation}
\end{lemma}

\begin{proof}
In the completed group algebra,
\[
 \sum_\xi p(\xi)e^\xi
   =\prod_{\alpha\in\Raiz^+}(1-e^\alpha)^{-1}.
\]
Multiplication by $(1-e^{\alpha_1})(1-e^{\alpha_2})$ cancels precisely the
two factors belonging to the simple roots, leaving
$\prod_{j=1}^4(1-e^{\beta_j})^{-1}$.  Comparing the coefficient of $e^\xi$
gives \eqref{eq:G2-partition-difference}.
\end{proof}

We shall also use the following elementary monotonicity fact.  It is included
to make the chain case below independent of any general atomicity statement.

\begin{lemma}\label{lem:root-string-monotonicity}
Let $\lambda$ and $\nu$ be dominant weights and let $\beta$ be a positive root.
Then
\[
       m(\nu,\lambda)\ge m(\nu+\beta,\lambda).
\]
\end{lemma}

\begin{proof}
Consider the $\mathfrak{sl}_2$-subalgebra corresponding to $\beta$.  The sum
\[
   \bigoplus_{k\in\mathbb Z}V(\lambda)_{\nu+k\beta}
\]
is a finite-dimensional module for this $\mathfrak{sl}_2$.  Its Cartan
eigenvalue on $V(\lambda)_{\nu+k\beta}$ is
$\langle\nu,\beta^\vee\rangle+2k$.  Since $\nu$ is dominant,
$\langle\nu,\beta^\vee\rangle\ge0$.  In every finite-dimensional irreducible
$\mathfrak{sl}_2$-module, occurrence of the weight $h+2$ with $h\ge0$ forces
occurrence of the weight $h$, with the same one-dimensional contribution.
After decomposing the displayed module into irreducible
$\mathfrak{sl}_2$-modules and summing these contributions, the claimed
inequality follows.
\end{proof}

The diamond is the only rank-two crown for which monotonicity alone does not
settle the sign.  Write
\[
   \lambda_i=\langle\lambda,\alpha_i^\vee\rangle,
   \qquad
   \mu_i=\langle\mu,\alpha_i^\vee\rangle.
\]

\begin{proposition}\label{prop:G2-diamond}
Assume that $\QO(\mu,\lambda)$ is a diamond and put
$\delta=\lambda-\mu$.  For $j=1,\ldots,4$, let $n_j$ be the multiplicity of
$\beta_j$ in a partition of $\delta$ by the four roots
\eqref{eq:G2-nonsimple-roots}.  Then
\begin{equation}\label{eq:G2-atomic-count}
 a(\mu,\lambda)
 =\#\left\{
 (n_1,n_2,n_3,n_4)\in\mathbb Z_{\ge0}^4:
 \delta=\sum_{j=1}^4n_j\beta_j,
 \ n_2\le\lambda_1,\ n_4\le\lambda_2
 \right\}.
\end{equation}
In particular, $a(\mu,\lambda)\ge0$.
\end{proposition}

\begin{proof}
Since the crown has two atoms, \cref{cor:all-simple-root-atoms} gives that
$\mu+\alpha_1$ and $\mu+\alpha_2$ are dominant and lie below $\lambda$.
From \eqref{eq:G2-simple-fundamental} this implies
$\mu_1\ge3$ and $\mu_2\ge1$; in particular
$\mu+\alpha_1+\alpha_2=(\mu_1-1)\omega_1+(\mu_2+1)\omega_2$ is dominant.
It is the join of the two atoms.  Hence
\cref{cor:atomic-finite-difference} gives
\begin{equation}\label{eq:G2-diamond-difference}
\begin{split}
 a(\mu,\lambda)={}&m(\mu,\lambda)
 -m(\mu+\alpha_1,\lambda)-m(\mu+\alpha_2,\lambda)\\
 &+m(\mu+\alpha_1+\alpha_2,\lambda).
\end{split}
\end{equation}

Set $L=\lambda+\rho$ and $M=\mu+\rho$.  Kostant's multiplicity formula
\cite{MR0109192}, followed by \cref{lem:G2-partition-difference}, transforms
\eqref{eq:G2-diamond-difference} into
\begin{equation}\label{eq:G2-Weyl-pstar}
 a(\mu,\lambda)
   =\sum_{w\in W}(-1)^{\ell(w)}p_*(wL-M).
\end{equation}
We first determine which Weyl-group terms can be nonzero.

If $p_*(u\alpha_1+v\alpha_2)\ne0$, write
\[
 u\alpha_1+v\alpha_2=\sum_{j=1}^4 n_j\beta_j,
 \qquad n_j\in\mathbb Z_{\ge0}.
\]
Using \eqref{eq:G2-nonsimple-roots},
\[
 u=n_1+2n_2+3n_3+3n_4,
 \qquad
 v=n_1+n_2+n_3+2n_4.
\]
Consequently
\[
 u-v=n_2+2n_3+n_4\ge0,
 \qquad
 3v-u=2n_1+n_2+3n_4\ge0,
\]
which is equivalent to
\begin{equation}\label{eq:G2-pstar-cone}
       v\le u\le3v.
\end{equation}
Let
$\gamma=\beta_1=\alpha_1+\alpha_2$ and
$\eta=\beta_3=3\alpha_1+\alpha_2$.  The identities
\[
 \gamma^\vee=\alpha_1^\vee+3\alpha_2^\vee,
 \qquad
 \eta^\vee=\alpha_1^\vee+\alpha_2^\vee
\]
follow by pairing both sides with $\alpha_1$ and $\alpha_2$ and using the
$\typeG_2$ Cartan matrix.  Therefore \eqref{eq:G2-pstar-cone} is equivalent to
\[
 \langle u\alpha_1+v\alpha_2,\gamma^\vee\rangle=3v-u\ge0,
 \qquad
 \langle u\alpha_1+v\alpha_2,\eta^\vee\rangle=u-v\ge0.
\]
Consequently, if $p_*(wL-M)\ne0$, strict dominance of $M$ gives
\[
       \langle wL,\gamma^\vee\rangle>0,
       \qquad
       \langle wL,\eta^\vee\rangle>0.
\]
Since $L$ is strictly dominant, this forces both $w^{-1}\gamma$ and
$w^{-1}\eta$ to be positive roots.

The only possibilities are $w=1,s_1,s_2$.  To see this, let
$\ell(w)\ge2$ and choose a reduced expression
$w=s_{i_1}\cdots s_{i_r}$.  Since $W(\typeG_2)$ is dihedral, the initial two
letters are either $s_1s_2$ or $s_2s_1$.  The inversion set of $w^{-1}$
contains
\[
 \alpha_{i_1},\qquad s_{i_1}(\alpha_{i_2}).
\]
Using
\[
 s_1(\alpha_2)=3\alpha_1+\alpha_2=\eta,
 \qquad
 s_2(\alpha_1)=\alpha_1+\alpha_2=\gamma,
\]
we obtain $w^{-1}\eta<0$ when the word begins with $s_1s_2$, and
$w^{-1}\gamma<0$ when it begins with $s_2s_1$.  Both alternatives contradict
the positivity of $w^{-1}\gamma$ and $w^{-1}\eta$ established above.
Therefore $\ell(w)\le1$, so $w\in\{1,s_1,s_2\}$.

Put $r_i=\langle L,\alpha_i^\vee\rangle=\lambda_i+1$.  Equation
\eqref{eq:G2-Weyl-pstar} therefore reduces to
\begin{equation}\label{eq:G2-three-terms}
 a(\mu,\lambda)
   =p_*(\delta)-p_*(\delta-r_1\alpha_1)
                    -p_*(\delta-r_2\alpha_2).
\end{equation}
It remains to identify the two subtracted terms inside the same set of
partitions counted by $p_*(\delta)$.

Suppose
\[
 \delta-r_1\alpha_1
    =n_1\beta_1+n_2\beta_2+n_3\beta_3+n_4\beta_4.
\]
Pairing with $\alpha_1^\vee$ gives
\[
   -n_1+n_2+3n_3=-r_1-(\mu_1+1),
\]
and hence $n_1\ge r_1$.  Let $\mathcal P_*(\delta)$ denote the set of all partitions of $\delta$ by the
four roots $\beta_j$.  Since $\alpha_1=\beta_2-\beta_1$, the map
\begin{equation}\label{eq:G2-bijection-one}
 (n_1,n_2,n_3,n_4)
 \longmapsto(n_1-r_1,n_2+r_1,n_3,n_4)
\end{equation}
is a bijection from the partitions counted by
$p_*(\delta-r_1\alpha_1)$ onto the set
\begin{equation}\label{eq:G2-B1}
  B_1=\{(n_j)\in\mathcal P_*(\delta):n_2\ge r_1\}.
\end{equation}

For the second subtracted term, suppose
\[
 \delta-r_2\alpha_2
    =n_1\beta_1+n_2\beta_2+n_3\beta_3+n_4\beta_4.
\]
Pairing with $\alpha_2^\vee$ yields
\[
    n_1-n_3+n_4=-r_2-(\mu_2+1),
\]
so $n_3\ge r_2$.  Since $\alpha_2=\beta_4-\beta_3$, the map
\begin{equation}\label{eq:G2-bijection-two}
 (n_1,n_2,n_3,n_4)
 \longmapsto(n_1,n_2,n_3-r_2,n_4+r_2)
\end{equation}
is a bijection onto
\begin{equation}\label{eq:G2-B2}
  B_2=\{(n_j)\in\mathcal P_*(\delta):n_4\ge r_2\}.
\end{equation}

Finally, $B_1$ and $B_2$ are disjoint.  Indeed, for a partition
$(n_j)\in\mathcal P_*(\delta)$ one has
\begin{align*}
 \lambda_1&=\mu_1-n_1+n_2+3n_3,\\
 \lambda_2&=\mu_2+n_1-n_3+n_4.
\end{align*}
If the partition belonged to both $B_1$ and $B_2$, then
$n_2\ge\lambda_1+1$ and $n_4\ge\lambda_2+1$, whence
\[
       n_1\ge\mu_1+3n_3+1,
       \qquad
       n_3\ge\mu_2+n_1+1.
\]
Combining the two inequalities gives
\[
       n_1\ge3n_1+\mu_1+3\mu_2+4,
\]
which is impossible.  Therefore \eqref{eq:G2-three-terms} counts precisely the
partitions in $\mathcal P_*(\delta)\setminus(B_1\cup B_2)$.  Since
$r_i=\lambda_i+1$, this is exactly the set displayed in
\eqref{eq:G2-atomic-count}.
\end{proof}

\begin{example}\label{ex:G2-coefficient-two}
The atomic numbers in type $\typeG_2$ need not be $0$ or $1$.  The smallest example is
obtained from
\[
   \lambda=\rho=\omega_1+\omega_2=(11),
   \qquad
   \mu=\omega_1=(10).
\]
Indeed,
\[
   \chr V(11)=\Theta_{(11)}+\Theta_{(20)}+2\Theta_{(10)},
\]
and therefore
\[
   a(10,11)=2.
\]
\end{example}

\begin{theorem}\label{thm:G2-atomic}
For every pair of dominant weights $\mu\le\lambda$ in type $\typeG_2$,
\[
             a(\mu,\lambda)\ge0.
\]
Consequently every irreducible character of type $\typeG_2$ admits an atomic
decomposition.
\end{theorem}

\begin{proof}
By \cref{cor:rank2-crowns}, the crown is a singleton, a chain, or a diamond.
In the singleton case \eqref{eq:crown-formula} gives
$a(\mu,\lambda)=m(\mu,\lambda)\ge0$.  In the chain case, if $\gamma$ is the
unique atom, then
\[
        a(\mu,\lambda)=m(\mu,\lambda)-m(\gamma,\lambda).
\]
By Stembridge's classification of coverings
\cite[Theorem~2.8]{Stembridge}, $\gamma-\mu$ is a positive root, so the last
difference is nonnegative by \cref{lem:root-string-monotonicity}.  The diamond
case is \cref{prop:G2-diamond}.
\end{proof}

\begin{remark}
The $\typeG_2$ proof uses only the crown structure, Kostant's multiplicity formula,
and elementary root combinatorics.  It is therefore independent of the
character recursion.  The stronger $q$-atomic result of
\cite{Muniz-Plaza-Rojas-G2} is obtained by different methods.
\end{remark}

\begin{proof}[Proof of \cref{thm:rank2-all}]
The rank-one case was recorded at the beginning of the section.  The cases
$\typeA_1\times\typeA_1$, $\typeA_2$, $\typeB_2$, and $\typeC_2$ follow from
\cref{prop:A1xA1-direct,prop:A2-direct,cor:B2-atomic,cor:C2-atomic},
respectively, and type $\typeG_2$ is \cref{thm:G2-atomic}, with the diamond
coefficients given by \cref{prop:G2-diamond}.
\end{proof}

\begin{corollary}
\label{cor:rank2-support}
Let $\mu\le\lambda$ be dominant weights in an arbitrary semisimple root
system, and let $I=\supp(\lambda-\mu)$.  If every connected component of
$\Raiz_I$ has rank at most two, then
\[
       a(\mu,\lambda)\ge0.
\]
Consequently, every negative atomic number has a connected support component
of rank at least three.  \Cref{thm:DE-obstruction} sharpens this later to a
connected component of type $\typeD$ or $\typeE$.
\end{corollary}

\begin{proof}
Apply \cref{thm:support-reduction}.
Every connected component of rank one or two has a nonnegative atomic number
by \cref{thm:rank2-all}; their product is nonnegative.
\end{proof}

\section{Atomic positivity on path Dynkin diagrams}\label{sec:path-positive}

We now prove positivity in every finite type whose Dynkin diagram is a path.
Here and below, a Dynkin diagram is called a \emph{path} when its underlying
unoriented graph is a path; edge multiplicities and arrows are retained in the
Cartan data but are ignored in this graph-theoretic terminology.  The
irreducible systems in question are
\[
   \typeA_n,\qquad \typeB_n,\qquad \typeC_n,\qquad \typeF_4,
   \qquad \typeG_2.
\]
The crosscut formula leaves two tasks.  First, on a path the dominant join of
any collection of atoms is Weyl-conjugate to the corresponding raw sum of
their differences.  Second, the raw Boolean finite difference can be peeled
off one root direction at a time: along an oriented path, alternating primitive
kernels and lowering quotients turns it into the dimension of an explicit
successive subquotient.

We begin with an order-theoretic observation.

\begin{lemma}\label{lem:path-unit-stabilization}
Let $\nu$ be a weight.  Suppose that
\[
   \nu=\nu_0,\nu_1,\ldots,\nu_t
\]
is a sequence such that $\nu_t$ is dominant and, for every $q$, there is a
simple root $\alpha_{i_q}$ satisfying
\[
   \langle\nu_{q-1},\alpha_{i_q}^\vee\rangle=-1,
   \qquad
   \nu_q=s_{i_q}\nu_{q-1}=\nu_{q-1}+\alpha_{i_q}.
\]
Then $\nu_t=\widehat\nu$.  In particular $\widehat\nu$ is Weyl-conjugate to
$\nu$.
\end{lemma}

\begin{proof}
Let $\eta\ge\nu$ be dominant.  We prove inductively that $\eta\ge\nu_q$ for
all $q$.  Suppose $\eta\ge\nu_{q-1}$ and write
\[
   \eta-\nu_{q-1}=\sum_j d_j\alpha_j,
   \qquad d_j\in\Z_{\ge0}.
\]
At $i=i_q$, dominance of $\eta$ gives
\[
 0\le\langle\eta,\alpha_i^\vee\rangle
   =-1+2d_i+\sum_{j\ne i}d_j\langle\alpha_j,\alpha_i^\vee\rangle
   \le -1+2d_i.
\]
Thus $d_i\ge1$, whence
$\eta\ge\nu_{q-1}+\alpha_i=\nu_q$.  Every dominant majorant of $\nu$ therefore
lies above $\nu_t$.  Since $\nu_t$ itself is dominant and lies above $\nu$,
it is the least dominant majorant.
\end{proof}

We shall repeatedly pass from a set of roots to the semisimple subalgebra that
it generates.  Fix a $\Weyl$-invariant inner product $(\ ,\ )$ on
$\Pesos\otimes_\Z\mathbb R$.  Call a subset $\Psi\subseteq\Raiz$ a
\emph{regular subsystem} if $\Psi=-\Psi$ and $\alpha+\beta\in\Psi$ whenever
$\alpha,\beta\in\Psi$ and $\alpha+\beta\in\Raiz$.  Such a $\Psi$ is a root
system in its real span, and
\[
   \mathfrak h_\Psi=\operatorname{span}_\C\{\alpha^\vee:\alpha\in\Psi\},
   \qquad
   \mathfrak g_\Psi=\mathfrak h_\Psi\oplus
        \bigoplus_{\alpha\in\Psi}\mathfrak g_\alpha
\]
is a semisimple subalgebra of $\mathfrak g$ with Cartan subalgebra
$\mathfrak h_\Psi$ and root system $\Psi$.

\begin{lemma}\label{lem:pi-system}
Let $\beta_1,\ldots,\beta_k\in\Raiz$ be linearly independent roots such that
$\beta_r-\beta_s\notin\Raiz$ whenever $r\ne s$, and put
$c_{rs}=\langle\beta_s,\beta_r^\vee\rangle$.  Then $c_{rs}\le0$ for $r\ne s$,
the matrix $C=(c_{rs})$ is a Cartan matrix of finite type, and the roots
occurring in the subalgebra of $\mathfrak g$ generated by
$\mathfrak g_{\pm\beta_1},\ldots,\mathfrak g_{\pm\beta_k}$ form a regular
subsystem $\Psi\subseteq\Raiz$ with simple system $\{\beta_1,\ldots,\beta_k\}$
and Cartan matrix $C$.
\end{lemma}

\begin{proof}
Fix $r\ne s$.  Since $\beta_s-\beta_r$ is not a root, the $\beta_r$-string
through $\beta_s$ is $\beta_s,\beta_s+\beta_r,\ldots,\beta_s+q\beta_r$ for
some $q\ge0$, and therefore $c_{rs}=-q\le0$.  The Gram matrix
$G=\bigl((\beta_r,\beta_s)\bigr)_{r,s}$ is positive definite because the $\beta_r$
are linearly independent, and $C=D^{-1}G$ with
$D=\operatorname{diag}\bigl((\beta_1,\beta_1)/2,\ldots,(\beta_k,\beta_k)/2\bigr)$.  The same relation
holds for every principal submatrix, so all principal minors of $C$ are
positive and $C$ is a Cartan matrix of finite type.

Choose $e_r\in\mathfrak g_{\beta_r}$ and $f_r\in\mathfrak g_{-\beta_r}$ with
$[e_r,f_r]=\beta_r^\vee$, and let $\mathfrak g_\Psi$ be the subalgebra they
generate.  Then $[\beta_r^\vee,\beta_s^\vee]=0$,
$[\beta_r^\vee,e_s]=c_{rs}e_s$ and $[\beta_r^\vee,f_s]=-c_{rs}f_s$, while
$[e_r,f_s]=0$ for $r\ne s$ because $\beta_r-\beta_s$ is neither a root nor
zero.  The root string displayed above gives $\beta_s+m\beta_r\notin\Raiz$ for
$m>q=-c_{rs}$, whence $(\operatorname{ad}e_r)^{1-c_{rs}}e_s=0$ and,
symmetrically, $(\operatorname{ad}f_r)^{1-c_{rs}}f_s=0$.  These are the Serre
relations for $C$, so there is a surjection
$\mathfrak g(C)\to\mathfrak g_\Psi$ from the semisimple Lie algebra with
Cartan matrix $C$.  Its kernel is an ideal meeting the Cartan subalgebra of
$\mathfrak g(C)$ trivially, because the coroots $\beta_r^\vee$ are linearly
independent; since every nonzero ideal of a semisimple Lie algebra contains a
simple summand, and hence meets its Cartan subalgebra, the kernel is zero.
Thus $\mathfrak g_\Psi\cong\mathfrak g(C)$ is semisimple with Cartan
subalgebra $\mathfrak h_\Psi=\operatorname{span}_\C\{\beta_r^\vee\}$ and
simple roots $\beta_1,\ldots,\beta_k$.  Every root space of $\mathfrak g_\Psi$
is spanned by iterated brackets of the $e_r$ and $f_r$, hence lies in a single
$\mathfrak g_\alpha$ with
$\alpha\in\Raiz\cap(\Z\beta_1+\cdots+\Z\beta_k)$; restriction to
$\mathfrak h_\Psi$ is injective on that lattice, so the root system $\Psi$ of
$\mathfrak g_\Psi$ is a subset of $\Raiz$.  Finally $\Psi$ is closed: if
$\alpha,\beta\in\Psi$ and $\alpha+\beta\in\Raiz$, then
$\mathfrak g_{\alpha+\beta}=[\mathfrak g_\alpha,\mathfrak g_\beta]
\subseteq\mathfrak g_\Psi$.
\end{proof}

The next observation is valid without a path hypothesis.

\begin{lemma}\label{lem:atom-disjoint-supports}
Let
\[
   \gamma_1=\mu+\beta_1,\ldots,\gamma_k=\mu+\beta_k
\]
be distinct atoms above a dominant weight $\mu$.  Then
\[
   \supp\beta_r\cap\supp\beta_s=\varnothing
   \qquad(r\ne s).
\]
Moreover, the roots $\beta_1,\ldots,\beta_k$ form a simple system of the
regular root subsystem that they generate.
\end{lemma}

\begin{proof}
By Stembridge's covering theorem \cite[Theorem~2.8]{Stembridge}, each
$\beta_r=\gamma_r-\mu$ is a positive root with connected support.  If
$j\in\supp\beta_r$, then
\[
   \mu<\mu+\alpha_j\le\gamma_r.
\]
Hence
\[
   \mu<\widehat{\mu+\alpha_j}\le\gamma_r,
\]
and the covering relation $\mu\lessdot\gamma_r$ forces
$\widehat{\mu+\alpha_j}=\gamma_r$.  Thus the same index $j$ cannot lie in the
supports of two distinct atoms.

The $\beta_r$ are therefore linearly independent.  If $r\ne s$, then
$\beta_r-\beta_s$ has both positive and negative simple-root coordinates, so
it is not a root.  \Cref{lem:pi-system} therefore applies and shows that
$\{\beta_1,\ldots,\beta_k\}$ is a simple system for the regular subsystem of
$\Raiz$ that it generates.
\end{proof}

On a path the supports in \cref{lem:atom-disjoint-supports} are pairwise
disjoint intervals.  Fix a subset $S$ of the atoms above $\mu$.  We call the
support of $\beta_r$, for $r\in S$, a \emph{selected block}; the selected
blocks are pairwise disjoint intervals of the path, and a block consisting of
one vertex is a \emph{singleton}.  A vertex lying in no selected block is a
\emph{gap}, and a \emph{one-vertex gap} is a gap both of whose neighbours lie
in selected blocks.  The \emph{label} of a vertex $j$ at a weight $\nu$ is its
Dynkin label $\langle\nu,\alpha_j^\vee\rangle$; the label of a gap is a
\emph{gap label}.  We now show that replacing each raw sum $x_S$ by its
least dominant majorant preserves its Weyl orbit and hence its multiplicity
in every finite-dimensional module.

\begin{lemma}\label{lem:path-crown-stabilization}
Assume that the Dynkin diagram of the irreducible root system $\Raiz$ is a
path.  Let
\[
   \gamma_r=\mu+\beta_r\qquad(1\le r\le k)
\]
be distinct atoms above $\mu$.  For $S\subseteq\{1,\ldots,k\}$ put
\[
   x_S=\mu+\sum_{r\in S}\beta_r.
\]
Then $\widehat{x_S}$ is Weyl-conjugate to $x_S$.  Consequently, for every
finite-dimensional $\mathfrak g$-module $V$,
\begin{equation}\label{eq:path-join-multiplicity}
   m_V(\widehat{x_S})=m_V(x_S).
\end{equation}
\end{lemma}

\begin{proof}
There is nothing to prove when $|S|\le1$.  If $\Raiz$ has rank two and
$|S|=2$, disjointness of supports forces the two atom differences to be the
two simple roots.  Then $x_S=\mu+\alpha_1+\alpha_2$ is dominant, since
\[
 \langle x_S,\alpha_1^\vee\rangle
   =\langle\mu+\alpha_2,\alpha_1^\vee\rangle+2\ge2,
 \qquad
 \langle x_S,\alpha_2^\vee\rangle
   =\langle\mu+\alpha_1,\alpha_2^\vee\rangle+2\ge2.
\]
Assume henceforth that $\operatorname{rank}\Raiz\ge3$.

By Stembridge's covering theorem \cite[Theorem~2.8]{Stembridge}, the root of
every nonsingleton selected block is the highest short root of the
subsystem on that block: the unique short root dominant for its simple
roots.  Here shortness is measured within the block subsystem; in a
simply-laced block this is its highest root.  We call this the locally short
dominant root.  The exceptional $\typeG_2$ covering case cannot
occur in rank at least three.

We first locate the negative labels of $x_S$.  Let $N$ be a nonsingleton
selected block with root $\beta_N$, and let $j\in N$.  If no selected block is
immediately adjacent to $N$ at $j$, then
\[
 \langle x_S,\alpha_j^\vee\rangle
   =\mu_j+\langle\beta_N,\alpha_j^\vee\rangle\ge0.
\]
If a selected block with root $\beta_R$ is immediately adjacent there, then
\[
 \langle x_S,\alpha_j^\vee\rangle
   =\langle\beta_N,\alpha_j^\vee\rangle
      +\langle\mu+\beta_R,\alpha_j^\vee\rangle\ge0.
\]
Thus every vertex in a nonsingleton selected block has nonnegative label.

Suppose next that the selected block is the singleton $\{j\}$.  If at most
one selected block is immediately adjacent to it, its label is at least $2$.
If selected blocks with roots $\beta_L,\beta_R$ meet it on both sides, put
\[
 p=-\langle\beta_L,\alpha_j^\vee\rangle,
 \qquad
 q=-\langle\beta_R,\alpha_j^\vee\rangle.
\]
The roots $\beta_L,\alpha_j,\beta_R$ have disjoint supports, so they are
linearly independent and their pairwise differences have both positive and
negative simple-root coordinates.  By \cref{lem:pi-system}, they are the
simple roots of a finite rank-three regular subsystem.  The two outer roots
are orthogonal, and each has negative inner product with $\alpha_j$, so this
subsystem is connected.  Its Dynkin diagram
is therefore of type $\typeA_3$, $\typeB_3$, or $\typeC_3$, so at least one of
its two bonds is simple and
\[
             \min\{p,q\}=1.
\]
Dominance of $\mu+\beta_L$ and $\mu+\beta_R$ gives $\mu_j\ge p,q$.  Hence
\[
 \langle x_S,\alpha_j^\vee\rangle
       =\mu_j+2-p-q
       \ge2-\min\{p,q\}\ge1.
\]
Selected vertices are therefore nonnegative.

Let now $j$ be a gap.  If at most one selected block is adjacent to $j$, its
label is nonnegative by dominance of $\mu$ or of the corresponding atom.
Thus a negative label can occur only at a one-vertex gap between two selected
blocks.  With $\beta_L,\beta_R$ and $p,q$ as above, the same rank-three
argument gives $\min\{p,q\}=1$, while dominance of the two atoms gives
$\mu_j\ge p,q$.  Therefore
\begin{equation}\label{eq:path-gap-minus-one}
 \langle x_S,\alpha_j^\vee\rangle
       =\mu_j-p-q\ge-1.
\end{equation}
Every negative label of $x_S$ is consequently equal to $-1$, at a one-vertex
gap.  Such gaps are pairwise nonadjacent.  Reflect at all of them; the
reflections commute, and after reflection every such gap has label $1$.  Let
$y$ be the resulting weight.

Only vertices adjacent to reflected gaps can have decreased labels.  We first
consider endpoints of nonsingleton selected blocks.  Note that in types
$\typeA$, $\typeB$ and $\typeC$ the bond joining such an endpoint to an
adjacent reflected gap is always simple.  Indeed, the only multiple bond in
these types joins the last two vertices $n-1$ and $n$; a reflected gap is a
one-vertex gap and hence not terminal, so it cannot be $n$, and it cannot be
$n-1$ with $n$ the exposed endpoint of a nonsingleton block, since such a
block would contain $n-1$ as well.  Reflection at a gap of label $-1$ across a
simple bond decreases the label of the adjacent vertex by exactly $1$.  In
type $\typeF_4$ the exterior bond may be multiple; that configuration is
treated separately below.

The required local check is now short.  In type $\typeA$, the locally short
root on an $\typeA_m$ block is $\alpha_1+\cdots+\alpha_m$, and its exposed endpoint
has local label $1$.  For a terminal $\typeB_m$ block, number its roots from
the exposed endpoint, with $\alpha_m$ short.  Its locally short dominant root
is $\alpha_1+\cdots+\alpha_m$, whose label at $\alpha_1$ is $1$, including
when $m=2$.  Every
nonterminal block in types $\typeB$ and $\typeC$ is of type $\typeA$.

For a terminal $\typeC_m$ block, number the roots, up to reversal, so that the
exposed endpoint is $\alpha_1$.  For $m=2$ the locally short root is
$\alpha_1+\alpha_2$, while for $m\ge3$ it is
\[
 \alpha_1+2\alpha_2+\cdots+2\alpha_{m-1}+\alpha_m.
\]
Its local labels at the exposed endpoint and the next inward vertex are $0$
and $1$, respectively, and the exterior bond is simple.  Hence a reflected
gap can make the exposed endpoint negative only by changing its label to
$-1$.

It remains to check $\typeF_4$.  Number the diagram
\[
        1-2\Rightarrow3-4,
        \qquad
        \langle\alpha_2,\alpha_3^\vee\rangle=-2,\quad
        \langle\alpha_3,\alpha_2^\vee\rangle=-1.
\]
A reflected gap is either $2$ or $3$.  If it is $3$, the only nonsingleton
block on its left is the $\typeA_2$ block $\{1,2\}$; its exposed label is $1$ and
the decrement is $1$.  If the gap is $2$, the only nonsingleton block on its
right is the $\typeA_2$ block $\{3,4\}$, whose root is $\alpha_3+\alpha_4$.
Its exposed local label at $3$ is $1$, while reflection at $2$ decreases that
label by $2$.  Thus this is the only $\typeF_4$ case in which a new negative
label can occur, and that label is again $-1$.

We must also check selected singletons adjacent to reflected gaps.  If a
singleton is adjacent to one reflected gap, its label before that reflection
is at least $2$: if a selected block lies immediately on the other side, this
follows from dominance of that atom, and otherwise it is $\mu_j+2\ge2$.
A gap reflection decreases such a label by at most $2$, so it remains
nonnegative.  If two reflected gaps flank the singleton, both incident bonds
are simple.  Indeed, in types $\typeB$ and $\typeC$ the multiple bond is
terminal, whereas a reflected gap is not terminal.  Two reflected gaps
flanking a singleton require at least five distinct vertices, so this
configuration cannot occur in $\typeF_4$.  The total decrement is therefore
$2$, so the singleton remains
nonnegative.

It follows that every negative label of $y$ is a $-1$ at the exposed endpoint
of either a terminal $\typeC_m$ block or the block $\{3,4\}$ in the
$\typeF_4$ configuration above.  Moreover there is at most one such endpoint:
in type $\typeC$ every such block contains the unique terminal multiple edge,
so two cannot occur among the disjoint selected blocks, and in $\typeF_4$
there is only the single configuration just described.

If no such endpoint occurs, $y$ is dominant.  Otherwise reflect at it.  For a
terminal $\typeC_m$ block, the reflected gap has label $1$ and the next inward
label is at least $1$.  The inward decrement is exactly $1$, since
$\langle\alpha_1,\alpha_2^\vee\rangle=-1$ also when $m=2$ and that inward
bond is multiple.  The endpoint reflection therefore changes the gap label
to $0$ and leaves the inward label nonnegative.  In the exceptional $\typeF_4$ case,
reflection at $3$ changes the label at the gap $2$ from $1$ to $0$ and
decreases the label at $4$ by $1$ from a value at least $1$.  Thus in either
case the resulting weight is dominant.

We have reached a dominant weight from $x_S$ by simple reflections, each
applied at Dynkin label $-1$.  By \cref{lem:path-unit-stabilization}, this
weight is $\widehat{x_S}$ and is Weyl-conjugate to $x_S$.  Weyl invariance of
weight multiplicities gives \eqref{eq:path-join-multiplicity}.
\end{proof}

We next isolate the representation-theoretic part.  The proof uses only
ordinary $\mathfrak{sl}_2$ strings and elementary exact sequences.

\begin{lemma}\label{lem:path-regular-boolean}
Let $\Psi\subseteq\Raiz$ be a regular semisimple root subsystem whose Dynkin
components are paths, with simple roots $\beta_1,\ldots,\beta_k$.  Let $\nu$
be a weight such that $\nu$ and every $\nu+\beta_i$ are $\Psi$-dominant.  Then
for every finite-dimensional $\mathfrak g$-module $V$,
\begin{equation}\label{eq:path-regular-boolean}
   \sum_{S\subseteq\{1,\ldots,k\}}(-1)^{|S|}
      m_V\!\left(\nu+\sum_{i\in S}\beta_i\right)\ge0.
\end{equation}
More precisely, fix an order of the path components, choose an endpoint of
each, and process the components one after another, taking the simple roots of
each in order from its chosen endpoint.  Starting with $V$, alternately take
the kernel of the corresponding raising operator and the quotient by the image
of the corresponding lowering operator, beginning with a kernel on each
component.  Let $\mathfrak o$ denote the component order together with these
endpoint choices, and write $R_{\Psi,\mathfrak o}(V)$ for the final
$\mathfrak h$-graded subquotient.
For the empty subsystem, use the convention $R_{\varnothing,\mathfrak o}(V)=V$.
Then the left-hand side of \eqref{eq:path-regular-boolean} is
$\dim R_{\Psi,\mathfrak o}(V)_\nu$.
\end{lemma}

\begin{proof}
Consider first one path component and number its simple roots from the chosen
endpoint:
\[
             \beta_1-\beta_2-\cdots-\beta_\ell .
\]
Choose $\mathfrak{sl}_2$-triples $(e_r,f_r,\beta_r^\vee)$.  For distinct
simple roots of the regular subsystem one has $[e_r,f_s]=0$; if they are
nonadjacent, their $\mathfrak{sl}_2$-triples commute.  Indeed, any ambient
root causing one of these brackets to be nonzero would belong to $\Psi$ by
regularity, contradicting the corresponding simple-root relation in $\Psi$.
We define successive subquotients by
\[
 V^{(0)}=V,\qquad
 V^{(r)}=
 \begin{cases}
   \ker(e_r:V^{(r-1)}\to V^{(r-1)}),&r\text{ odd},\\[1mm]
   V^{(r-1)}/f_rV^{(r-1)},&r\text{ even}.
 \end{cases}
\]
We justify recursively that every operator in this definition acts on the
space indicated.  A commuting operator preserves the kernel and the image of
a given operator, so it acts on that kernel and on the corresponding
quotient.  In particular, after stage $t$ the space $V^{(t)}$ retains the
full $\mathfrak{sl}_2(\beta_s)$-action for every $s\ge t+2$: each of the
operators used at stages $1,\ldots,t$ commutes with that entire triple.
For $r\ge2$, the space $V^{(r-2)}$ thus retains the full current triple.
The last operation before stage $r$ uses $f_{r-1}$ when $r$ is odd and
$e_{r-1}$ when $r$ is even.  The identities
$[e_r,f_{r-1}]=0$ and $[f_r,e_{r-1}]=0$, respectively, make the chosen
stage-$r$ operator act on $V^{(r-1)}$.  This establishes the recursion,
starting from the full action on $V^{(0)}$.
All these operators are homogeneous for the full $\mathfrak h$-grading, so
their kernels and images are graded and every $V^{(r)}$ remains
$\mathfrak h$-graded.  Nonzero rescaling of a chosen root vector changes
neither its kernel nor its image, so introduces no additional choice in the
construction.

We use the elementary $\mathfrak{sl}_2$ fact that, in a finite-dimensional
module, the map from weight $h$ to weight $h+2$ induced by $e$ is surjective
for $h\ge-1$, while the reverse map induced by $f$ is injective.
These assertions apply to the individual full weight spaces needed here.
Indeed, whenever a graded space $M$ retains the full
$\mathfrak{sl}_2(\beta_r)$-action, the finite-dimensional space
\[
             \bigoplus_{q\in\Z}M_{\eta+q\beta_r}
\]
is a module for this $\mathfrak{sl}_2$, with distinct Cartan eigenvalues
$\langle\eta,\beta_r^\vee\rangle+2q$.  Its eigenspaces at $q=0,1$ are
exactly $M_\eta$ and $M_{\eta+\beta_r}$.

Fix $r$ and let $T$ be any set of simple roots of $\Psi$ that have not yet
been processed, that is, a subset of $\{\beta_{r+1},\ldots,\beta_\ell\}$
together with an arbitrary set of simple roots of the components not yet
treated.  Put $\beta_T=\sum_{\beta\in T}\beta$ and $\eta=\nu+\beta_T$.  Simple
roots lying in different components are nonadjacent, so among the roots of $T$
only $\beta_{r+1}$ can be adjacent to $\beta_r$.  Hence, with the second term
omitted when $r=\ell$,
\begin{equation}\label{eq:path-current-label}
 \langle\eta,\beta_r^\vee\rangle
   =\langle\nu,\beta_r^\vee\rangle
     -\mathbf 1_{\{\beta_{r+1}\in T\}}
        \bigl(-\langle\beta_{r+1},\beta_r^\vee\rangle\bigr)
   \ge0,
\end{equation}
by dominance of $\nu$ and of $\nu+\beta_{r+1}$.

Suppose first that $r$ is odd.  If $r=1$, \eqref{eq:path-current-label} makes
\[
 e_1:V_\eta\longrightarrow V_{\eta+\beta_1}
\]
surjective.  If $r>1$, all operations defining $V^{(r-2)}$ involve roots nonadjacent to
$\beta_r$, so $V^{(r-2)}$ is a finite-dimensional
$\mathfrak{sl}_2(\beta_r)$-module.  Moreover
$V^{(r-1)}=V^{(r-2)}/f_{r-1}V^{(r-2)}$.  The map
\[
 e_r:V^{(r-2)}_\eta\longrightarrow
       V^{(r-2)}_{\eta+\beta_r}
\]
is surjective by \eqref{eq:path-current-label}, and it descends to the
quotient because $[e_r,f_{r-1}]=0$.  Surjectivity descends as well.  Thus
\begin{equation}\label{eq:path-kernel-exact}
 0\longrightarrow V^{(r)}_\eta
 \longrightarrow V^{(r-1)}_\eta
 \xrightarrow{\ e_r\ }V^{(r-1)}_{\eta+\beta_r}
 \longrightarrow0 .
\end{equation}

Now suppose that $r$ is even.  All operations defining $V^{(r-2)}$ involve
roots nonadjacent to $\beta_r$, so $V^{(r-2)}$ is a finite-dimensional
$\mathfrak{sl}_2(\beta_r)$-module, and
$V^{(r-1)}=\ker(e_{r-1}:V^{(r-2)}\to V^{(r-2)})$.  Again by
\eqref{eq:path-current-label},
\[
 f_r:V^{(r-2)}_{\eta+\beta_r}\longrightarrow V^{(r-2)}_\eta
\]
is injective.  Since $[e_{r-1},f_r]=0$, it restricts injectively to the two
kernels, and hence
\begin{equation}\label{eq:path-quotient-exact}
 0\longrightarrow V^{(r-1)}_{\eta+\beta_r}
 \xrightarrow{\ f_r\ }V^{(r-1)}_\eta
 \longrightarrow V^{(r)}_\eta
 \longrightarrow0 .
\end{equation}

Both exact sequences give
\begin{equation}\label{eq:path-successive-difference}
 \dim V^{(r)}_\eta
   =\dim V^{(r-1)}_\eta
      -\dim V^{(r-1)}_{\eta+\beta_r}.
\end{equation}
In applying induction at stage $r-1$, both $T$ and
$T\cup\{\beta_r\}$ consist of roots not yet processed at that stage.
Thus \eqref{eq:path-successive-difference} yields, for every such $T$,
\[
 \dim V^{(r)}_{\nu+\beta_T}
   =\sum_{S\subseteq\{1,\ldots,r\}}(-1)^{|S|}
      m_V(\nu+\beta_T+\beta_S),
 \qquad \beta_S:=\sum_{i\in S}\beta_i.
\]
Taking $r=\ell$ and $T=\varnothing$ gives the desired Boolean difference for
this component as the dimension of $V^{(\ell)}_\nu$.

For a forest, process the components one at a time in the chosen order,
restarting with a kernel at the chosen endpoint of each component.  The $\mathfrak{sl}_2$-triples belonging to
different components commute, so after one component has been processed the
resulting subquotient is still a module for every remaining component.  The
same argument therefore applies successively and yields
\eqref{eq:path-regular-boolean}, with final space $R_{\Psi,\mathfrak o}(V)$.
The inclusion of arbitrary roots from later components in $T$ ensures that
the dimension identities compose across components.  If there are no roots,
both sides are $\dim V_\nu$ by the stated convention.
\end{proof}

We can now prove the all-ranks positivity theorem.

\begin{theorem}\label{thm:path-positive}
Let $\mathfrak g$ be semisimple and suppose that every connected component of
its Dynkin diagram is a path.  Then every atomic number is nonnegative.
In particular, every irreducible root system of type
\[
       \typeA_n,\qquad \typeB_n,\qquad \typeC_n,
       \qquad \typeF_4,\qquad \typeG_2
\]
is atomic.
\end{theorem}

\begin{proof}
By \cref{thm:support-reduction} it suffices to treat an irreducible path.  Fix
$\mu\le\lambda$.  If $\mu=\lambda$, then $a(\mu,\lambda)=1$.
Otherwise the finite interval $[\mu,\lambda]$ has at least one atom.  Let
\[
       \gamma_1=\mu+\beta_1,\ldots,\gamma_k=\mu+\beta_k
\]
be the atoms of $[\mu,\lambda]$, and let $\Psi$ be the regular subsystem
generated by the $\beta_i$.  By \cref{lem:atom-disjoint-supports}, the
supports of the $\beta_i$ are disjoint intervals.  Consequently the Dynkin
diagram of the regular subsystem $\Psi$ is a forest of paths.  To see this
directly, write $\beta_r=\sum_i b_{ri}\alpha_i$ with coefficients positive on
its support.  Two such roots are orthogonal if their supports are not joined
by an ambient edge.  If the unique edge joining their supports has endpoints
$i,j$, then
\[
             (\beta_r,\beta_s)=b_{ri}b_{sj}(\alpha_i,\alpha_j)<0.
\]
Thus a bond of $\Psi$ joins exactly those two blocks that are immediately
adjacent in the ambient path.  Ordering the blocks along that path gives
the claimed forest, with the induced Cartan multiplicities retained.  For
$S\subseteq\{1,\ldots,k\}$, the full weight-lattice join defined in
\eqref{eq:full-root-join} is, by disjointness of the simple-root supports,
\[
       x_S=\mu+\sum_{i\in S}\beta_i.
\]
Repeated use of \cref{lem:dominant-covering-join} gives
\[
       \bigvee_{i\in S}\gamma_i=\widehat{x_S}.
\]
By \cref{lem:path-crown-stabilization} and Weyl invariance,
\[
       m\!\left(\bigvee_{i\in S}\gamma_i,\lambda\right)
          =m(x_S,\lambda).
\]
\Cref{cor:atomic-finite-difference} therefore gives
\[
 a(\mu,\lambda)
   =\sum_{S\subseteq\{1,\ldots,k\}}(-1)^{|S|}
       m\!\left(\mu+\sum_{i\in S}\beta_i,\lambda\right).
\]
Every $\beta_i$ is a positive ambient root, so its coroot is a nonnegative
integral combination of the ambient simple coroots.  Ambient dominance of
$\mu$ and of every $\mu+\beta_i=\gamma_i$ therefore implies their
$\Psi$-dominance.  The last expression is consequently nonnegative by
\cref{lem:path-regular-boolean}.
\end{proof}

\begin{corollary}
\label{cor:path-primitive-model}
Assume that $\Raiz$ is irreducible and its Dynkin diagram is a path.  Let
$\mu\le\lambda$, let
\[
       \gamma_i=\mu+\beta_i\qquad(1\le i\le k)
\]
be the atoms of $[\mu,\lambda]$, and let $\Psi$ be the regular subsystem with
simple roots $\beta_1,\ldots,\beta_k$.  Fix an order of the components of the
Dynkin forest of $\Psi$ and an endpoint of each, let $\mathfrak o$ record
these choices, and let
$R_{\Psi,\mathfrak o}(V(\lambda))$ be the successive subquotient constructed
in \cref{lem:path-regular-boolean}.  Then
\begin{equation}\label{eq:path-primitive-model}
       a(\mu,\lambda)=\dim R_{\Psi,\mathfrak o}(V(\lambda))_\mu.
\end{equation}
\end{corollary}

\begin{proof}
If $\mu=\lambda$, the subsystem is empty and both sides of
\eqref{eq:path-primitive-model} equal $1$.  Otherwise, the proof of
\cref{thm:path-positive} identifies $a(\mu,\lambda)$ with the
Boolean difference in \eqref{eq:path-regular-boolean}.  The refined statement
of that lemma identifies this difference with
$\dim R_{\Psi,\mathfrak o}(V(\lambda))_\mu$.
\end{proof}

\begin{remark}\label{rem:path-threshold}
The path hypothesis enters the two parts of the argument in complementary
ways.  In \cref{lem:path-crown-stabilization}, a vertex has at most two
neighbors, and every negative label encountered in the stated stabilization
procedure is $-1$.  The corresponding reflection adds exactly one simple
root, so \cref{lem:path-unit-stabilization} identifies the final weight with
the least dominant majorant without changing its Weyl orbit.
In \cref{lem:path-regular-boolean}, orienting a
path gives each current vertex at most one unprocessed neighbor.  Dominance of
that neighboring shift keeps the relevant $\mathfrak{sl}_2$ label
nonnegative, while all earlier vertices except the immediate predecessor are
nonadjacent and hence commute with the current $\mathfrak{sl}_2$.  The
predecessor is handled by a single kernel or quotient.

A fork destroys this one-dimensional elimination scheme: at a branch vertex
one cannot arrange both the processed and unprocessed interactions into a
single predecessor and a single successor.  This does not by itself force the
raw Boolean difference to be negative---indeed the $\typeD_4$ raw difference
is nonnegative by \cref{lem:D4-raw-positive}---but it explains why the path
argument ceases to apply and why a separate local analysis is needed in
\cref{sec:D4}.
\end{remark}

\section{Forks and negative atomic numbers}\label{sec:D4}\label{sec:DE-obstruction}

We use the usual Bourbaki numbering.  In type $\typeD_4$ the diagram is
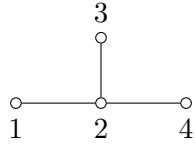
\[
\begin{tikzpicture}[baseline=-.4ex,scale=.75]
  \node[circle,draw,inner sep=1.4pt,label=below:$1$] (a1) at (0,0) {};
  \node[circle,draw,inner sep=1.4pt,label=below:$2$] (a2) at (1.5,0) {};
  \node[circle,draw,inner sep=1.4pt,label=above:$3$] (a3) at (1.5,1.15) {};
  \node[circle,draw,inner sep=1.4pt,label=below:$4$] (a4) at (3,0) {};
  \draw (a1)--(a2)--(a4);
  \draw (a2)--(a3);
\end{tikzpicture}
\]
so $\alpha_2$ is the trivalent simple root and
$\alpha_1,\alpha_3,\alpha_4$ are the three outer roots.  The triality group is
the diagram-automorphism group $S_3$ permuting these three outer roots.  A
string $a_1a_2a_3a_4$ denotes $\sum_i a_i\omega_i$.

For the whole $\typeD/\typeE$ family, we use the Bourbaki layout and numbering
of \cref{fig:DE-numbering}.

\begin{figure}[ht]
\centering
\pgfmathsetmacro{\dvr}{.15}
\begin{tabular}{cc}
\begin{tikzpicture}[scale=0.43]
\draw (-1,0) node[anchor=east] {$\typeD_n$};
\draw (0,0)--(2,0);
\draw[dashed] (2,0)--(4,0);
\draw (4,0)--(6,0.7);
\draw (4,0)--(6,-0.7);
\draw[fill=white] (0,0) circle (\dvr cm) node[below=6pt]{$\scriptstyle 1$};
\draw[fill=white] (2,0) circle (\dvr cm) node[below=4pt]{$\scriptstyle 2$};
\draw[fill=white] (4,0) circle (\dvr cm) node[below=4pt]{$\scriptstyle n-2$};
\draw[fill=white] (6,0.7) circle (\dvr cm) node[right=3pt]{$\scriptstyle n-1$};
\draw[fill=white] (6,-0.7) circle (\dvr cm) node[right=3pt]{$\scriptstyle n$};
\end{tikzpicture}
&
\begin{tikzpicture}[scale=0.43]
\draw (-1,0) node[anchor=east] {$\typeE_6$};
\draw (0,0)--(8,0);
\draw (4,0)--+(0,2);
\draw[fill=white] (0,0) circle (\dvr cm) node[below=4pt]{$\scriptstyle 1$};
\draw[fill=white] (2,0) circle (\dvr cm) node[below=4pt]{$\scriptstyle 3$};
\draw[fill=white] (4,0) circle (\dvr cm) node[below=4pt]{$\scriptstyle 4$};
\draw[fill=white] (6,0) circle (\dvr cm) node[below=4pt]{$\scriptstyle 5$};
\draw[fill=white] (8,0) circle (\dvr cm) node[below=4pt]{$\scriptstyle 6$};
\draw[fill=white] (4,2) circle (\dvr cm) node[right=3pt]{$\scriptstyle 2$};
\end{tikzpicture}
\\[2.0em]
\begin{tikzpicture}[scale=0.40]
\draw (-1,0) node[anchor=east] {$\typeE_7$};
\draw (0,0)--(10,0);
\draw (4,0)--+(0,2);
\draw[fill=white] (0,0) circle (\dvr cm) node[below=4pt]{$\scriptstyle 1$};
\draw[fill=white] (2,0) circle (\dvr cm) node[below=4pt]{$\scriptstyle 3$};
\draw[fill=white] (4,0) circle (\dvr cm) node[below=4pt]{$\scriptstyle 4$};
\draw[fill=white] (6,0) circle (\dvr cm) node[below=4pt]{$\scriptstyle 5$};
\draw[fill=white] (8,0) circle (\dvr cm) node[below=4pt]{$\scriptstyle 6$};
\draw[fill=white] (10,0) circle (\dvr cm) node[below=4pt]{$\scriptstyle 7$};
\draw[fill=white] (4,2) circle (\dvr cm) node[right=3pt]{$\scriptstyle 2$};
\end{tikzpicture}
&
\begin{tikzpicture}[scale=0.36]
\draw (-1,0) node[anchor=east] {$\typeE_8$};
\draw (0,0)--(12,0);
\draw (4,0)--+(0,2);
\draw[fill=white] (0,0) circle (\dvr cm) node[below=4pt]{$\scriptstyle 1$};
\draw[fill=white] (2,0) circle (\dvr cm) node[below=4pt]{$\scriptstyle 3$};
\draw[fill=white] (4,0) circle (\dvr cm) node[below=4pt]{$\scriptstyle 4$};
\draw[fill=white] (6,0) circle (\dvr cm) node[below=4pt]{$\scriptstyle 5$};
\draw[fill=white] (8,0) circle (\dvr cm) node[below=4pt]{$\scriptstyle 6$};
\draw[fill=white] (10,0) circle (\dvr cm) node[below=4pt]{$\scriptstyle 7$};
\draw[fill=white] (12,0) circle (\dvr cm) node[below=4pt]{$\scriptstyle 8$};
\draw[fill=white] (4,2) circle (\dvr cm) node[right=3pt]{$\scriptstyle 2$};
\end{tikzpicture}
\end{tabular}
\caption{The $\typeD/\typeE$ diagrams in the Bourbaki numbering.  In
$\typeD_4$ the trivalent root is $\alpha_2$; in the exceptional $\typeE$ diagrams the
trivalent root is $\alpha_4$.}
\label{fig:DE-numbering}
\end{figure}
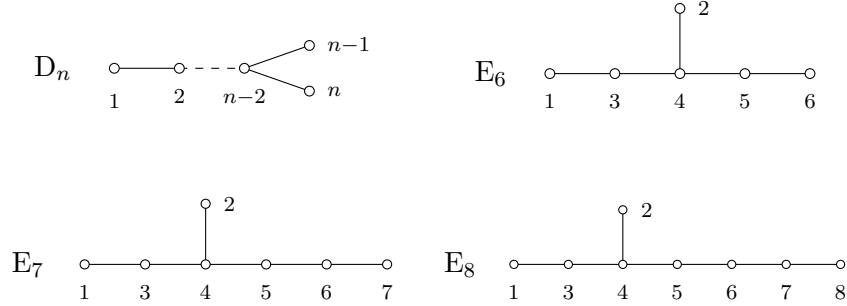

\subsection{A primitive finite-difference tool}\label{subsec:primitive-finite-difference}

We first record a representation-theoretic finite-difference tool used in the fork analysis.  Let $V$ be a finite-dimensional
$\mathfrak g$-module and write $m_V(\nu)=\dim V_\nu$, with the convention
that $m_V(\nu)=0$ when $\nu$ is not a weight of $V$.

For a set $J$ of pairwise nonadjacent simple roots, define the simultaneous
primitive space
\[
 \operatorname{Prim}_J(V)_\nu
   =\{v\in V_\nu:e_jv=0\text{ for every }j\in J\}.
\]

\begin{lemma}\label{lem:primitive-finite-difference}
Let $V$ be a finite-dimensional $\mathfrak g$-module and let $J$ be a set of
pairwise nonadjacent simple roots.
\begin{enumerate}[label=\textnormal{(\roman*)}]
\item \label{item:prim-identity}
If $\nu$ is dominant with respect to the subsystem generated by $J$, then
\begin{equation}\label{eq:primitive-difference}
 \sum_{S\subseteq J}(-1)^{|S|}
       m_V\!\left(\nu+\sum_{j\in S}\alpha_j\right)
   =\dim\operatorname{Prim}_J(V)_\nu .
\end{equation}
\item \label{item:prim-inequality}
If in addition $c\notin J$, the weight $\nu+\alpha_c$ is $J$-dominant, and
$\langle\nu+\alpha_c,\alpha_c^\vee\rangle>0$, then
\begin{equation}\label{eq:star-boolean-positive}
 \sum_{S\subseteq J\cup\{c\}}(-1)^{|S|}
       m_V\!\left(\nu+\sum_{j\in S}\alpha_j\right)\ge0.
\end{equation}
\end{enumerate}
\end{lemma}

\begin{proof}
Let
\[
 \mathfrak s_J=\bigoplus_{j\in J}\mathfrak{sl}_2^{(j)},
 \qquad
 \mathfrak h_J=\operatorname{span}_{\mathbb C}
      \{\alpha_j^\vee:j\in J\},
 \qquad
 \mathfrak t_J=\{h\in\mathfrak h:\alpha_j(h)=0\text{ for all }j\in J\}.
\]
Because the vertices in $J$ are pairwise nonadjacent,
$\langle\alpha_i,\alpha_j^\vee\rangle=2\delta_{ij}$ for $i,j\in J$.
Hence $\mathfrak h_J\cap\mathfrak t_J=0$, and comparison of dimensions gives
\[
 \mathfrak h=\mathfrak h_J\oplus\mathfrak t_J.
\]
Moreover $\mathfrak t_J$ centralizes $\mathfrak s_J$.  After fixing a
$\mathfrak t_J$-character, a full $\mathfrak h$-weight is uniquely determined
by its $\mathfrak h_J$-restriction.  Every irreducible summand of the resulting
$\mathfrak s_J$-module has the form
\[
 M=\boxtimes_{j\in J}L(n_j),
\]
and every $\mathfrak s_J$-weight space of $M$ is one-dimensional.  Fix such a
summand.  Suppose first that $M_\nu=0$.  If $M_{\nu+\alpha_S}\ne0$ for some
$S\subseteq J$, where $\alpha_S=\sum_{j\in S}\alpha_j$, then for every
$j\in J$ the integer
$\langle\nu,\alpha_j^\vee\rangle+2\cdot\mathbf 1_{\{j\in S\}}$ is a weight
of $L(n_j)$, hence lies in $[-n_j,n_j]$ and is congruent to $n_j$ modulo $2$;
since $\langle\nu,\alpha_j^\vee\rangle\ge0$ by $J$-dominance, the integer
$\langle\nu,\alpha_j^\vee\rangle$ then lies in $[-n_j,n_j]$ as well, so
$M_\nu\ne0$.  This contradiction shows that all the terms $M_{\nu+\alpha_S}$
vanish, so the contribution of $M$ to both sides of
\eqref{eq:primitive-difference} is zero.  This is the only point at which
$J$-dominance of $\nu$ is used.  If $M_\nu\ne0$, define
\[
 \epsilon_j=
 \begin{cases}
 1,&M_{\nu+\alpha_j}\ne0,\\
 0,&M_{\nu+\alpha_j}=0.
 \end{cases}
\]
Because the $J$-directions are independent,
\[
 \mathbf 1_{\{M_{\nu+\sum_{j\in S}\alpha_j}\ne0\}}
      =\prod_{j\in S}\epsilon_j.
\]
Hence the contribution of $M$ to the left-hand side of
\eqref{eq:primitive-difference} is
\[
 \sum_{S\subseteq J}(-1)^{|S|}\prod_{j\in S}\epsilon_j
      =\prod_{j\in J}(1-\epsilon_j).
\]
This equals $1$ exactly when no raising step $\nu\mapsto\nu+\alpha_j$ is
available, i.e. exactly when the one-dimensional space $M_\nu$ is annihilated
by every $e_j$, and otherwise it equals $0$.  Summing over all irreducible
summands gives
\eqref{eq:primitive-difference}.

For $j\in J$ we have $[e_j,f_c]=0$.  Hence
\[
 f_c:\operatorname{Prim}_J(V)_{\nu+\alpha_c}
       \longrightarrow \operatorname{Prim}_J(V)_\nu.
\]
On the source weight the $\mathfrak{sl}_2^{(c)}$-weight is
$\langle\nu+\alpha_c,\alpha_c^\vee\rangle>0$, so the lowering operator $f_c$
is injective there.  Thus
\[
 \dim\operatorname{Prim}_J(V)_\nu
 \ge \dim\operatorname{Prim}_J(V)_{\nu+\alpha_c}.
\]
Using \eqref{eq:primitive-difference} on the two sides gives
\eqref{eq:star-boolean-positive}.
\end{proof}

\subsection{The defect at a trivalent vertex}\label{subsec:D4-defect}

Put
\[
 J=\{1,3,4\},\qquad
 \sigma=\alpha_1+\alpha_2+\alpha_3+\alpha_4,
 \qquad
 \theta=\alpha_1+2\alpha_2+\alpha_3+\alpha_4.
\]
Thus $\theta$ is the highest root.  For a finite-dimensional $\typeD_4$-module
$V=V(\lambda)$ and an arbitrary weight $\nu$, write
$m_V(\nu)=\dim V_\nu$.  Define the raw Boolean finite difference
\begin{equation}\label{eq:D4-raw-boolean}
 B_\lambda(\mu)
   =\sum_{S\subseteq\{1,2,3,4\}}(-1)^{|S|}
      m_V\!\left(\mu+\sum_{i\in S}\alpha_i\right).
\end{equation}

\begin{lemma}\label{lem:D4-raw-positive}
For every dominant weight $\mu$ and every highest weight $\lambda$,
\[
             B_\lambda(\mu)\ge0.
\]
More precisely, if all three outer Dynkin labels of $\mu$ are positive, then
\begin{equation}\label{eq:D4-raw-cokernel}
 B_\lambda(\mu)
 =\dim\frac{\operatorname{Prim}_J(V)_\mu}
 {f_2\operatorname{Prim}_J(V)_{\mu+\alpha_2}}.
\end{equation}
If at least one outer Dynkin label is zero, then
\begin{equation}\label{eq:D4-raw-wall}
 B_\lambda(\mu)=\dim\operatorname{Prim}_J(V)_\mu.
\end{equation}
\end{lemma}

\begin{proof}
Apply \cref{lem:primitive-finite-difference}\ref{item:prim-identity} to the
independent set $J=\{1,3,4\}$.  If $\mu_1,\mu_3,\mu_4>0$, both $\mu$ and
$\mu+\alpha_2$ are $J$-dominant, and $f_2$ commutes with all three outer
raising operators.  Since
$\langle\mu+\alpha_2,\alpha_2^\vee\rangle=\mu_2+2>0$, the central lowering
operator is injective on the source weight space.  This gives
\eqref{eq:D4-raw-cokernel}.

Suppose instead that $\mu_i=0$ for some $i\in J$.  In the half of
\eqref{eq:D4-raw-boolean} containing $\alpha_2$, pair the terms indexed by
$T\subseteq J\setminus\{i\}$ and $T\cup\{i\}$.  The weight
$\mu+\alpha_2+\sum_{j\in T}\alpha_j$ has $i$-th Dynkin label $-1$, so the
simple reflection $s_i$ carries it to the paired weight.  Weyl invariance of
weight multiplicities makes the two terms cancel.  The remaining half is the
$J$-Boolean difference at the $J$-dominant weight $\mu$, hence equals
$\dim\operatorname{Prim}_J(V)_\mu$ by
\eqref{eq:primitive-difference}.
\end{proof}

The next elementary stabilization calculation isolates exactly where the raw
Boolean cube differs from the dominant crown.

\begin{proposition}\label{prop:D4-defect}
Let $\mu\le\lambda$ be dominant weights in type $\typeD_4$, and write
$\mu_i=\langle\mu,\alpha_i^\vee\rangle$.
\begin{enumerate}[label=\textnormal{(\roman*)}]
\item If $\mu_2\ge2$, then
\[
             a(\mu,\lambda)=B_\lambda(\mu)\ge0.
\]
\item If $\mu_2=0$, put $Z=\{i\in J:\mu_i=0\}$ and
\[
 \beta_Z=
 \begin{cases}
  \alpha_2+\displaystyle\sum_{i\in Z}\alpha_i,&|Z|\le2,\\[1mm]
  \theta,&|Z|=3.
 \end{cases}
\]
Then $\mu+\beta_Z$ is the least element of $(\mu,\lambda]$ whenever the latter
is nonempty.  In particular $[\mu,\lambda]$ has at most one atom, namely
$\mu+\beta_Z$, and $a(\mu,\lambda)\ge0$.
\item If $\mu_2=1$, put
\[
             \tau=\mu+\sigma.
\]
With the convention $a(\tau,\lambda)=0$ when $\tau\nleq\lambda$, one has
\begin{equation}\label{eq:D4-defect-atomic}
       B_\lambda(\mu)=a(\mu,\lambda)+a(\tau,\lambda),
\end{equation}
and, whenever $\tau\le\lambda$,
\begin{equation}\label{eq:D4-defect-multiplicity}
 a(\mu,\lambda)
   =B_\lambda(\mu)+m(\mu+\theta,\lambda)-m(\tau,\lambda).
\end{equation}
In particular $a(\tau,\lambda)=m(\tau,\lambda)-m(\mu+\theta,\lambda)\ge0$.
\end{enumerate}
\end{proposition}

\begin{proof}
For $S\subseteq\{1,2,3,4\}$ set
$x_S=\mu+\sum_{i\in S}\alpha_i$, and write $r=|S\cap J|$.  Its Dynkin labels
are
\begin{align*}
 \langle x_S,\alpha_2^\vee\rangle
   &=\mu_2+2\mathbf 1_{\{2\in S\}}-r,\\
 \langle x_S,\alpha_i^\vee\rangle
   &=\mu_i+2\mathbf 1_{\{i\in S\}}-
      \mathbf 1_{\{2\in S\}},\qquad i\in J.
\end{align*}
Assume first that $\mu_2\ge2$.  If $2\notin S$, every outer label is
nonnegative and the central label is $\mu_2-r\ge-1$.  Equality $-1$ requires
$\mu_2=2$ and $r=3$; reflecting at $2$ changes the central label to $1$ and
each selected outer label from $\mu_i+2$ to $\mu_i+1\ge1$, so the result is
dominant.  If $2\in S$, the central label is
$\mu_2+2-r\ge\mu_2-1\ge1$.  A negative outer label can occur only at an
unselected $i\in J$ with $\mu_i=0$, where it equals $-1$.  Put
\[
 Z_S=\{i\in J\setminus S:\mu_i=0\}.
\]
The reflections $s_i$, $i\in Z_S$, commute.  After applying all of them, each
affected outer label is $1$, and the central label is
\[
 \mu_2+2-r-|Z_S|
   \ge \mu_2+2-r-(3-r)=\mu_2-1\ge1.
\]
Hence the stabilized weight is dominant and is obtained only by reflections
at label $-1$.  By \cref{lem:path-unit-stabilization} the stabilized weight is
$\widehat{x_S}$, so $m(\widehat{x_S},\lambda)=m(x_S,\lambda)$ for every
$S\subseteq\{1,2,3,4\}$.  Substituting this in
\eqref{eq:atomic-finite-difference-simple} gives exactly
\eqref{eq:D4-raw-boolean}, and (i) follows from
\cref{lem:D4-raw-positive}.

Assume $\mu_2=0$, and let
$Z=\{i\in J:\mu_i=0\}$.  The Dynkin labels of $\nu=\mu+\alpha_2$ are
\[
 \langle\nu,\alpha_2^\vee\rangle=2,
 \qquad
 \langle\nu,\alpha_i^\vee\rangle=\mu_i-1\quad(i\in J).
\]
Thus the negative labels are exactly the mutually orthogonal vertices in
$Z$.  Reflecting at all of them gives
\[
 \nu_Z=\mu+\alpha_2+\sum_{i\in Z}\alpha_i.
\]
Its outer labels are nonnegative and its central label is $2-|Z|$.  Hence
$\nu_Z$ is dominant when $|Z|\le2$.  If $|Z|=3$, its central label is $-1$;
reflecting once more at $2$ gives
\[
 s_2\nu_Z=\mu+2\alpha_2+\alpha_1+\alpha_3+\alpha_4
          =\mu+\theta,
\]
whose central label is $1$ and whose three outer labels are $0$.  By the unit
stabilization argument, in both cases
\[
 \widehat{\mu+\alpha_2}=\mu+\beta_Z
\]
with $\beta_Z$ as in the statement.
Now let $i\in J$.  Any dominant majorant of $\mu+\alpha_i$ must repair the
central label $-1$, and therefore contains at least one copy of $\alpha_2$.
It is consequently also a majorant of $\mu+\alpha_2$, so
\[
 \widehat{\mu+\alpha_2}\le\widehat{\mu+\alpha_i}.
\]
Let now $\nu$ be any dominant weight with $\mu<\nu\le\lambda$, and choose $i$
with $\alpha_i$ occurring with positive coefficient in $\nu-\mu$.  Then
$\nu\ge\mu+\alpha_i$, and $\nu$ is dominant, so
$\nu\ge\widehat{\mu+\alpha_i}\ge\mu+\beta_Z$; for $i=2$ this is immediate.
Hence $\mu+\beta_Z$ is the least element of $(\mu,\lambda]$ whenever that
interval is nonempty, and it is then its unique atom.  Each such $\beta_Z$ is
a positive root of $\typeD_4$, so
\cref{lem:root-string-monotonicity} gives
$m(\mu,\lambda)\ge m(\mu+\beta_Z,\lambda)$ whenever that atom lies below
$\lambda$.  In this case \cref{cor:atomic-finite-difference} gives
$a(\mu,\lambda)=m(\mu,\lambda)-m(\mu+\beta_Z,\lambda)\ge0$.
If the interval has no atom, then $\lambda=\mu$ and $a(\mu,\mu)=1$.
This proves (ii).

Finally suppose $\mu_2=1$.  For
$S\subseteq\{1,2,3,4\}$ put $r=|S\cap J|$.  The Dynkin labels of
$x_S=\mu+\sum_{i\in S}\alpha_i$ are
\begin{align*}
 \langle x_S,\alpha_2^\vee\rangle
    &=1+2\mathbf 1_{\{2\in S\}}-r,\\
 \langle x_S,\alpha_i^\vee\rangle
    &=\mu_i+2\mathbf 1_{\{i\in S\}}-
      \mathbf 1_{\{2\in S\}},\qquad i\in J.
\end{align*}
If $2\notin S$ and $r\le1$, all labels are nonnegative.  If $2\notin S$
and $r=2$, the central label is $-1$; reflecting at $2$ makes it $1$, and the
only possible new negative label is $-1$ at the unique unselected outer
vertex with $\mu_i=0$, which is removed by one outer reflection.  If
$2\in S$, the central label is $3-r\ge0$, while every negative outer label is
exactly $-1$; reflecting at those mutually orthogonal outer vertices lowers
the central label by their number, which is at most $3-r$, so the result is
dominant.  Thus each of these raw weights reaches a dominant weight by
reflections at Dynkin label $-1$.  By \cref{lem:path-unit-stabilization}
that weight is its least dominant majorant, so Weyl invariance gives
$m(\widehat{x_S},\lambda)=m(x_S,\lambda)$ for every $S\ne J$, while
\eqref{eq:atomic-finite-difference-simple} expresses $a(\mu,\lambda)$ as the
alternating sum of the $m(\widehat{x_S},\lambda)$.

The sole remaining subset is $S=J$, for which
\[
 x_J=\mu+\alpha_1+\alpha_3+\alpha_4,
 \qquad
 \langle x_J,\alpha_2^\vee\rangle=-2.
\]
The weight
\[
 \tau=x_J+\alpha_2=\mu+\sigma
\]
is dominant: its central label is $0$ and its outer labels are
$\mu_i+1\ge1$.  It is also the least dominant majorant.  Indeed, if
$\eta\ge x_J$ is dominant and $\eta-x_J=\sum_i d_i\alpha_i$, then
\[
 0\le\langle\eta,\alpha_2^\vee\rangle
   =-2+2d_2-d_1-d_3-d_4\le -2+2d_2,
\]
so $d_2\ge1$ and therefore $\eta\ge x_J+\alpha_2=\tau$.  By contrast, the
simple reflection uses the coefficient
$-\langle x_J,\alpha_2^\vee\rangle=2$:
\[
 s_2x_J=x_J+2\alpha_2=\mu+\theta.
\]
In particular, $\tau$ and $x_J$ are not Weyl-conjugate: for the invariant
inner product fixed above,
\[
 (\tau,\tau)-(x_J,x_J)
   =2(x_J,\alpha_2)+(\alpha_2,\alpha_2)
   =-(\alpha_2,\alpha_2)<0.
\]
Thus Weyl invariance identifies the multiplicity at $x_J$ with that at
$\mu+\theta$, but gives no such identification with the multiplicity at
its least dominant majorant $\tau$.
The subset $J$ has odd cardinality.  Consequently replacing the raw term
$-m_V(x_J)=-m(\mu+\theta,\lambda)$ by its crown term
$-m(\tau,\lambda)$ gives the multiplicity identity
\eqref{eq:D4-defect-multiplicity}, whether or not $\tau\le\lambda$.
If $\tau\nleq\lambda$, then neither $\tau$ nor $\mu+\theta$ is a weight of
$V(\lambda)$: both are dominant, and $\mu+\theta=\tau+\alpha_2\ge\tau$.
Both multiplicities therefore vanish, giving $a(\mu,\lambda)=B_\lambda(\mu)$
and proving \eqref{eq:D4-defect-atomic} under the stated zero convention.
Hence assume $\tau\le\lambda$.  Now $\tau_2=0$ and all three outer Dynkin
labels of $\tau$ are positive.
Part~(ii), applied with lower weight $\tau$, therefore shows that the only
possible atom above $\tau$ is
$\tau+\alpha_2=\mu+\theta$.  If $\mu+\theta\le\lambda$, this is the
unique atom of $[\tau,\lambda]$, and \cref{cor:atomic-finite-difference} gives
\[
 a(\tau,\lambda)=m(\tau,\lambda)-m(\mu+\theta,\lambda).
\]
If $\mu+\theta\nleq\lambda$, the interval $[\tau,\lambda]$ has no atom;
then $m(\mu+\theta,\lambda)=0$ and the same identity still holds.  In both
cases it is equivalent to \eqref{eq:D4-defect-atomic}, and
\cref{lem:root-string-monotonicity}, applied at $\tau$ with
$\beta=\alpha_2$, gives $a(\tau,\lambda)\ge0$.
\end{proof}

For the two refinements below it is convenient to record a safe finite
difference form of Kostant's formula.  Let $p_{\rm ns}$ denote the Kostant
partition function obtained by using only the nonsimple positive roots.

\begin{lemma}\label{lem:D4-pns}
For any dominant $\mu\le\lambda$ in $\typeD_4$,
\begin{equation}\label{eq:D4-pns}
 B_\lambda(\mu)
  =\sum_{w\in W(\typeD_4)}(-1)^{\ell(w)}
    p_{\rm ns}\bigl(w(\lambda+\rho)-(\mu+\rho)\bigr).
\end{equation}
\end{lemma}

\begin{proof}
Insert Kostant's multiplicity formula into
\eqref{eq:D4-raw-boolean} and interchange the two finite sums.  For a fixed
Weyl-group term the inner Boolean difference is
\[
 \sum_{S\subseteq\{1,2,3,4\}}(-1)^{|S|}
 p\!\left(\xi-\sum_{i\in S}\alpha_i\right).
\]
Multiplication of the Kostant generating function by
$\prod_i(1-e^{-\alpha_i})$ cancels precisely the four simple-root factors,
leaving the generating function for partitions by the nonsimple positive
roots.  The displayed sum is therefore $p_{\rm ns}(\xi)$.
\end{proof}

The eight nonsimple positive roots are
\begin{equation}\label{eq:D4-nonsimple-roots}
\begin{gathered}
 \gamma_1=\alpha_1+\alpha_2,\quad
 \gamma_3=\alpha_2+\alpha_3,\quad
 \gamma_4=\alpha_2+\alpha_4,\\
 h_{13}=\alpha_1+\alpha_2+\alpha_3,\quad
 h_{14}=\alpha_1+\alpha_2+\alpha_4,\quad
 h_{34}=\alpha_2+\alpha_3+\alpha_4,\\
 \sigma=\alpha_1+\alpha_2+\alpha_3+\alpha_4,
 \qquad
 \theta=\alpha_1+2\alpha_2+\alpha_3+\alpha_4.
\end{gathered}
\end{equation}

\begin{lemma}
\label{lem:simply-laced-parabolic-bound}
Let $\Phi$ be a finite simply-laced root system with simple roots
$\{\alpha_j:j\in I\}$, let $L$ be a strictly dominant integral weight, and write
\[
       L-wL=\sum_{j\in I} q_j\alpha_j
       \qquad(q_j\in\mathbb Z_{\ge0}).
\]
Fix $i\in I$.  If $w$ does not belong to the standard parabolic subgroup
generated by the simple reflections $s_j$ with $j\ne i$, then
\begin{equation}\label{eq:simply-laced-parabolic-bound}
               q_i\ge \langle L,\alpha_i^\vee\rangle.
\end{equation}
\end{lemma}

\begin{proof}
Let $\omega_i^\vee$ be the $i$-th fundamental coweight.  Since
$q_i=\langle L-wL,\omega_i^\vee\rangle$,
\[
 q_i=\langle L,\omega_i^\vee-w^{-1}\omega_i^\vee\rangle.
\]
The coweight $\omega_i^\vee$ is dominant, so
$\omega_i^\vee-w^{-1}\omega_i^\vee$ is a nonnegative integral combination
of the simple coroots.  Its stabilizer is precisely the standard parabolic
subgroup generated by the $s_j$, $j\ne i$.  If $w$ is outside that
parabolic, the difference is nonzero.  Its $\alpha_i^\vee$-coefficient cannot be zero: in simply-laced type the
remaining simple coroots are orthogonal to the fundamental coweight
$\omega_i^\vee$, so if
$\delta^\vee=\omega_i^\vee-w^{-1}\omega_i^\vee$ were a nonzero combination of
only those coroots, then
\[
 \|w^{-1}\omega_i^\vee\|^2
   =\|\omega_i^\vee-\delta^\vee\|^2
   =\|\omega_i^\vee\|^2+\|\delta^\vee\|^2
   >\|\omega_i^\vee\|^2,
\]
contradicting Weyl invariance of the norm.  Thus the
$\alpha_i^\vee$-coefficient is at least $1$.  Strict dominance of $L$ now
gives \eqref{eq:simply-laced-parabolic-bound}.
\end{proof}

\begin{lemma}\label{lem:D4-boundary-estimates}
Let $\mu\le\lambda$ be dominant weights in type $\typeD_4$.
Assume $\mu_2=1$ and put $\tau=\mu+\sigma$.
\begin{enumerate}[label=\textnormal{(\alph*)}]
\item If $\lambda=\tau$, then $B_\lambda(\mu)=1$.  If, for some
$i\in J$ and $q\ge1$,
\[
       \lambda=\tau+q(\alpha_2+\alpha_i),
\]
then $B_\lambda(\mu)\ge1$.
\item Suppose $\tau\le\lambda$, write
\[
 \varepsilon=\lambda-\tau
   =u_1\alpha_1+u_2\alpha_2+u_3\alpha_3+u_4\alpha_4,
 \qquad
 E=u_1+u_3+u_4-u_2.
\]
Then
\[
 E<0\Longrightarrow a(\tau,\lambda)=0,
 \qquad
 E=0\Longrightarrow a(\tau,\lambda)=1\ \text{ and }\ B_\lambda(\mu)\ge1.
\]
\end{enumerate}
\end{lemma}

\begin{proof}
For (a), first take $\lambda=\tau$.  Then
$\delta=\lambda-\mu=\sigma$.  In \eqref{eq:D4-pns} the identity term is
$p_{\rm ns}(\sigma)=1$.  Indeed, the central coefficient of $\sigma$ is $1$,
and every nonsimple positive root in \eqref{eq:D4-nonsimple-roots} has
positive central coefficient.  Thus a partition of $\sigma$ can contain only
one nonsimple root with central coefficient $1$; matching the three outer
coefficients, all equal to $1$, forces that root to be $\sigma$ itself.  Put $L=\lambda+\rho$.  Here $\lambda_2=0$, so
\cref{lem:simply-laced-parabolic-bound} implies that every
$w\notin W_J=\langle s_1,s_3,s_4\rangle$ subtracts central coefficient at
least $1$.  The residual central coefficient is therefore negative or zero.  In the
zero case the residual vector cannot vanish: vanishing would give
$w(\lambda+\rho)=\mu+\rho$, but $\lambda+\rho$ and $\mu+\rho$ are distinct
strictly dominant weights, whereas a Weyl orbit contains a unique dominant
weight.  Thus in the case of zero central simple-root coefficient the residual vector is
nonzero.  Since every nonsimple positive root has positive central
coefficient, both possibilities give $p_{\rm ns}=0$.  If $1\ne w\in W_J$, then
\[
 L-wL=\sum_{i\in T}(\lambda_i+1)\alpha_i
\]
for a nonempty $T\subseteq J$.  Since
$\lambda_i=\tau_i=\mu_i+1$, each coefficient $\lambda_i+1\ge2$, whereas the
outer coefficients of $\sigma$ are all $1$; hence the residual vector has a
negative outer coefficient.  Thus all nonidentity Weyl terms vanish and
$B_\tau(\mu)=1$.

For the second assertion use triality and take $i=1$.  Then
\[
 \delta=\lambda-\mu
   =(q+1)\alpha_1+(q+1)\alpha_2+\alpha_3+\alpha_4,
 \qquad \lambda_2=q.
\]
Let $L=\lambda+\rho$.  By \cref{lem:simply-laced-parabolic-bound}, every
$w\notin W_J$ subtracts central coefficient at least
$L_2=\lambda_2+1=q+1$, which equals the central coefficient of $\delta$.
If the residual central coefficient is negative, the corresponding
$p_{\rm ns}$ term is zero.  If it is zero, the residual vector is again
nonzero: otherwise $w(\lambda+\rho)=\mu+\rho$ would identify two distinct
strictly dominant weights in one Weyl orbit.  Since every nonsimple positive
root has positive central coefficient, this case also contributes zero.
Hence only $W_J$ contributes.  The excess
$F(x)=x_1+x_3+x_4-x_2$ satisfies $F(\delta)=2$; among the nonsimple roots,
$F(\gamma_i)=0$, $F(h_{13})=F(h_{14})=F(h_{34})=F(\theta)=1$, and
$F(\sigma)=2$.  Since $F(\delta)=2$, every nonsimple-root partition of $\delta$ is of one
of two forms: either it contains one copy of $\sigma$ and all remaining roots
have excess $0$, or it contains two roots from
\[
 H=\{h_{13},h_{14},h_{34},\theta\}
\]
and all remaining roots have excess $0$.  In the first case
$\delta-\sigma=q\gamma_1$, giving $q\gamma_1+\sigma$.  In the second case,
the $\alpha_3$- and $\alpha_4$-coefficients of $\delta$ are both $1$; among
unordered pairs from $H$, the only pair whose coefficients in both positions
are at most $1$ is $h_{13}+h_{14}$.  The remainder is
$(q-1)\gamma_1$.  Therefore the identity argument has exactly the two
partitions
\[
 (q-1)\gamma_1+h_{13}+h_{14},
 \qquad
 q\gamma_1+\sigma.
\]
For $w_T=\prod_{i\in T}s_i\in W_J$ the outer reflections commute and
\[
 L-w_TL=\sum_{i\in T}r_i\alpha_i,
 \qquad r_i=\lambda_i+1.
\]
Now $r_1=\mu_1+q+2>q+1=\delta_1$, so every term with $1\in T$ vanishes.  For
$i=3,4$, dominance of $\lambda$ gives $q\le\mu_i+1$, while
$r_i=\mu_i+2-q$.  Thus $r_i\le1$ is possible exactly when
$q=\mu_i+1$, equivalently $\lambda_i=0$, and then $r_i=1$.  In that case
\[
 p_{\rm ns}(\delta-\alpha_3)=1,
 \qquad
 p_{\rm ns}(\delta-\alpha_4)=1,
\]
with unique partitions $q\gamma_1+h_{14}$ and $q\gamma_1+h_{13}$,
respectively; if both labels vanish, then
$p_{\rm ns}(\delta-\alpha_3-\alpha_4)=1$, with unique partition
$(q+1)\gamma_1$.  Therefore
\[
 B_\lambda(\mu)
 =2-\mathbf 1_{\{\lambda_3=0\}}
    -\mathbf 1_{\{\lambda_4=0\}}
    +\mathbf 1_{\{\lambda_3=\lambda_4=0\}}
 \ge1.
\]

For (b), first consider the one-atom coefficient
$a(\tau,\lambda)=m(\tau,\lambda)-m(\tau+\alpha_2,\lambda)$.  Applying
Kostant's formula and cancelling only the factor corresponding to
$\alpha_2$ gives a partition function using all positive roots except
$\alpha_2$.  Every such root
$x_1\alpha_1+x_2\alpha_2+x_3\alpha_3+x_4\alpha_4$ satisfies
\[
       x_1+x_3+x_4-x_2\ge0,
\]
and equality holds precisely for $\gamma_1,\gamma_3,\gamma_4$.
If $E<0$, the identity term therefore vanishes.  Here
$\lambda_2=2u_2-(u_1+u_3+u_4)=u_2-E$.  Thus
$\lambda_2+1>u_2$.  By \cref{lem:simply-laced-parabolic-bound}, a Weyl element outside
$W_J$ makes the central simple-root coefficient of the partition argument negative, while
an element of $W_J$ subtracts only outer simple roots and therefore decreases
the displayed excess.  Hence every Weyl term vanishes and
$a(\tau,\lambda)=0$.

If $E=0$, the identity term has the unique partition
\[
       \varepsilon=u_1\gamma_1+u_3\gamma_3+u_4\gamma_4,
\]
For $w\notin W_J$, \cref{lem:simply-laced-parabolic-bound} now subtracts central
coefficient at least $\lambda_2+1=u_2+1$, and for nontrivial $w\in W_J$ the
outer subtraction makes the excess negative.  Thus every nonidentity Weyl
term vanishes and $a(\tau,\lambda)=1$.
It remains to prove $B_\lambda(\mu)\ge1$.  Put
$k_i=u_i$ for $i\in J$ and $K=k_1+k_3+k_4=u_2$.  The difference
$\delta=\lambda-\mu=\sigma+\varepsilon$ has ``outer-minus-central'' excess
$2$.  In \eqref{eq:D4-nonsimple-roots}, the three $\gamma_i$ have excess
$0$, the four roots
\[
      h_{13},\ h_{14},\ h_{34},\ \theta
\]
have excess $1$, and $\sigma$ has excess $2$.  Consequently every
nonsimple-root partition of $\delta$ is either the canonical partition
\[
      \sigma+k_1\gamma_1+k_3\gamma_3+k_4\gamma_4
\]
or consists of two of the four excess-one roots, with repetition allowed,
plus excess-zero roots.

The central coefficient of $\delta=\sigma+\varepsilon$ is $1+K$, since each
$\gamma_i$ has central coefficient $1$ and $u_1+u_3+u_4=K$.  As
$\lambda_2=K$, \cref{lem:simply-laced-parabolic-bound} shows that every
$w\notin W_J$ subtracts central coefficient at least $\lambda_2+1=K+1$, so the
residual central coefficient is at most $0$.  If it is negative the
corresponding $p_{\rm ns}$ term vanishes; if it is zero then, every nonsimple
positive root having positive central coefficient, a nonzero term would force
the entire residual vector $w(\lambda+\rho)-(\mu+\rho)$ to vanish, identifying
two distinct strictly dominant weights in one Weyl orbit.  Hence
\cref{lem:simply-laced-parabolic-bound} again excludes every Weyl element
outside $W_J$.  If $r_i=\lambda_i+1$, a term indexed by
$T\subseteq J$ has excess $2-\sum_{i\in T}r_i$; hence a nonidentity term can
occur only when $\sum_{i\in T}r_i\le2$.  For reference, the ten unordered
pairs of the four excess-one roots have simple-root coordinates
\[
\begin{array}{c|c@{\qquad}c|c}
2h_{13}&(2,2,2,0)&h_{13}+h_{14}&(2,2,1,1)\\
h_{13}+h_{34}&(1,2,2,1)&h_{13}+\theta&(2,3,2,1)\\
2h_{14}&(2,2,0,2)&h_{14}+h_{34}&(1,2,1,2)\\
h_{14}+\theta&(2,3,1,2)&2h_{34}&(0,2,2,2)\\
h_{34}+\theta&(1,3,2,2)&2\theta&(2,4,2,2).
\end{array}
\]
The preceding list can be counted explicitly.  Let
\[
 H=\{h_{13},h_{14},h_{34},\theta\}.
\]
For a two-element multiset $P$ of elements of $H$, let
$d(P)=(d_1,d_3,d_4)$ be the three outer coefficients of the sum of the two
roots.  The ten vectors $d(P)$, read from the displayed table, are
\[
\begin{gathered}
 (2,2,0),(2,1,1),(1,2,1),(2,2,1),(2,0,2),\\
 (1,1,2),(2,1,2),(0,2,2),(1,2,2),(2,2,2).
\end{gathered}
\]
Since
\[
 \delta=(1+k_1)\alpha_1+(1+K)\alpha_2
       +(1+k_3)\alpha_3+(1+k_4)\alpha_4,
\]
the remainder $\delta-P$ is a nonnegative combination of
$\gamma_1,\gamma_3,\gamma_4$ exactly when
\[
 d_i(P)\le1+k_i\qquad(i\in J).
\]
Indeed, $\delta$ and $P$ both have excess $2$, so the central coefficient of
$\delta-P$ equals the sum of its outer coefficients.  Thus the displayed
inequalities are sufficient as well as necessary.
When this holds, that remainder has the unique expression
$n_1\gamma_1+n_3\gamma_3+n_4\gamma_4$, because its outer coefficients are
$n_1,n_3,n_4$.  Put
\[
 s=|\{i\in J:k_i>0\}|.
\]
Then the identity term is
\[
 p_{\rm ns}(\delta)=1+N_s,
 \qquad
 N_0=0,\quad N_1=1,\quad N_2=4,\quad N_3=10.
\]
The initial $1$ is the canonical partition
$\sigma+k_1\gamma_1+k_3\gamma_3+k_4\gamma_4$.  These four values are read
from the ten demand vectors without suppressing a case distinction.  If
$s=0$, every bound is $d_i(P)\le1$, and no demand vector qualifies.  If
$s=1$, after permuting the outer indices the bounds are $(2,1,1)$, and the
unique admissible vector is $(2,1,1)$.  If $s=2$, the bounds may be taken as
$(2,2,1)$, and the admissible vectors are
\[
 (2,2,0),\quad(2,1,1),\quad(1,2,1),\quad(2,2,1),
\]
so $N_2=4$.  If $s=3$, all three bounds are at least $2$, so all ten demand
vectors are admissible.

We now bound the nonidentity Weyl contribution without another case table.
For $i\in J$ put
\[
 r_i=\lambda_i+1=\mu_i+2k_i-K+2\ge1,
\quad
 J_1=\{i:r_i=1\},\quad J_2=\{i:r_i=2\}.
\]
For $T\subseteq J$ the argument of $p_{\rm ns}$ has excess
$2-\sum_{i\in T}r_i$, so it vanishes when
$\sum_{i\in T}r_i>2$.  The remaining partition numbers are explicit:
\begin{align*}
 p_{\rm ns}(\delta-\alpha_i)
   &=\begin{cases}1,&k_i=0,\\4,&k_i>0,\end{cases}
      &&(i\in J_1),\\
 p_{\rm ns}(\delta-2\alpha_i)
   &=\begin{cases}0,&k_i=0,\\1,&k_i>0,\end{cases}
      &&(i\in J_2),\\
 p_{\rm ns}(\delta-\alpha_i-\alpha_j)&=1
      &&(i,j\in J_1,\ i\ne j).
\end{align*}
Indeed, in the first line a partition contains exactly one root from $H$;
if $k_i=0$, only the member of $H$ not involving the $i$-th outer root is
possible, whereas if $k_i>0$ all four members of $H$ are possible.  In the
last two lines the excess is zero, so only the three $\gamma$-roots occur and
the outer coefficients determine their multiplicities uniquely.

Let
\[
 a_0=|\{i\in J_1:k_i=0\}|,
 \qquad a_1=|\{i\in J_1:k_i>0\}|,
 \qquad b_1=|\{i\in J_2:k_i>0\}|.
\]
The amount subtracted from the identity term after the positive two-reflection
terms are restored is therefore
\begin{equation}\label{eq:D4-Weyl-correction-count}
 C=a_0+4a_1+b_1-\binom{|J_1|}{2}.
\end{equation}
The relation $r_i=\mu_i+2k_i-K+2$ now gives the required bounds.  If $s=0$,
then $r_i=\mu_i+2\ge2$ and $b_1=0$, so $C=0$.  If $s=1$, the unique index
with $k_i>0$ has $r_i\ge k_i+2\ge3$; hence $a_1=b_1=0$ and
$C=|J_1|-\binom{|J_1|}{2}\le1$.  If $s=2$, two positive indices cannot both
belong to $J_1$, since $r_i=r_j=1$ would imply simultaneously
$k_j\ge k_i+1$ and $k_i\ge k_j+1$.  If one positive index belongs to $J_1$,
the other cannot belong to $J_2$, and \eqref{eq:D4-Weyl-correction-count}
gives $C\le4$; if no positive index belongs to $J_1$, then $b_1\le2$ and
$a_0\le1$, so $C\le3$.  Finally, if $s=3$, then $a_0=0$ and
\[
 C=4|J_1|+b_1-\binom{|J_1|}{2}\le9,
\]
with equality only possible for $|J_1|=3$; for $|J_1|=2,1,0$ the respective
upper bounds are $8,6,3$.  Consequently
\[
 B_\lambda(\mu)=1+N_s-C
 \ge\begin{cases}
 1,&s=0,1,2,\\
 2,&s=3,
 \end{cases}
\]
and in particular $B_\lambda(\mu)\ge1$.
\end{proof}

\begin{theorem}
\label{thm:D4-necessary}
Let $\mu\le\lambda$ be dominant weights in type $\typeD_4$, and write
\[
 \delta=\lambda-\mu
   =d_1\alpha_1+d_2\alpha_2+d_3\alpha_3+d_4\alpha_4.
\]
If $a(\mu,\lambda)<0$, then
\begin{enumerate}[label=\textnormal{(\roman*)}]
\item \label{item:D4-central}
$\mu_2=\langle\mu,\alpha_2^\vee\rangle=1$;
\item \label{item:D4-depth}
for some $i\in J$, with $\{i,j,k\}=J$,
\begin{equation}\label{eq:D4-depth}
 \delta\ge
 \beta_i:=\alpha_i+2\alpha_2+2\alpha_j+2\alpha_k;
\end{equation}
\item \label{item:D4-balance}
\begin{equation}\label{eq:D4-balance}
             d_1+d_3+d_4\ge d_2+3.
\end{equation}
\end{enumerate}
\end{theorem}

\begin{proof}
\Cref{prop:D4-defect} gives
$a(\mu,\lambda)\ge0$ when $\mu_2=0$ or $\mu_2\ge2$.  Therefore
$a(\mu,\lambda)<0$ forces $\mu_2=1$, proving
\ref{item:D4-central}.
Assume henceforth that $\mu_2=1$ and put $\tau=\mu+\sigma$.
By \eqref{eq:D4-defect-atomic} and \cref{lem:D4-raw-positive}, negativity
forces $a(\tau,\lambda)>0$, so $\tau\le\lambda$ and hence every $d_i\ge1$.
Write $\varepsilon=\lambda-\tau=(d_1-1,d_2-1,d_3-1,d_4-1)$ in simple-root
coordinates.

Suppose \ref{item:D4-depth} fails.  Then either $d_2=1$, or $d_2\ge2$ and at
most one outer coefficient $d_i$ is at least $2$.  In either case every
connected component of $\supp(\varepsilon)$ has rank at most two.  Since the
ambient system is simply laced, these components have type $\typeA_1$ or $\typeA_2$.
\Cref{thm:support-reduction} therefore writes $a(\tau,\lambda)$ as a product
of atomic numbers of these components.  In type $\typeA_1$ every irreducible
character is its own girdle, so the corresponding factor is
$\delta_{\tau_{I_r}\lambda_{I_r}}$; a component of $\supp(\varepsilon)$
carries a nonzero coefficient of $\varepsilon=\lambda-\tau$, so such a factor
is $0$.  An $\typeA_2$ factor is given by \eqref{eq:A2-atomic-formula} and
lies in $\{0,1\}$.  Hence
$a(\tau,\lambda)\in\{0,1\}$, and positivity forces the value $1$.  Then
$\supp(\varepsilon)$ has no $\typeA_1$ component, and any $\typeA_2$
component of the $\typeD_4$ diagram contains the central vertex $2$, so
$\supp(\varepsilon)$ is empty or is a single component $\{2,i\}$ with
$i\in J$; \eqref{eq:A2-atomic-formula} then forces equal coefficients.  Thus
\[
       \varepsilon=0
       \quad\text{or}\quad
       \varepsilon=q(\alpha_2+\alpha_i)
       \quad(q\ge1)
\]
for some $i\in J$.  By \cref{lem:D4-boundary-estimates}(a), however,
$B_\lambda(\mu)\ge1$ in all these cases.  Equation
\eqref{eq:D4-defect-atomic} then gives $a(\mu,\lambda)\ge0$, a contradiction.
This proves \ref{item:D4-depth}.

Finally set
\[
 E=(d_1-1)+(d_3-1)+(d_4-1)-(d_2-1)
   =d_1+d_3+d_4-d_2-2.
\]
If $E<0$, \cref{lem:D4-boundary-estimates}(b) gives
$a(\tau,\lambda)=0$; if $E=0$, it gives
$a(\tau,\lambda)=1$ and $B_\lambda(\mu)\ge1$.  Neither case permits
$B_\lambda(\mu)-a(\tau,\lambda)<0$.  Thus negativity forces $E\ge1$, which
is precisely \eqref{eq:D4-balance}.
\end{proof}

\begin{proposition}\label{prop:D4-minimal-family}
Let $i\in J$, let $\{i,j,k\}=J$, and put
\[
 \beta_i=\alpha_i+2\alpha_2+2\alpha_j+2\alpha_k.
\]
For every dominant weight $\mu$ with $\mu_2=1$ such that
$\lambda=\mu+\beta_i$ is dominant,
\begin{equation}\label{eq:D4-minimal-family}
 a(\mu,\lambda)=
 \begin{cases}
 -2,&\mu_i=0,\\
 -1,&\mu_i>0.
 \end{cases}
\end{equation}
\end{proposition}

\begin{proof}
By triality and \Cref{prop:atomic-dynkin} it is enough to take $i=1$.  Put $\tau=\mu+\sigma$.  Then
\[
 \lambda-\tau=\alpha_2+\alpha_3+\alpha_4,
\]
which is the highest root of the regular $\typeA_3$ subsystem on
$\{2,3,4\}$.  The weight $\tau+\alpha_2$ is dominant and lies below
$\lambda$, since
$\lambda-(\tau+\alpha_2)=\alpha_3+\alpha_4$; hence
\cref{prop:D4-defect}(ii), applied with lower weight $\tau$, shows that it is
the unique atom of $[\tau,\lambda]$.

Both differences $\lambda-\tau=\alpha_2+\alpha_3+\alpha_4$ and
$\lambda-(\tau+\alpha_2)=\alpha_3+\alpha_4$ have support contained in
$\{2,3,4\}$.  Therefore the
standard multiplicity restriction to the support
\cite[Proposition~2.4(1)]{BZ1}, which is also the multiplicity input in the
proof of \cref{thm:support-reduction}, identifies the ambient $\typeD_4$
multiplicities of $\tau$ and $\tau+\alpha_2$ with the corresponding weight
multiplicities in the regular $\typeA_3$ subsystem.  We may consequently compute
both multiplicities entirely in that subsystem.  Its positive roots are
\[
 \alpha_2,\alpha_3,\alpha_4,
 \alpha_2+\alpha_3,\alpha_2+\alpha_4,
 \alpha_2+\alpha_3+\alpha_4.
\]
The difference $\lambda-\tau=\alpha_2+\alpha_3+\alpha_4$ has exactly four
Kostant partitions, written here as multisets of positive roots:
\[
 \{\alpha_2+\alpha_3+\alpha_4\},\qquad
 \{\alpha_2+\alpha_3,\alpha_4\},\qquad
 \{\alpha_2+\alpha_4,\alpha_3\},\qquad
 \{\alpha_2,\alpha_3,\alpha_4\}.
\]
For the restricted highest weight, $\lambda_2=0$ and
$\lambda_3,\lambda_4\ge2$.  Thus the $s_2$ term in Kostant's formula subtracts
$(\lambda_2+1)\alpha_2=\alpha_2$ and leaves
$\alpha_3+\alpha_4$, which has the unique partition
$\{\alpha_3,\alpha_4\}$.  If a Weyl element of this $\typeA_3$ subsystem involves $s_3$,
\cref{lem:simply-laced-parabolic-bound}, applied to the simply-laced subsystem,
gives an $\alpha_3$-subtraction of at least
$\lambda_3+1\ge3$, larger than the coefficient $1$ of $\alpha_3$ in
$\lambda-\tau$; hence that Kostant argument has negative $\alpha_3$
coordinate.  If a Weyl element involves $s_4$,
\cref{lem:simply-laced-parabolic-bound} gives an
$\alpha_4$-subtraction of at least $\lambda_4+1\ge3>1$, so its Kostant
argument has negative $\alpha_4$ coordinate.  The only remaining nonidentity
element is therefore $s_2$.  Consequently
\[
 m(\tau,\lambda)=4-1=3.
\]
For the weight $\tau+\alpha_2$ the identity argument is
$\alpha_3+\alpha_4$, with the single partition
$\{\alpha_3,\alpha_4\}$.  The $s_2$ term has negative $\alpha_2$ coordinate,
and any term involving $s_3$ or $s_4$ has a negative corresponding outer
coordinate by the same coefficient bound.  Therefore
\[
 m(\tau+\alpha_2,\lambda)=1.
\]
Thus
\begin{equation}\label{eq:D4-minimal-shifted}
             a(\tau,\lambda)=2.
\end{equation}

It remains to compute the raw Boolean term.  Here
$\lambda-\mu=\beta_1=(1,2,2,2)$ in simple-root coordinates.
By \cref{lem:D4-pns}, the identity Weyl term has the unique nonsimple-root
partition
\[
       \beta_1=h_{34}+\sigma.
\]
Indeed, its central coefficient is $2$, so a partition consists either of
$\theta$ alone or of two roots with central coefficient $1$.  The former
does not equal $\beta_1$.  In the latter case both roots must contain
$\alpha_3$ and $\alpha_4$, and exactly one must contain $\alpha_1$; the list
\eqref{eq:D4-nonsimple-roots} therefore forces $h_{34}$ and $\sigma$.
To determine the nonidentity terms, put $L=\lambda+\rho$.  Since the simple-root coefficients of $\beta_1$ at $\alpha_3$ and
$\alpha_4$ are both $2$, while $L_3=\mu_3+3>2$ and $L_4=\mu_4+3>2$,
\cref{lem:simply-laced-parabolic-bound} forces every contributing $w$ to lie in
the rank-two parabolic $\langle s_1,s_2\rangle$.  If $\mu_1>0$, then also
$L_1=\mu_1+1>1$, the $\alpha_1$-coefficient of $\beta_1$, so only $1$ and
$s_2$ remain; the $s_2$ argument is
$(1,1,2,2)$ and has no nonsimple-root partition, since central coefficient
$1$ allows only one nonsimple root, whose outer coefficients are at most $1$.
If $\mu_1=0$, the six
elements of $\langle s_1,s_2\rangle\cong W(\typeA_2)$ give, respectively,
first and central simple-root coefficients
\[
 (1,2),\ (0,2),\ (1,1),\ (-1,1),\ (0,0),\ (-1,0)
\]
after subtraction from $\beta_1$ (in the order
$1,s_1,s_2,s_1s_2,s_2s_1,s_1s_2s_1$).  Only the first two can be partitioned
by nonsimple roots: each remaining argument has a negative coefficient or
has central coefficient at most $1$ while its $\alpha_3$ and $\alpha_4$
coefficients remain $2$.  The $s_1$ term has
\[
       \beta_1-\alpha_1=2h_{34}
\]
as its unique partition, since both roots must contain $\alpha_3$ and
$\alpha_4$ and neither can contain $\alpha_1$.  Hence
\[
 B_\lambda(\mu)=
 \begin{cases}
 0,&\mu_1=0,\\
 1,&\mu_1>0.
 \end{cases}
\]
Combining this with \eqref{eq:D4-defect-atomic} and
\eqref{eq:D4-minimal-shifted} gives \eqref{eq:D4-minimal-family}.
\end{proof}

\begin{proposition}
\label{prop:DE-propagation}
Every irreducible root system of type $\typeD_n$ $(n\ge4)$,
$\typeE_6$, $\typeE_7$, or $\typeE_8$ contains infinitely many pairs of
dominant weights $\mu\le\lambda$ with negative atomic number.  More precisely,
if $I=\supp(\lambda-\mu)$ is connected and of type $\typeD_4$ and, after identifying
it with the Bourbaki-numbered $\typeD_4$, the restricted pair is $((0100),(0022))$, then
\[
      a(\mu,\lambda)=-2.
\]
Likewise, the restricted pair $((1100),(1022))$ gives atomic number $-1$.
\end{proposition}

\begin{proof}
For type $\typeD_4$ the infinitude follows already from
\cref{prop:D4-minimal-family}.  In every other listed type, choose the
$\typeD_4$ subdiagram $I$ formed by a trivalent vertex and its three neighbors,
and fix one of the displayed $\typeD_4$ pairs on $I$.  Regard its root difference
as the same nonnegative combination of the ambient simple roots in $I$.
Write the chosen $\typeD_4$ difference as
\[
 \delta=\sum_{i\in I}d_i\alpha_i,
 \qquad d_i\ge0.
\]
Use the prescribed Dynkin labels for $\mu$ on $I$.  For $j\notin I$, choose
\begin{equation}\label{eq:DE-extension-bound}
 \mu_j\ge -\langle\delta,\alpha_j^\vee\rangle.
\end{equation}
The right-hand side is $0$ unless $j$ is adjacent to $I$; since the ambient
$\typeD/\typeE$ diagram is simply laced, if $j$ is adjacent to the unique
$i\in I$ then it equals $d_i$.  With
$\lambda=\mu+\delta$ we have, for every $j\notin I$,
\[
 \lambda_j=\mu_j+\langle\delta,\alpha_j^\vee\rangle\ge0,
\]
while for $i\in I$ the labels of $\lambda$ are exactly those of the prescribed
$\typeD_4$ upper weight because $\delta$ has no coefficient outside $I$.
Thus both $\mu$ and $\lambda$ are dominant and
$\supp(\lambda-\mu)=I$.  \Cref{thm:support-reduction} identifies the ambient
atomic number with the
displayed $\typeD_4$ value.  Increasing any one outside label while preserving
\eqref{eq:DE-extension-bound} produces infinitely many distinct ambient
pairs.
\end{proof}

\subsection{Propagation to types $\typeD$ and $\typeE$, and locality}
\Cref{thm:path-positive} yields the following obstruction.

\begin{theorem}\label{thm:DE-obstruction}
Let $\mu\le\lambda$ be dominant weights.  If
\[
       a(\mu,\lambda)<0,
\]
then some connected component of $\supp(\lambda-\mu)$ has Dynkin type
$\typeD$ or $\typeE$.  Equivalently, the subdiagram spanned by
$\supp(\lambda-\mu)$ has a trivalent vertex.
\end{theorem}

\begin{proof}
By \cref{thm:support-reduction}, the atomic number factors over the connected
components of $\supp(\lambda-\mu)$.  If every component were a path, each
factor would be nonnegative by \cref{thm:path-positive}.  Thus a negative
factor must come from a connected finite Dynkin diagram that is not a path,
and the classification of finite root systems leaves precisely the types
$\typeD$ and $\typeE$.
\end{proof}

Conversely, every irreducible finite type with a fork fails universal atomic
positivity by \cref{prop:DE-propagation}.  Consequently we obtain the following
classification.

\begin{corollary}
\label{cor:universal-path-classification}
For an irreducible finite root system, all atomic numbers are nonnegative if
and only if its Dynkin diagram is a path.  Thus the universally
positive irreducible types are
\[
       \typeA_n,\quad \typeB_n,\quad \typeC_n,\quad \typeF_4,
       \quad \typeG_2,
\]
whereas universal positivity fails in
\[
       \typeD_n\ (n\ge4),\qquad \typeE_6,\quad \typeE_7,\quad \typeE_8.
\]
\end{corollary}

There is a sharper local expectation: the Dynkin label at a trivalent root
should equal $1$, while $\lambda-\mu$ should be ``high enough.''  In type
$\typeD_4$ this local statement is no longer conjectural.  \Cref{prop:D4-defect}
and \cref{thm:D4-necessary} prove that negativity forces the
trivalent label to be $1$, forces domination of one of the three minimal
vectors $\beta_i$, and also forces the balance inequality
\eqref{eq:D4-balance}.  Determining an effective local criterion that is also
sufficient remains open.

The relevance of $\typeD_4$ is structural rather than tied to the exceptional
M\"obius value $\pm2$.  For the negative pair
$(\mu,\lambda)=((1100),(1022))$, the three atoms are
\[
 (3000)=\mu+\alpha_1,\qquad
 (1020)=\mu+\alpha_3,\qquad
 (1002)=\mu+\alpha_4.
\]
These three dominant simple-root increments cover $\mu$.  The remaining
dominant covering is $\widehat{\mu+\alpha_2}=(0111)$, which lies above both
$(1020)$ and $(1002)$ and hence is not an atom.  Thus
\cref{lem:atoms-simple-roots} excludes any further atom.
Their joins satisfy
\begin{align*}
 (1020)\vee(1002)&=(0111),\\
 (3000)\vee(1020)&=(3000)\vee(1002)=(2011),\\
 (3000)\vee(1020)\vee(1002)&=(2011).
\end{align*}
Indeed, the raw sum for the pair $\{3,4\}$ has central label $-1$, and its
reflection at $2$ gives $(0111)$.  For the pair $\{1,3\}$, successive
reflections at labels $-1$ at vertices $2$ and $4$ give $(2011)$; the pair
$\{1,4\}$ is symmetric.  These are the least dominant majorants by
\cref{lem:path-unit-stabilization}.  The triple has least dominant majorant
$\mu+\sigma=(2011)$ by the calculation in \cref{prop:D4-defect}(iii).
Hence the crown consists of
\[
 (1100),\quad(3000),\quad(1020),\quad(1002),\quad(0111),\quad(2011).
\]
\Cref{prop:crosscut-lower} gives the M\"obius values, in this order,
\[
 1,\quad -1,\quad -1,\quad -1,\quad 1,\quad 1.
\]
Thus negativity already occurs when every nonzero crown M\"obius value is
$\pm1$.  What remains open is whether the obstruction is always detectable on
the smallest trivalent neighborhood.

\begin{question}\label{q:local-criterion}
For a negative atomic number in type $\typeD_n$ or $\typeE_n$, is the
obstruction always detected after restriction to the radius-one $\typeD_4$
neighborhood of a trivalent vertex, or can longer arms change the effective
boundary condition in an essential way?
\end{question}

The distinction is important.  \Cref{thm:DE-obstruction} proves that a
trivalent component is necessary, and \cref{thm:support-reduction} proves that
a pair
supported on a $\typeD_4$ subdiagram is genuinely a $\typeD_4$ problem.  Neither result
shows that every negative pair in a larger $\typeD/\typeE$ diagram admits a
radius-one reduction.

\section{The stable chamber}\label{sec:stabilization}

Deep inside the dominant chamber the wall corrections disappear.  The inverse
relation $a*\kappa=\delta$ becomes translation invariant, with fundamental
solution the nonsimple-root partition function $p_{\rm ns}$; atomic numbers
therefore become independent of $\mu$.  Lecouvey and Lenart observed the
corresponding limit statement \cite[(17)]{MR4178925}, together with the
character limits \cite[Proposition~2.7]{MR4178925}, and proved a $t$-analogue
on the crystal $B(\infty)$ \cite[Corollary~3.4]{MR4178925}.  Here we give an
effective finite range on which the atomic number itself equals that partition
number.  Consequently \cref{cor:negativity-is-shallow} confines every negative
atomic number to boundary slabs.

Throughout this section $p$ denotes Kostant's partition function for
$\Raiz^+$ and, as in \cref{sec:D4}, $p_{\rm ns}$ denotes the partition
function using only the nonsimple positive roots $\Raiz^{\rm ns}$; both are
extended by zero outside $\Z_{\ge0}\RaizSimp$.  For $S\subseteq I$ we write
$\alpha_S=\sum_{i\in S}\alpha_i$.

\begin{lemma}\label{lem:pns-general}
For every $\eta$ in the root lattice,
\begin{equation}\label{eq:pns-general}
   p_{\rm ns}(\eta)=\sum_{S\subseteq I}(-1)^{|S|}\,p(\eta-\alpha_S).
\end{equation}
\end{lemma}

\begin{proof}
Multiplying $\sum_\eta p(\eta)e^\eta=\prod_{\alpha\in\Raiz^+}(1-e^\alpha)^{-1}$
by $\prod_{i\in I}(1-e^{\alpha_i})$ cancels exactly the simple-root factors
and leaves $\prod_{\alpha\in\Raiz^{\rm ns}}(1-e^{\alpha})^{-1}$.  This is
\cref{lem:G2-partition-difference} in type $\typeG_2$ and the inner
computation of \cref{lem:D4-pns} in type $\typeD_4$.
\end{proof}

\begin{lemma}\label{lem:kostant-stabilization}
Let $\lambda\in\Pesos^+$ and let $\eta\le\lambda$ with
$\lambda-\eta=\sum_{j\in I}p_j\alpha_j$.  If
\[
   \langle\lambda,\alpha_i^\vee\rangle\ \ge\ \max_{j\in I}p_j
   \qquad\text{for every }i\in I,
\]
then $m(\eta,\lambda)=p(\lambda-\eta)$.
\end{lemma}

\begin{proof}
By Kostant's formula \cite{MR0109192},
$m(\eta,\lambda)=\sum_{w\in\Weyl}(-1)^{\ell(w)}p\bigl(w(\lambda+\rho)-(\eta+\rho)\bigr)$.
Fix $w\ne1$ and choose $i$ with $w\alpha_i\in\Raiz^-$; put $u=ws_i$ and
$\beta=u\alpha_i=-w\alpha_i\in\Raiz^+$.  From
$s_i(\lambda+\rho)=(\lambda+\rho)-\langle\lambda+\rho,\alpha_i^\vee\rangle\alpha_i$
we get
\[
   (\lambda+\rho)-w(\lambda+\rho)
     =\bigl[(\lambda+\rho)-u(\lambda+\rho)\bigr]
      +\langle\lambda+\rho,\alpha_i^\vee\rangle\,\beta .
\]
The bracket lies in $\Z_{\ge0}\RaizSimp$ because $\lambda+\rho$ is dominant.
Write $\beta=\sum_j b_j\alpha_j$ with $b_j\ge0$ and choose $j_0$ with
$b_{j_0}\ge1$.  Then the coefficient of $\alpha_{j_0}$ in
$(\lambda+\rho)-w(\lambda+\rho)$ is at least
$\langle\lambda,\alpha_i^\vee\rangle+1$.  Since
\[
   w(\lambda+\rho)-(\eta+\rho)
     =(\lambda-\eta)-\bigl[(\lambda+\rho)-w(\lambda+\rho)\bigr],
\]
its coefficient of $\alpha_{j_0}$ is at most
$p_{j_0}-\langle\lambda,\alpha_i^\vee\rangle-1\le-1$, so the corresponding
value of $p$ vanishes.  Only $w=1$ survives.
\end{proof}

For $i\in I$ set
\begin{equation}\label{eq:b-constant}
   b_i=-\sum_{j\ne i}\langle\alpha_j,\alpha_i^\vee\rangle
      =\sum_{j\ne i}\bigl|\langle\alpha_j,\alpha_i^\vee\rangle\bigr| .
\end{equation}
Inspecting the finite Dynkin diagrams gives $b_i\le3$ in every type: a node
with three neighbours occurs only in types $\typeD$ and $\typeE$, where all
bonds are simple; a node with a double or triple bond has at most one further
neighbour, joined by a simple bond.

\begin{theorem}\label{thm:atomic-stabilization}
Let $\mu\le\lambda$ be dominant and write
$\lambda-\mu=\sum_{j\in I}n_j\alpha_j$.  Assume
\begin{align}
  \langle\mu,\alpha_i^\vee\rangle&\ \ge\ b_i
     &&\text{for every }i\in I,\label{eq:stab-hyp-1}\\
  \langle\lambda,\alpha_i^\vee\rangle&\ \ge\ \max_{j\in I}n_j
     &&\text{for every }i\in I.\label{eq:stab-hyp-2}
\end{align}
Then
\begin{equation}\label{eq:atomic-stable-value}
   a(\mu,\lambda)=p_{\rm ns}(\lambda-\mu).
\end{equation}
\end{theorem}

\begin{proof}
By \eqref{eq:stab-hyp-1}, for every $S\subseteq I$ and every $i\in I$,
\[
   \langle\mu+\alpha_S,\alpha_i^\vee\rangle
     =\langle\mu,\alpha_i^\vee\rangle
       +\sum_{j\in S}\langle\alpha_j,\alpha_i^\vee\rangle
     \ge\langle\mu,\alpha_i^\vee\rangle-b_i\ge0,
\]
the term $j=i$ contributing $+2$; hence $\mu+\alpha_S$ is dominant, and in
particular $\widehat{\mu+\alpha_S}=\mu+\alpha_S$.  By
\cref{lem:atoms-simple-roots} every atom of $[\mu,\lambda]$ has the form
$\widehat{\mu+\alpha_i}=\mu+\alpha_i$; conversely $\mu+\alpha_i$ covers
$\mu$, since no weight lies strictly between them in the root order, and
$\mu+\alpha_i\le\lambda$ if and only if $n_i\ge1$.  Hence
\[
   \mathcal A_\mu(\lambda)=\{\mu+\alpha_i:\ i\in\supp(\lambda-\mu)\}.
\]
For $S\subseteq\supp(\lambda-\mu)$ the coordinatewise join of the weights
$\mu+\alpha_i$, $i\in S$, is $\mu+\alpha_S$, so repeated use of
\cref{lem:dominant-covering-join} gives
$\bigvee_{i\in S}(\mu+\alpha_i)=\widehat{\mu+\alpha_S}=\mu+\alpha_S$.
Formula \eqref{eq:atomic-finite-difference} therefore
reads
\[
   a(\mu,\lambda)
     =\sum_{S\subseteq\supp(\lambda-\mu)}(-1)^{|S|}
        m(\mu+\alpha_S,\lambda).
\]
For each such $S$ the weight $\lambda-\mu-\alpha_S$ has $\alpha_j$-coefficient
at most $n_j$, so \eqref{eq:stab-hyp-2} and
\cref{lem:kostant-stabilization} give
$m(\mu+\alpha_S,\lambda)=p(\lambda-\mu-\alpha_S)$.  If
$S\not\subseteq\supp(\lambda-\mu)$ then $\lambda-\mu-\alpha_S$ has a negative
coefficient and $p$ vanishes on it, so the sum may be extended over all
$S\subseteq I$ without change.  Now apply \cref{lem:pns-general}.
\end{proof}

\begin{remark}\label{rem:stabilization-second-proof}
\Cref{thm:atomic-stabilization} can also be deduced from
\cref{thm:atomic-recursion}, without Kostant's formula.  Fix
$\eta_0\in\Z_{\ge0}\RaizSimp$ and choose $\mu$ dominant enough that
$\mu+\eta-\xi$ is dominant for every $\eta\le\eta_0$ in
$\Z_{\ge0}\RaizSimp$ and every $\xi\in\mathcal F_0$; this is a finite set of
lower bounds on the labels of $\mu$.  For such $\mu$ all signs in
\eqref{eq:atomic-recursion} equal $1$ and all shifted arguments are dominant,
so with $f(\eta)=a(\mu,\mu+\eta)$, extended by zero off
$\Z_{\ge0}\RaizSimp$,
\[
   \sum_{\xi\in\mathcal F_0}c_\xi f(\eta-\xi)=\delta_{\eta,0}
   \qquad(\eta\le\eta_0).
\]
By \eqref{eq:ns-generating} the function $p_{\rm ns}$ satisfies the same
recursion with the same initial condition, so $f=p_{\rm ns}$ on
$\{\eta\le\eta_0\}$ by induction on the height.  This argument is shorter but
gives a less explicit range than \eqref{eq:stab-hyp-1}--\eqref{eq:stab-hyp-2}.
\end{remark}

\begin{corollary}
\label{cor:eventual-positivity}
For every $\eta\in\Z_{\ge0}\RaizSimp$ there is a constant $C(\eta)$, namely
\[
   C(\eta)=\max_{i\in I}
     \max\Bigl\{\,b_i,\ \max_{j\in I}n_j-\langle\eta,\alpha_i^\vee\rangle\Bigr\}
   \qquad\Bigl(\eta=\sum_j n_j\alpha_j\Bigr),
\]
such that every dominant $\mu$ with $\langle\mu,\alpha_i^\vee\rangle\ge C(\eta)$
for all $i$ satisfies
\[
   a(\mu,\mu+\eta)=p_{\rm ns}(\eta)\ \ge\ 0 .
\]
\end{corollary}

\begin{proof}
The two hypotheses of \cref{thm:atomic-stabilization} for the pair
$(\mu,\mu+\eta)$ read $\langle\mu,\alpha_i^\vee\rangle\ge b_i$ and
$\langle\mu,\alpha_i^\vee\rangle\ge\max_j n_j-\langle\eta,\alpha_i^\vee\rangle$.
\end{proof}

\begin{corollary}
\label{cor:negativity-is-shallow}
If $a(\mu,\lambda)<0$ then
$\langle\mu,\alpha_i^\vee\rangle<C(\lambda-\mu)$ for at least one $i\in I$.
Equivalently, for each fixed $\eta$, the set of dominant weights $\mu$ for
which $a(\mu,\mu+\eta)<0$ is contained in the finite union of boundary slabs
\[
   \bigcup_{i\in I}
   \{\mu\in\Pesos^+:\langle\mu,\alpha_i^\vee\rangle<C(\eta)\}.
\]
\end{corollary}

\begin{remark}
\label{rem:D4-stable}
Take $\eta=\beta_1=\alpha_1+2\alpha_2+2\alpha_3+2\alpha_4$, the vector of
\cref{prop:D4-minimal-family}.  In the notation
\eqref{eq:D4-nonsimple-roots} the only decomposition of $\beta_1$ into
nonsimple positive roots is $\beta_1=\sigma+h_{34}$, so
$p_{\rm ns}(\beta_1)=1$.  Accordingly, for $\mu=3\rho=(3333)$ --- which
satisfies \eqref{eq:stab-hyp-1} and \eqref{eq:stab-hyp-2} exactly, since
$b_2=3$ and $\max_j n_j=2$ --- one has
$a\bigl(3\rho,\,3\rho+\beta_1\bigr)=1$.  By contrast
\cref{prop:D4-minimal-family} gives
\[
   a(0100,0022)=-2,
   \qquad
   a(5155,5077)=-1 .
\]
Thus negativity survives however deep $\mu$ is in the three outer directions,
provided the trivalent label stays equal to $1$; and it is exactly
\eqref{eq:stab-hyp-1} at the trivalent node, where $b_2=3$, that fails.  This
is the quantitative counterpart of \cref{thm:D4-necessary}.
\end{remark}

We finish with two questions suggested by the preceding arguments.

\begin{question}\label{q:path-intrinsic-model}
The successive model $R_{\Psi,\mathfrak o}$ of
\cref{cor:path-primitive-model} depends on the order of the path components
and on the endpoint from which each is processed.  Are the
subquotients obtained from opposite orientations naturally isomorphic, or are
they different realizations of a canonical representation-theoretic object
attached to $(\mu,\lambda)$?
\end{question}

\begin{question}\label{q:DE-stabilization-defects}
Let $\gamma_i=\mu+\beta_i$ ($1\le i\le k$) be the atoms of
$[\mu,\lambda]$ in type $\typeD$ or $\typeE$, and put
$x_S=\mu+\sum_{i\in S}\beta_i$ for $S\subseteq\{1,\ldots,k\}$.
Which local crown configurations satisfy
$\widehat{x_S}\in\Weyl x_S$ for every $S$, as on paths?
For those that do not, can one determine from the local configuration the
multiplicity corrections
\[
 m(\widehat{x_S},\lambda)-m(x_S,\lambda)
\]
and their signed sum?  By disjointness of atom supports and
\cref{lem:dominant-covering-join}, the crown joins are $\widehat{x_S}$, so
\cref{cor:atomic-finite-difference} gives
\[
 a(\mu,\lambda)-\sum_S(-1)^{|S|}m(x_S,\lambda)
   =\sum_S(-1)^{|S|}
      \bigl(m(\widehat{x_S},\lambda)-m(x_S,\lambda)\bigr).
\]
Failure of Weyl conjugacy alone need not make an individual correction
nonzero for a given $\lambda$, and nonzero corrections may cancel in this
signed sum.
\end{question}

\section*{Acknowledgment and AI disclosure}

This work originated in research conducted by the authors.  
WS wrote and executed the \LiE\ code used for exhaustive computational
searches, within the ranges investigated, in ranks up to $10$, producing
millions of data samples.  The authors analysed these data and formulated
the initial machinery, proved positivity in ranks 2 and 3, and made precise
conjectures on the negativity in types $\typeD/\typeE$. All the main ideas
in this paper are due and related to this effort. This is contained in
substantial part in the Ph.D. dissertation of FH.

In subsequent work, the authors used OpenAI's Codex and Anthropic's Claude
Code to assist in identifying
further patterns in the data, critically examining and refining mathematical
proofs, and preparing and revising the manuscript.  The authors take full
responsibility for the mathematical arguments, computations, references, and
final text.


\begin{thebibliography}{MPRA25}

\bibitem[BZ90]{BZ1}
A.~D. Berenshtein and A.~V. Zelevinskii.
\newblock When is the multiplicity of a weight equal to {$1$}?
\newblock {\em Funct. Anal. Appl.}, 24(4):259--269, 1990.

\bibitem[Kos59]{MR0109192}
Bertram Kostant.
\newblock A formula for the multiplicity of a weight.
\newblock {\em Trans. Amer. Math. Soc.}, 93(1):53--73, 1959.

\bibitem[LL21]{MR4178925}
C{\'e}dric Lecouvey and Cristian Lenart.
\newblock Atomic decomposition of characters and crystals.
\newblock {\em Adv. Math.}, 376:107453, 2021.

\bibitem[MPRA25]{Muniz-Plaza-Rojas-G2}
B{\'a}rbara Muniz, David Plaza, and Claudia Rojas-And{\'\i}as.
\newblock Atomic decomposition for an affine {Weyl} group of type {$G_2$},
  2025.
\newblock arXiv:2512.02559.

\bibitem[PS26]{Plaza-Sagurie}
David Plaza and Yamil Sagurie.
\newblock Positivity of pre-canonical bases for spherical {Hecke} algebras,
  2026.
\newblock arXiv:2608.07703.

\bibitem[PT25]{Patimo}
Leonardo Patimo and Jacinta Torres.
\newblock Atoms and charge in type {$C_2$}.
\newblock {\em Algebr. Comb.}, 8(2):521--574, 2025.

\bibitem[Sch12]{Schutzer}
Waldeck Sch{\"u}tzer.
\newblock A new character formula for {Lie} algebras and {Lie} groups.
\newblock {\em J. Lie Theory}, 22(3):817--838, 2012.

\bibitem[Shi01]{Shimozono}
Mark Shimozono.
\newblock Multi-atoms and monotonicity of generalized {Kostka} polynomials.
\newblock {\em European J. Combin.}, 22(3):395--414, 2001.

\bibitem[Sta70]{Stanley}
Richard~P. Stanley.
\newblock Structure of incidence algebras and their automorphism groups.
\newblock {\em Bull. Amer. Math. Soc.}, 76:1236--1239, 1970.

\bibitem[Sta12]{StanleyEC1}
Richard~P. Stanley.
\newblock {\em Enumerative Combinatorics. Volume 1}, volume~49 of {\em
  Cambridge Studies in Advanced Mathematics}.
\newblock Cambridge University Press, Cambridge, second edition, 2012.

\bibitem[Ste98]{Stembridge}
John~R. Stembridge.
\newblock The partial order of dominant weights.
\newblock {\em Adv. Math.}, 136(2):340--364, 1998.

\bibitem[vL94]{LiE}
M.~A.~A. van Leeuwen.
\newblock {LiE}, a software package for {Lie} group computations.
\newblock {\em Euromath Bull.}, 1(2):83--94, 1994.

\bibitem[Wal13]{Walton2}
Mark~A. Walton.
\newblock Polytope expansion of {Lie} characters and applications.
\newblock {\em J. Math. Phys.}, 54(12):121701, 2013.

\end{thebibliography}
\end{document}